\documentclass[reqno,11pt]{preprint}
\usepackage[marginparwidth=1in]{geometry}
\usepackage[full]{textcomp}
\usepackage[osf]{newtxtext}
\usepackage{cabin}
\usepackage{enumerate}
\usepackage{enumitem}
\usepackage{bbm}
\usepackage{dsfont}

\usepackage[varqu,varl]{inconsolata}
\usepackage[cal=boondoxo]{mathalfa}

\usepackage{comment}
\usepackage{hyperref}
\usepackage{breakurl}
\usepackage{mathrsfs}
\usepackage{booktabs}
\usepackage{caption}

\usepackage{tikz}
\usepackage{amsmath} 
\usepackage{mhequ} 
\usepackage{mhsymb} 
\usepackage{mhenvs} 
\usepackage{microtype}
\usepackage{amssymb} 
\usepackage{eufrak}

\usepackage{wasysym}
\usepackage{centernot}

\usepackage{shortcutsfonts}

\usepackage{oldgerm}

\newcommand{\Exp}{\mathbf{E}}

\renewcommand{\R}{\mathbb{R}}

\renewcommand{\N}{\mathbb{N}}

\renewcommand{\T}{\mathbb{T}}
\renewcommand{\Z}{\mathbb{Z}}
\renewcommand{\E}{\mathbb{E}}
\renewcommand{\P}{\mathbb{P}}

\renewcommand{\J}{\mathbb{J}}

\newcommand{\cA}{\mathcal{A}}
\newcommand{\cB}{\mathcal{B}}
\newcommand{\cC}{\mathcal{C}}
\newcommand{\cD}{\mathcal{D}}
\newcommand{\cF}{\mathcal{F}}
\newcommand{\cG}{\mathcal{G}}
\newcommand{\cH}{\mathcal{H}}

\newcommand{\cK}{\mathcal{K}}
\newcommand{\cL}{\mathcal{L}}
\newcommand{\cM}{\mathcal{M}}
\newcommand{\cN}{\mathcal{N}}
\newcommand{\cS}{\mathcal{S}}
\newcommand{\cV}{\mathcal{V}}

\newcommand{\cZ}{\mathcal{Z}}

\newcommand{\dd}{\mathrm{d}}      %%%%%% d

\newcommand{\Id}{\mathrm{Id}}

\newcommand{\SH}{\mathscr{H}}

\newcommand{\1}{\mathds{1}}

\newcommand{\vphi}{\varphi}
\def\sym{{\mathrm{sym}}}
\newcommand{\Di}{\mathrm{Diag}}
\newcommand{\oD}{\mathrm{off}}
\newcommand{\oDi}{\mathrm{off_1}}
\newcommand{\ooDi}{\mathrm{off_2}}
\newcommand{\bulk}{\mathrm{bulk}}

\colorlet{darkblue}{blue!90!black}
\colorlet{darkred}{red!90!black}
\colorlet{darkgreen}{green!50!black}
\colorlet{darkyellow}{yellow!90!black}
\def\one{\mathrm{(I)}}
\def\two{\mathrm{(II)}}

\newcommand{\half}{\frac{1}{2}}

\def\indN#1{{	\J^N_{#1}	}}
\newcommand{\fock}{\Gamma L^2}	% Fock space

\def\sint{I}					% Stochastic integral
\def\wc{\SH}					% Wiener Chaos
\def\swn{{	\eta	}}
\newcommand{\gen}{\cL^N}
\newcommand{\gens}{\cL_0^{N}}
\newcommand{\gensy}{\cL_0}
\newcommand{\gena}{\cA^{N}}
\newcommand{\genap}{\cA^{N}_+}
\newcommand{\genam}{\cA^{N}_-}

\newcommand{\nonlin}{\cK^{N}}

\newcommand{\Op}{\mathcal{H}^N}

\newcommand{\genapc}{A_+^N}
\newcommand{\nf}{\mathfrak{n}^N}
\newcommand{\hf}{\mathfrak{h}}

\def\energy{\CE^N}
\newcommand{\Ll}{\mathrm{L}}
\newcommand{\LB}{\mathrm{LB}}
\newcommand{\UB}{\mathrm{UB}}
\newcommand{\LlN}{\mathrm{L}^N}
\newcommand{\LBN}{\mathrm{LB}^N}
\newcommand{\UBN}{\mathrm{UB}^N}

\title{The stationary AKPZ equation:\\ logarithmic superdiffusivity}

\begin{document}

\maketitle

\vspace{-2cm}

\noindent{\large \bf Giuseppe Cannizzaro$^1$,  Dirk Erhard$^2$, Fabio Toninelli$^3$}
\newline

\noindent{\small $^1$University of Warwick, UK\\%
    $^2$Universidade Federal da Bahia, Brazil\\%
    $^3$Technical University of Vienna, Austria\\}

\noindent\email{giuseppe.cannizzaro@warwick.ac.uk, 
erharddirk@gmail.com, \\
fabio.toninelli@tuwien.ac.at}
\newline

\begin{abstract}
  We study the  two-dimensional Anisotropic KPZ equation
  (AKPZ) formally given by
  \begin{equ}
    \partial_t H=\frac12\Delta H+\lambda((\partial_1 H)^2-(\partial_2 H)^2)+\xi\,,
  \end{equ}
  where $\xi$ is a space-time white noise and $\lambda$ is a strictly positive constant. 
  While the classical
  two-dimensional KPZ equation, whose nonlinearity is
  $|\nabla H|^2=(\partial_1 H)^2+(\partial_2 H)^2$, can be linearised via the Cole-Hopf transformation, 
  this is not the case for AKPZ. We prove
  that the stationary solution to AKPZ (whose invariant measure is the Gaussian
  Free Field) is superdiffusive: its diffusion coefficient diverges
  for large times as $\sqrt{\log t}$ up to $\log\log t$ corrections, in a Tauberian sense. Morally, 
  this says that the correlation length grows with time like $t^{1/2}\times (\log t)^{1/4}$.
  Moreover, we show that if the process is
  rescaled diffusively ($t\to t/\eps^2, x\to x/\eps, \eps\to0$), then it evolves
  non-trivially already on time-scales of order approximately  $1/\sqrt{|\log\eps|}\ll1$.  Both claims hold as
  soon as the coefficient $\lambda$ of the nonlinearity is non-zero.
  % and the constant $\delta$ is uniformly bounded away from zero for
  % $\lambda$  small
%  . %  Based on the mode-coupling approximation (see e.g.~\cite{S14}), we conjecture that the 
  % optimal value is $\delta=1/2$.  
  These results are in contrast with the belief, common in the 
  mathematics community, that the AKPZ equation is diffusive at large scales and, 
  under simple diffusive scaling, converges the two-dimensional Stochastic Heat
  Equation (2dSHE) with additive noise (i.e. the case $\lambda=0$).
\end{abstract}


\setcounter{tocdepth}{2}       % Put subsubsections in the table of contents
\tableofcontents

\section{Introduction}
The KPZ equation is a stochastic PDE that formally is written as
\begin{equ}[eq:UBkpzformal]
  \partial_t H= \nu\Delta H+\langle \nabla H,Q\nabla H\rangle + \sqrt{D}\xi,
\end{equ}
where  $H=H(t,x)$ depends on time $t\geq 0$ and $x$, the spatial $d$-dimensional coordinate (e.g. $x\in\R^d$ or $\T^d$), 
$\xi$ is a space-time (white) noise, $\nu,D$ are two positive constants,
and $Q$ is a $d\times d$ matrix.  The KPZ equation was originally derived as a description for 
$(d+1)$-dimensional
stochastic growth: the Laplacian is a smoothing term that overall
flattens the interface, the noise models the microscopic local
randomness, while the non-linear term encodes
the slope-dependence of the growth mechanism. Indeed, at a heuristic
level, the connection between a specific (microscopic) growth model
and the KPZ equation is that $Q$ is proportional to the Hessian $D^2 v$ of the
average speed of growth $v$ of the microscopic model, seen as a function of the average interface slope.

The SPDE \eqref{eq:UBkpzformal} is well known to be analytically ill-posed if $\xi$
is a white noise, due to the non-linear term, so that in order to study the large-scale properties of its solution, 
a standard approach
is to focus on a regularised version of it obtained by smoothing either the noise or the nonlinearity (or both). 
%One
%natural goal is to understand the large-scale properties of the
%regularised equation. 
In the spirit of Renormalization Group, one would like to determine
whether the nonlinearity is \emph{relevant} or not, i.e. if it affects
the asymptotic behaviour in a qualitative way, in particular by
changing the growth and roughness exponents with respect to those of
the linear equation obtained by setting $Q\equiv 0$. Note that the
latter is just the $d$-dimensional Stochastic Heat Equation (SHE) with
additive noise. Already in the seminal paper \cite{Kardar} it was
predicted that, if $d\ge 3$ and the nonlinearity is small enough (say
if the norm of $Q$ is small), then the nonlinearity is irrelevant and
the scaling limit is given by the solution of SHE (up to a finite
renormalization of $\nu,D$). A recent series of works
(see~\cite{Magnen,Gu2018b,CCM2,LZ,CNN}) has confirmed this prediction
mathematically (with the important restriction that $Q$ is assumed to
be proportional to the identity matrix: only in this case one can
linearize \eqref{eq:UBkpzformal} via the Cole-Hopf transform, and map
it to a problem of directed polymers in random environment). As for
$d=1$,~\cite{Kardar} conjectures, and it is by now well established
(see~\cite{AKQ14, BQS11, CQ13, MQR18}), that the nonlinearity, no
matter its strength provided it is non-zero, is relevant and changes
the growth exponent from $\beta=1/4$ to $\beta=1/3$.  In dimension
$d=2$, the situation is subtler since finer details of the equation,
and in particular the structure of the matrix $Q$, might affect the
relevance claim.  Indeed, it was predicted
in~\cite{W91,barabasi1995fractal} that if $\det Q>0$ (Isotropic KPZ
equation) then the nonlinearity is relevant and gives rise to
non-trivial and model-independent growth and roughness exponents. In
view of the above-mentioned connection between $Q$ and the Hessian of
$v$, the condition $\det Q>0$ corresponds to growth models with
strictly convex or concave speed of growth.  In the complementary
case, $\det Q\le 0$ (Anisotropic KPZ or AKPZ equation), the
physicists' prediction, based on non-rigorous, one-loop
Renormalization Group computations
(see~\cite{W91,barabasi1995fractal}), states that the equation has the
same scaling limit as the 2dSHE.

A first clear indication that the isotropic and anisotropic versions
of the equation have a radically different behaviour is obtained by
looking at the equation where the nonlinearity parameter $Q$ is
scaled to zero together with $\eps$ (the noise regularisation
parameter).   
In the case of the isotropic KPZ equation with $Q=\lambda\Id$, $\lambda>0$, i.e. nonlinearity
$\lambda |\nabla H|^2$, it was found in \cite{CSZ} that, taking $\lambda=\hat \lambda/\sqrt{|\log\eps|}$, 
$H$ tends as $\eps\to0$ to the
solution of the linear equation with renormalised coefficients if
$\hat\lambda$ is smaller than a precisely identified threshold
$\hat\lambda_c$, and the noise strength in the limiting linear equation
diverges as $\hat \lambda\to\hat \lambda_c$. In contrast, for the
(stationary) AKPZ equation, the findings of~\cite{CES19} imply that there is no phase transition in this scaling.

\medskip

In the present work, we study the regularised AKPZ equation at stationarity with the specific choice
$Q=\lambda\, {\rm diag}\,(+1,-1)$ (in which case the stationary state is given by the Gaussian Free Field~\cite{CES19}), 
and \emph{we do not scale $\lambda$ down to zero}. As remarked in~\cite{da2003nonlinear}, this choice of 
$Q$ is the only one, modulo rotations, for which the stationary state is Gaussian.
Our main results state that in contrast with  the stochastic heat equation the AKPZ equation \emph{is not even asymptotically invariant under diffusive scaling}. In fact, while the former is scale
invariant under diffusive scaling, i.e. time  scaled as $t/\eps^2$
and space as $x/\eps$, we find that as soon as $\lambda>0$, the
stationary and diffusively rescaled process
$H^{\eps}(t,x)\eqdef H(t/\eps^2,x/\eps)$ evolves non-trivially already on
time-scales of order $|\log \eps|^{-1/2}\ll1$, up to  corrections polynomial in $\log|\log\eps|$.
By ``evolves non-trivially'' we mean for instance that, if
$\phi$ is a test function of zero total mass, the normalised covariance at different times of the locally averaged field 
$H^{\eps}[t](\phi)\eqdef\int \phi(x)H^{\eps}(t,x)\dd x$, 
\begin{equ}
  \label{normcov}
  \frac{{\rm Cov}(H^{\eps}[t](\phi),H^{\eps}[0](\phi))}{{\rm Var}(H^{\eps}[0](\phi))}, %\qquad H^{\eps}[t](\phi):=\int \phi(x)H^{\eps}(t,x)\dd x,
\end{equ}
is strictly smaller than $1$ uniformly in $\eps$, for
$t\approx |\log \eps|^{-1/2}$ (see Theorem \ref{thm:main2} and the
subsequent comments).  Moreover, we show that the diffusion
coefficient $D(t)$, which (once multiplied by $t$) measures the mean
square distance of spreading of correlations as a function of time,
grows in time as $\approx |\log t|^{1/2}$ for $t$ large as soon as
$\lambda>0$ (see Theorem \ref{thm:BD} for the precise formulation),
thus excluding diffusive behaviour since the linear equation instead
is known to diffuse at constant rate $D(t)=1$.  We emphasize that
logarithmic super-diffusivity for the AKPZ equation was \emph{not
  expected} in the mathematical literature \cite{BCF,BCT}, and we are
not aware of predictions in this sense even in the relevant physics
literature \cite{W91,barabasi1995fractal,Healy}.  Based on the
``mode-coupling'' heuristics we give in Appendix \ref{app:aheuri}, it
is reasonable to expect that, once the logarithmic corrections to the
scaling are taken into account, the large-scale behavior of the
equation is Gaussian. A first result in this direction was recently
obtained by the authors in \cite{WAKPZ}: in the case where the
strength $\lambda$ of the nonlinearity is suitably scaled to zero, the
AKPZ equation scales to the stochastic heat equation with renormalized
coefficients.

Finally, it is also interesting to look at more local quantities, such
as the time-dependence of the variance of the height increment at a
single point, $H(t,0)-H(0,0)$.  Since $H$ fails to be a function, we
will study the variance of the height tested against a fixed test
function of compact support.  According to the physicists' predictions
\cite{W91,barabasi1995fractal} and to numerical simulations
\cite{Healy}, this should grow asymptotically like $\log t$, as for
the linear equation. In Theorem~\ref{thm:logt} we prove an upper bound
of this order (implying that the growth exponent $\beta$ is zero); as
we explain in Remark \ref{rem:logt}, this is not in contradiction with
our finding of anomalous diffusivity.

To put our result into a wider context, let us mention that  $\sqrt{\log t}$-behaviour for the diffusion coefficient has been
conjectured also for a whole universality class of two-dimensional
(self)-interacting diffusions, including tracer particles in non-ideal
fluids \cite{Adler}, self-repelling random walks and Brownian polymers \cite{peliti1987random,amit1983asymptotic,obukhov1983renormalisation,Balint}
and the diffusion of a tracer particle in the curl of the
two-dimensional GFF \cite{Balint}. The best rigorous result we are aware of in this
context are super-diffusivity lower and upper bounds of order $\log\log t$ and 
 $\log t$ respectively, obtained in~\cite{LRY} for lattice gas models
and in \cite{Balint} for self-repelling polymers and for the diffusion in
the curl of the GFF. We believe that the tools developed in the present
paper (and in particular Theorem~\ref{thm:Main}) will help to
significantly improve the estimates for these models. 

The crucial ingredient of the proof is a control of the variance of
the time integral of the nonlinearity, that is obtained via an
iterative argument inspired by the works \cite{Landim2004,Yau}, where
the authors study the super-diffusivity of the asymmetric simple
exclusion process (ASEP) in dimensions $d=1,2$.  In particular,
\cite{Yau} proves $(\log t)^{2/3}$ super-diffusion for $2d$-ASEP. Let
us emphasize that, while the iterative method of \cite{Yau} gives a
logarithmic correction to diffusivity at any finite step $k$ of the
iteration, and the limit $k\to \infty$ is needed to pin down the power
of the logarithm to $2/3$, in our case at step $k$ we get only a
$|\log \log t|^k/k!$ correction and we need to take a $k$ diverging
with $t$ to get the $\sqrt{\log t}$ result.  This difference is not a
technical limitation of our estimates but rather it reflects a
different structure of the operators involved in the two problems. The
different symmetry properties of $2d$-ASEP and the AKPZ equation are
also responsible for the different exponents, $2/3$ versus $1/2$, in
the logarithmic super-diffusivity corrections; this was already
pointed out in \cite{LRY,Balint} in the context of lattice gases and
self-repelling polymers.

To conclude this introduction, let us recall that there are several
 microscopic $(2+1)$-dimensional growth models that are known to
belong to the AKPZ universality class, in the sense that their speed
of growth satisfies $\det(D^2 v)\le0$. These include the
Gates-Westcott model \cite{MR1478060,Ler}, certain two-dimensional
arrays of interlaced particle systems \cite{Borodin2014} and the
domino shuffling algorithm \cite{CT} just to mention a few (other
growth processes like the 6-vertex dynamics of
\cite{borodin2017irreversible} and the q-Whittaker particle system
\cite{BC14} should belong to this class, but an explicit computation of $v$ is
not possible since their stationary measures are non-determinantal;
see also \cite{MR3966870} for further references).  Typical results
that have been proven for such models are the scaling of stationary
fluctuations (at fixed time) to a Gaussian Free Field, a logarithmic
upper bound on height fluctuation growth
\cite{Toninelli2017,MR4033679,Ler} (similar to Theorem \ref{thm:logt}
below) and CLTs for height fluctuations on the scale $\sqrt{\log t}$
for certain non-stationary, ``integrable'' initial conditions
\cite{Borodin2014}. However, the more challenging issue of studying the
large-scale diffusivity (or super-diffusivity) properties of these
models is entirely unexplored. While logarithmic super-diffusivity effects
are quite hard to be observed numerically,  the
$(\log t)^{2/3}$ behaviour for two-dimensional asymmetric simple
exclusion has been very recently exhibited in simulations \cite{Krug}.  It would be extremely
interesting to study the super-diffusivity phenomenon we determine for
the continuum AKPZ equation also for discrete growth models in the
same universality class.

\subsection{The AKPZ equation and the main results}

In order to avoid integrability issues arising in the infinite volume regime (that are anyway addressed in~\cite{CK})
we study the solution $H_N$ of the regularised AKPZ equation on a large torus of size $N\in\N$, which is given by
\begin{align}\label{e:kpz:reg}
\partial_t H_{N} = \tfrac{1}{2} \Delta H_N
+
 \lambda \tilde\cN(H_N) + \xi\,,\qquad  H_N(0)=\tilde\eta
\end{align}
where\footnote{The tildas on $\tilde \cN,\tilde \eta$ are there because we will actually work with analogous quantities
that are denoted by the same symbols, without tildas.} 
$H_N=H_N(t,x)$ for $t\geq 0$ and $x\in\T^2_N$, the two-dimensional torus of side length $2\pi N$,  
\begin{itemize}[noitemsep, label=-]
\item $\tilde\eta$ is a Gaussian free field (GFF) on $\T_{N}^2$ with covariance
\begin{equ}
\E[\tilde\eta(\varphi)\tilde\eta(\psi)]=\langle(-\Delta)^{-1}\varphi,\psi\rangle_{L^2(\T_N^2)} \,,\qquad\text{for all $\varphi,\psi\in H^{-1}(\T_N^2)$,}
\end{equ}
 so that in particular, the  $0$ Fourier mode of $\varphi$ and $\psi$ is $0$,
\item $\xi$ is a space-time white noise on $\R_+\times\T_N^2$ independent of $\tilde\eta$ with covariance
\begin{equ}
\E[\xi(\varphi)\xi(\psi)]=\left\langle\varphi,\psi\right\rangle_{L^2(\R_+\times\T_N^2)}\,,\qquad\text{for all $\varphi,\psi\in L^2(\R_+\times\T_N^2)$,}
\end{equ}
\item
  the ``nonlinearity'' $\tilde\cN\eqdef \tilde\cN^1$ is defined as
  \begin{equ}[eq:nonlinH]
    \tilde \cN[H_N] \eqdef \Pi_1 \Big((\Pi_1 \partial_1 H_N)^2 - (\Pi_1 \partial_2 H_N)^2\Big),
  \end{equ}
and, for $M\in\N$, $\Pi_M$ is the operator acting in Fourier space by cutting the modes
larger than $M$, i.e. 
\begin{equ}
  \label{eq:PIN}
\widehat{\Pi_M w}(k)\eqdef \hat w(k) \mathds{1}_{|k| \le M}\,,
\end{equ} 
$\hat w(k)$ is the $k$-th Fourier component of $w$ (see below for our conventions on Fourier transforms) 
and $|k|$ denotes the Euclidean norm of $k$.\footnote{In \cite{CES19}, the r.h.s. of
  \eqref{eq:PIN} was defined with $\mathds 1_{|k|_\infty\le M}$ instead;
  however, all results proven in \cite{CES19} hold true with the
  definition \eqref{eq:PIN}; in this respect, it is important that
  both norms and have the symmetries of $\mathbb Z^2$. In
  this work we prefer to work with the Euclidean norm because it
  slightly simplifies certain technical steps.}
\item $\lambda>0$ is a constant that regulates the strength of the nonlinearity.
\end{itemize}
As was proven in \cite{CES19} (see also Lemma \ref{lem:generator} below), 
the periodic GFF $\tilde \eta$ is a stationary state for the process {\it independently} of $\lambda$ and of the cut-off parameter 
which above is set to be equal to $1$. 
From now on, ${\bf P}={\bf P}^{N}$ and $\Exp=\Exp^{N}$ will respectively denote 
the law and expectation of the stationary space-time process $H_N$, while  
$\P=\P^N$ and $\E=\E^N$ will be used for the law and expectation 
with respect to the stationary measure (the GFF).

The goal of the present paper is to understand the large-scale properties of $H_N$ as a
space-time process in comparison with the linear case
$\lambda=0$, that is simply the stochastic heat equation with additive noise. 

The first observable we consider is the {\it bulk diffusivity} which can be thought of as a measure of 
how the correlations of a process spread in space as a function of
time. 
The definition we will work with is in terms of the following Green-Kubo formula 
\begin{equ}[e:BD]
  D_N(t) = 1 + 2\frac{\lambda^2}{t}\int_0^t\int_0^s
  \int_{\T_N^2}\Exp\Big[\tilde\cN[H_N](r,x)
  \tilde\cN[H_N](0,0)]\Big]\dd x\, \dd r\, \dd s
\end{equ}
which has the advantage of being well-defined since our regularisation
of the nonlinearity ensures that $\tilde\cN[H_N]$ is smooth even if
$H_N$ is not, so that point-wise evaluation is allowed.  The
heuristics connecting the spread of the correlations of $H_N$ to the
formula above is given in Appendix~\ref{sec:heuristics}.  For now, we
simply remark that~\eqref{e:BD} is the analog in the present context
of the definition used in \cite{spohn2012large,Landim2004,Yau} for the
bulk diffusion coefficient of the asymmetric exclusion processes on
$\mathbb Z^d$, or in \cite{BQS11} for the bulk
diffusion coefficient of the one-dimensional KPZ equation.

A crucial feature of the bulk diffusivity is that it provides a way to discern if a process behaves diffusively or not. 
Indeed, while for the linear equation, which is diffusive, $D_N$ is constant in time 
(in case of~\eqref{e:kpz:reg} with $\lambda=0$, clearly $D_N\equiv 1$), 
an indication of superdiffusive behaviour can be obtained by showing that $D_N$ diverges in time 
as $t\to\infty$. For technical reasons, we will work
with the Laplace transform of $D_N$, defined for $\mu>0$ as
\begin{equ}[e:DLT]
  \mathcal D_{N}(\mu)= \mu\int_0^\infty  e^{-\mu t} t\,D_{N}(t)\, \dd t\,.
\end{equ}
The expression above differs from the usual Laplace transform in that we weighted the exponential in such a way 
that $t\mapsto \mu e^{-\mu t}$ is a probability density, which will make some expressions later on more pleasant.
We are ready to state our first result on the bulk diffusivity of $H_N$. 

Set for lightness of notation
\begin{equ}
  \Ll(x,0):=1+\lambda^2\log(1+x^{-1})
\end{equ}
(the second argument of $\Ll$ is there just for coherence with the notation introduced in \eqref{e:L} below)
and note that
\[   \Ll(x,0)\stackrel{x\to0^+}\sim \lambda^2|\log x|.\]
\begin{theorem}\label{thm:BD}
Let $\lambda>0$ and, for $N\in\N$, $D_N$ be defined according to~\eqref{e:BD} and $\cD_N$ be its Laplace transform 
as in~\eqref{e:DLT}. Then, for every $\delta>0$ there exists a constant $c_\bulk<\infty$ such that  for any $\mu>0$ sufficiently small 
\begin{equ}[e:BDUBound]
\limsup_{N\to\infty}\cD_N(\mu)\leq \frac{c_\bulk}\mu\sqrt{\Ll(\mu,0)}\big(\log\Ll(\mu,0)\big)^{5+\delta} 
\end{equ}
and 
\begin{equ}[e:BDLBound]
\liminf_{N\to\infty}\cD_N(\mu)\geq \frac{1}{\mu \,c_\bulk}\sqrt{\Ll(\mu,0)}\big(\log\Ll(\mu,0)\big)^{-5-\delta}  \,.
\end{equ}
%Moreover, the exponent $\delta$ is bounded away from zero for $\lambda\to0$.
\end{theorem}
Let us point out that by translating~\cite[Lemma 1]{QV} into our setting, 
the upper bound~\eqref{e:BDUBound} can be turned into $D_N(t)\lesssim (1+\lambda^2\log(1+t))^{1/2+o(1)}$. 
In general  the same cannot be said for the lower bound but, 
thanks to~\cite[Ch. XIII.5]{feller2008introduction}, $\mu\int_0^\infty e^{-\mu t} t\, f(t)\dd t \sim \frac1\mu(\log
(1/\mu))^{1/2}  $ as $\mu\to0$  implies $\frac1T\int_0^T t\,f(t)\dd t\sim T
(\log T)^{1/2}$ as $T\to\infty$. 
Thus, Theorem \ref{thm:BD} says that, contrary to the linear stochastic heat equation, 
{\it the bulk diffusivity of the Anisotropic KPZ equation grows essentially as the square root of the logarithm of time}, 
at least in a weak Tauberian sense, thus suggesting a superdiffusive behaviour. 
Note that instead for the KPZ equation in $d=1$,~\cite{BQS11} showed that the bulk diffusion coefficient 
grows in time as $t^{1/3}$.

As a side remark, the control of the sub-dominant corrections in Theorem \ref{thm:BD} is sharper than the one obtained in \cite{Yau} for $2d$-ASEP.
\medskip

A natural question to ask when analysing stochastic PDEs of the form~\eqref{e:kpz:reg} is what 
happens when the regularisation is removed and this is closely related
to  
the large scale properties of $H_N$. To understand this point, let us
pretend for a moment that the equation is defined on the whole plane
instead of the torus, and note that rescaling the solution $H\eqdef H_\infty$ of~\eqref{e:kpz:reg} on $\R^2$, 
diffusively, i.e. $H^{\eps}(t,x)\eqdef H(t/\eps^2,x/\eps)$ for $\eps>0$, one obtains the equation
\begin{equ}[e:AKPZ:rescaled]
\partial_t H^{\eps} = \tfrac{1}{2} \Delta H^{\eps}
+
 \lambda \tilde\cN^{1/\eps}(H^{\eps}) + \xi^{\eps}
\end{equ}
 where the nonlinearity is now smoothed via a Fourier cut-off at $\eps^{-1}$ 
and the rescaled noise $\xi^{\eps}$ is equal in distribution to the original $\xi$. 

Since $H^\eps$ (and $H$) are merely distributions (even for $\eps>0$ fixed since the noise is not regularised), 
the random variables to be considered in this context are 
\begin{equ}
H^{\eps}(t)[\phi]\eqdef \int_{\R^2}\phi(x) H^{\eps}(t,x)\dd x= H(t/\eps^2)[\phi^{(\eps)}]\,,\qquad t\geq 0
\end{equ}
for $\phi$ a smooth real-valued test function 
[from now on, for technical simplicity, $\phi$ is assumed to be at least $C^1$ and of compact support], and 
$\phi^{(\eps)}(\cdot)\eqdef \eps^2 \phi(\eps \cdot)$. Again, we want to avoid integrability issues, so we will be actually
looking at the periodic version of the quantity above, namely
\begin{equ}[e:PerHRescaled]
H^\eps_N(t)[\phi]\eqdef H_N(t/\eps^2)[\phi^{(\eps)}]
\end{equ}
in the regime when $N\gg\eps^{-1}$ (morally, we are sending $N\to\infty$ first and then $\eps\to 0$). 
For any fixed time $t$ the distribution of $H_N(t)$ (and $H^\eps_N(t)$) is the same for both $\lambda> 0$ and $\lambda=0$ 
and is given by the GFF $\tilde \eta$; therefore, in order to set apart the behaviour in the two cases, 
we will focus on the covariance between 
$ H^\eps_N(t)[\varphi]$ and $ H^\eps_N(s)[\varphi]$, which depends only on $t-s$ by stationarity, or equivalently on
the variance
\begin{equ}\label{e:Vi}
  V_\varphi^{\eps,N}(t)= \Exp\left[H^\eps_N(t)[\varphi]-H^\eps_N(0)[\varphi]\right]^2\,,\qquad t>0\,,
\end{equ}
whose Laplace transform is
\begin{equ}[eq:Vphi]
  \mathcal V_\varphi^{\eps,N}(\mu)= \mu\int_0^\infty  e^{-\mu t} V_\varphi^{\eps,N}(t)\, \dd t\,, \qquad \mu>0\,.
\end{equ}
To motivate the next result, let us recall that the linear equation ($\lambda=0$) in the whole plane
(i.e. for $N=\infty$) is invariant in law under diffusive scaling,
i.e. $H^{\eps}|_{\lambda=0}\overset{\mathrm{law}}{=}H|_{\lambda=0}$,
as is apparent from \eqref{e:AKPZ:rescaled}.
Equivalently, the random variable
$H(t)[\phi]|_{\lambda=0}-H(0)[\phi]|_{\lambda=0}$ has the same law as
$H(t/\eps^2)[\phi^{(\eps)}]|_{\lambda=0}-H(0)[\phi^{(\eps)}]|_{\lambda=0}$.
In fact, an explicit computation shows that for any $t>0$
\begin{equ}
\lim_{N\to\infty}\left.V^{\eps,N}_{\phi}(t)\right|_{\lambda=0}=\left.V^{ \infty}_\phi(t)\right|_{\lambda=0}\eqdef\frac1{2\pi^2}\int_{\mathbb R^2}\frac{|\hat\phi(k)|^2}{|k|^2}\left(1-e^{-\frac{|k|^2}2 t}\right)\dd k
\end{equ}
and consequently, the Laplace transform satisfies
\begin{equ}
\lim_{N\to\infty}\left.\cV^{\eps,N}_{\phi}(\mu)\right|_{\lambda=0} = \left.\mathcal V_\phi^{\infty}(\mu)\right|_{\lambda=0}= \frac1{4\pi^2}\int_{\mathbb R^2}\frac{|\hat\phi(k)|^2}{\mu+ \tfrac{1}{2}|k|^2}\dd k
\end{equ}
for any $\mu>0$. Note the following:
\begin{itemize}[noitemsep,label=-]
\item if $\int \phi(x)\dd x\ne0$ (so that $\hat\phi(k)$ tends to a
  non-zero constant for $k\to0$) then $t\mapsto \left.V^{\infty}_{\phi}(t)\right|_{\lambda=0}$ is a
  strictly increasing function that starts from $0$ and grows as
  $ \log t$ for $t\to\infty$, or equivalently, its Laplace 
  transform, $\mu\mapsto \left.\mathcal V_\phi^{\infty}(\mu)\right|_{\lambda=0}$, is a strictly positive
  function that tends to zero as $\mu\to\infty$ and to $+\infty$ as
  $\mu\to 0$;
\item if instead $\int \phi(x)\dd x=0$ (so that $\hat\phi(k)=O(k)$ as $k\to0$, due to the smoothness of $\phi$), 
then $t\mapsto \left.V^{\infty}_{\phi}(t)\right|_{\lambda=0}$ is again a strictly increasing function that starts from $0$ 
but this time tends to a positive constant $v_\phi$ as  $t\to\infty$ ($v_\phi$ equals twice the variance of $H[\phi]$). 
For the Laplace transform we then have that $\left.\mathcal V_\phi^{\infty}(\mu)\right|_{\lambda=0}$ is strictly positive, 
uniformly bounded above and tends to zero as $\mu\to\infty$ and to $v_\phi$ for $\mu\to0$.
\end{itemize}

It is now natural to ask if the AKPZ equation   is at least
asymptotically diffusively scale invariant, i.e. if scale invariance holds
asymptotically when first $N\to\infty$ and then $\eps\to0$. Our next
result (see in particular Corollary \ref{cor:t+t-} and the subsequent discussion) corroborates Theorem \ref{thm:BD} and again strongly indicates that
\emph{this is not the case}. More precisely, it suggests that, in order to stand any chance for $H^\eps_N$ to converge 
to some limit, one should rescale time as
$t\mapsto t/(\eps^2|\log\eps|^{1/2})$ (possibly up to corrections polynomial in $\log|\log \varepsilon|$).

\begin{theorem}\label{thm:main2} 
  For $N\in \N$ and $\lambda>0$, let $H_N$ be the solution of~\eqref{e:kpz:reg} started from the invariant measure and let  $\phi:\R^2\mapsto \R$ be compactly supported and $C^\infty$. For every $\delta>0$ there exists $c_\delta>0$ independent of $\phi$ such that the following statements hold for some constants $a_\phi,b>0$:
  \begin{itemize}[label=-]

  \item  defining  $\cV^{\eps,N}_\phi$  according to~\eqref{eq:Vphi},
    \begin{equ}[e:UBVL]
    \limsup_{N\to\infty}\cV_{\phi}^{\eps,N}(\mu)\le \frac {c_\delta}\mu\sqrt{\Ll(\mu\eps^2,0)}(\log\Ll(\mu\eps^2,0))^{5+\delta}\|\phi\|^2_{L^2(\R^2)}\,;
    \end{equ}
  \item      if $\mu=\mu(\eps)\in [a_\phi,  (1/c_\delta)\sqrt{\Ll(\eps,0)}(\log\Ll(\eps,0))^{-5-\delta}]$, then
\begin{equ}[e:UBVLmodo1]
  \liminf_{\eps\to0}    \liminf_{N\to\infty}\cV_{\phi}^{\eps,N}(\mu)\ge b
  \|\phi\|_{-1}^2\eqdef b \int_{\R^2}\frac{|\hat\phi(p)|^2}{|p|^2}\dd p\,
\end{equ}
(the  integral is finite iff $\int_{\R^2}\phi(x)\dd x=0$.)
  \end{itemize}
\end{theorem}
The restriction $\mu(\eps)\ge a_\phi$ in the lower bound is purely technical; at any rate, the interesting regime for our purposes (see the proof of Corollary \ref{cor:t+t-}) corresponds to $\mu(\eps)$ diverging as $\approx\sqrt{|\log\eps|}$.
To appreciate the meaning of Theorem \ref{thm:main2}, note that,
 defining
 \begin{equ}
    t_-(\eps)\eqdef \frac1{ \sqrt{|\log\eps|}}(\log|\log\eps|)^{-5-\delta},\quad
    t_+(\eps)\eqdef \frac 1{\sqrt{|\log\eps|}}(\log|\log\eps|)^{5+\delta},
    \label{e:t+t-}
 \end{equ}
     $\cV^{\eps,N}_\phi(\mu)$  is essentially zero if $\mu\gtrsim 1/t_-(\eps)$ while it is strictly positive (or exploding, if $\int_{\R^2}\phi(x)\dd x\ne0$) if  $\mu\lesssim 1/t_+(\eps)$.
 % also that  the right hand side in~\eqref{e:UBVLmodo0} explodes unless one chooses $\mu\gtrsim\sqrt{|\log\eps|}$. Thus, u
 Using the scaling relation
\begin{equ}
\mu \int_0^\infty e^{-\mu t} V_\phi^{\eps,N}(t/\tau)\dd t 
= \cV_\phi^{\eps,N}(\mu \tau)\,,\quad \tau>0\,,
	\end{equ}
we see that the ``correct'' time scale to observe non-trivial correlations of the process $H^\eps_N$ is $\approx1/(\eps^2\sqrt{|\log\eps|})$.

This observation can be made sharper in the case $\phi$ has zero average. In fact, with little extra work, 
we will deduce from Theorem \ref{thm:main2} the following corollary. 

  \begin{corollary}
    \label{cor:t+t-}
    Let $\phi$ be a compactly supported, $C^\infty$ test function of zero mean, and let $\delta>0$. One has, with $t_\pm(\eps)$ defined as in \eqref{e:t+t-},
    \begin{equ}
      \label{e:cor1}
  \inf_{t\le t_-(\eps)}\liminf_{N\to\infty}\frac{{\rm Cov}(H^{\eps}_N(t)[\phi], H^{\eps}_N(0)[\phi])}{{\rm
      Var}(H^{\eps}_N(0)[\phi])}=1.
\end{equ}
On the other hand, there exists $t=t(\eps)\in (t_-(\eps),t_+(\eps))$ and $a<1$ independent of $\eps,\phi$ such that
\begin{equ}
  \label{e:cor2}
 \limsup_{\eps\to0} \limsup_{N\to\infty}\frac{{\rm Cov}(H^{\eps}_N(t)[\phi], H^{\eps}_N(0)[\phi])}{{\rm
      Var}(H^{\eps}_N(0)[\phi])}\le a.
\end{equ}
  \end{corollary}
 In other words,
$H^{\eps}_N(t)[\phi]$ and $H^{\eps}_N(0)[\phi]$ are almost perfectly correlated for times smaller than $t_-(\eps)$ but,  contrary to what happens in the linear case, they
 decorrelate non-trivially already on a time-scale of order $t_-(\eps)\le t(\eps)\le t_+(\eps)\ll1$.
 To see the relation with  Theorem \ref{thm:main2}, note first that if $\int \phi(x)\dd x=0$ then by stationarity the variance
of $H_N^\eps(0)[\phi]$ is finite
uniformly in both $N$ and $\eps$ (for $N\to\infty$, it tends to $(2\pi)^{-2}\|\phi\|^2_{-1}$).
Note also that
\begin{equ}
  \label{e:nat}\frac{{\rm Cov}(H^{\eps}_N(t)[\phi], H^{\eps}_N(0)[\phi])}{{\rm
      Var}(H^{\eps}_N(0)[\phi])}= 1-\frac{  V^{\eps,N}_{\phi}(t)}{2{\rm
      Var}(H^{\eps}_N(0)[\phi])}.
  % V^{\eps,N}_{\phi}(t)=2{\rm
  %     Var}(H^{\eps}_N(0)[\phi])-2{\rm Cov}(H^{\eps}_N(t)[\phi], H^{\eps}_N(0)[\phi]).
\end{equ}
\begin{remark}
  The existence of the $N\to\infty$ limits is shown in~\cite{CK} but with a slightly different regularisation (the cut-off chosen is 
  smooth in Fourier space) so we preferred to state the above results with $\liminf$ and $\limsup$. 
  Actually, as will appear from the proof, Theorem \ref{thm:main2} holds in the more general setting where $\eps\to0$ and
  $N\to\infty$ jointly, with $N\eps\to\infty$.
\end{remark}

% We formulate the following conjecture 
% for which we provide a heuristic in Appendix~\ref{app:aheuri}.

% \begin{conjecture}\label{con:delta}
%   The optimal value of $
% \delta$ for which the statements of Theorems~\ref{thm:BD} and
% \ref{thm:main2} hold is $\delta=1/2$.
% %, we have that there exists a constant $c_\bulk>0$ such that
% %\begin{equ}
% %\lim_{N\to\infty}\cD_N(\mu) = \frac{1+c_\bulk}\mu\left(1+\lambda^2\log\left(1+\fr%ac1{\mu}\right)\right)^{\tfrac12}\,.%
% %	\end{equ}
% 	\end{conjecture}
	
        \medskip

Our last result is a bit different in spirit and our main motivation here is to establish a connection 
with similar statements proven for \emph{discrete} growth models in the AKPZ universality class, as for instance in~\cite{Toninelli2017,MR4033679,Ler}. In the discrete setting, one natural viewpoint is to  look  at
the large-time behaviour of the height at a single point, and in particular at the growth of its variance. 
Since, as remarked above, point evaluation is not possible in the present context, 
we look at the locally averaged field, i.e. we test 
$H_N$ against a \emph{fixed} test function $\phi$ and obtain an upper bound on $V_\phi^{1,N}(t)$ in the
$N\to\infty$ limit, for $t$ arbitrarily large.

\begin{theorem}
  \label{thm:logt}
  For $N\in \N$, let $H_N$ be the solution of~\eqref{e:kpz:reg}, started from the invariant measure.  For any
  compactly supported test function $\varphi$ on $\mathbb R^2$ there exists $c_\phi>0$ such that, for every $t>0$,
  \begin{equ}[eq:logt1]
   \limsup_{N\to\infty}  V_\phi^{1,N}(t)\le c_\phi(1+\lambda^2)\max(\log t, 1).
 \end{equ}
\end{theorem}

\begin{remark}
  \label{rem:logt}
  It is well known (and it can be checked using the explicit solution) that for the linear equation, one has
\begin{equ}[eq:logt2]
 \lim_{N\to\infty}  \left.V_\phi^{1,N}(t)\right|_{\lambda=0}\stackrel{t\to\infty}\sim c_\phi \log t , \quad \lambda=0.
\end{equ}
While \eqref{eq:logt1} and \eqref{eq:logt2} show the same large-time
behaviour (at least as an upper bound), this \emph{is not in contradiction with
the fact that the relevant scaling for the process with   $\lambda>0$ is different from diffusive} 
as shown in Theorems~\ref{thm:BD} and~ \ref{thm:main2}. Indeed, the $\log t$ behaviour is not a distinguishing feature of the $2$-dimensional stochastic heat equation. 
For instance consider the fractional stochastic heat equation
\begin{equ}
  \partial_t Z=- \tfrac12(-\Delta)^{\theta}Z+(-\Delta)^{\frac{\theta-1}{2}}\xi
\end{equ}
with $Z=Z(t,x),\, t\ge 0, \, x\in \mathbb R^2$, $\theta\in (0,1)$, $\xi$ a space-time white noise as above and 
$(-\Delta)^\theta$ acting in Fourier space as a multiplication by $|k|^{2\theta}$. 
It is easily checked that the GFF on the whole plane is stationary for this equation and $Z$ is scale invariant under 
the superdiffusive scaling $t\to t/\eps^{2\theta},\,x\to x/\eps$. 
Nonetheless, an explicit computation shows that for the stationary process
\begin{equ}[eq:shedtheta]
  \Exp[z(t)[\phi]-z(0)[\phi]]^2= \frac1{2\pi^2}\int_{\mathbb R^2}\frac{|\hat\phi(k)|^2}{|k|^2}\left(
  1-e^{-\frac t2 |k|^{2\theta}}\right)\dd k\stackrel{t\to\infty}\sim c_{\phi,\theta}\log t,
\end{equ}
as is the case for the usual stochastic heat equation where $\theta=1$. Note that, in contrast, in dimension  $d=1$, the large-time behaviour of \eqref{eq:shedtheta} is power-law for large $t$, with a $\theta$-dependent  exponent $1/(2\theta)$.
\end{remark}

\subsection*{Organization of the article}
The rest of this work is organized as follows. In Section
\ref{sec:prelim} we turn the equation \eqref{e:kpz:reg} into a
regularized Burgers equation, we introduce some preliminary formalism
and we recall some basic results from \cite{CES19}. Section
\ref{sec:core} is the core of the work and the main outcome are upper
and lower bounds on the variance of the time integral of the
nonlinearity. Given those bounds, Theorems \ref{thm:BD} and
\ref{thm:main2} are proven in Section
\ref{sec:provathmmain}. The proof of Theorem \ref{thm:logt} is instead
based on different (simpler) tools and it is contained in Section
\ref{sec:logt}. Finally, in the appendix we provide a heuristic for the Green-Kubo formula~\eqref{e:BD}, a heuristic argument explaining our main result of $\sqrt{\log t}$ diffusivity, and collect some technical results.

\subsection*{Notation}

For $N>0$, let $\Z_N\eqdef \Z/N$ and 
$\T_N^2$ be the two-dimensional torus of side length $2\pi N$. If $N=1$ then we simply write $\T^2$ instead of $\T_N^2$.
We denote by $\{e_k\}_{k\in\Z_N^2}$ the Fourier basis defined via 
$e_k(x) \eqdef \frac{1}{2\pi} e^{i k \cdot x}$ which, for all $j,\,k\in\Z_N^2$, satisfies 
$\langle e_k, e_{-j}\rangle_{L^2(\T^2_N)}= \delta_{k,j}  N^2$. 
% {\color{red} I think the following sentence can be omitted: The basis functions $e_k$ can be decomposed in their real and imaginary part, so that $e_k=a_k+i b_k$ 
% and the system $\{a_k\}_{k \in \Z_{\mathrm{diag}}^2} \sqcup \{b_k\}_{k \in \Z_{\mathrm{diag}}^2\setminus\{0\}}$ 
% forms a real valued orthogonal basis of $L^2(\T_L^2)$, where $\Z_{\mathrm{diag}}^2= \{(k_1,k_2)\in\Z^2:\, k_1\geq k_2\}$. }

The Fourier transform of a given function $\phi\in L^2(\T^2_N)$ will be represented as 
$\cF(\phi)$ or by $\hat\phi$ and, for $k\in\Z^2_N$
is given by the formula
	\begin{equation}\label{e:FT}
	\cF(\varphi)(k) =\hat\varphi(k)\eqdef  \int_{\T_N^2} \varphi(x) e_{-k}(x)\dd x\,, 
	\end{equation}
so that in particular
\begin{equ}[e:FourierRep]
\varphi(x) = \frac{1}{N^2}\sum_{k\in\Z_N^2} \hat\varphi(k) e_k(x)\,,\qquad\text{for all $x\in\T_N^2$. }
\end{equ}
For any real valued distribution $\eta\in\cD'(\T_N^2)$ and $k\in\Z_N^2$, 
we will denote its Fourier transform by 
\begin{equation}\label{e:complexPairing}
\hat \eta(k)\eqdef \eta(e_{-k})%=\eta(a_k)-i \eta(b_k)
\end{equation} 
and note that $\overline{\eta(e_k)}=\eta(e_{-k})$. 
Moreover, we recall that the Laplacian $\Delta$ on $\T_N^2$ has eigenfunctions $\{e_k\}_{k \in \Z_N^2}$ 
with eigenvalues $\{-|k|^2\,:\,k\in\Z_N^2\}$, so that, for $\theta>0$, we can define the operator $(-\Delta)^\theta$
by its action on the basis elements 
	\begin{equation}\label{e:fLapla}
	(-\Delta)^\theta e_k\eqdef |k|^{2\theta}e_k\,,
	\end{equation}	
for $k\in\Z_N^2$. 

Throughout the paper, we will write $a\lesssim b$ if there exists a constant $C>0$ such that $a\leq C b$ and $a\sim b$ if $a\lesssim b$ and $b\lesssim a$. We will adopt the previous notations only in case in which the hidden constants do not depend on any quantity which is relevant for the result.

\section{Preliminaries}
\label{sec:prelim}

The aim of this section is twofold. On the one hand, we will state some basic tools from Wiener space analysis that 
we will need in the rest of the paper while on the other hand
we will reduce the analysis of~\eqref{e:kpz:reg} to the torus of length size $1$, i.e. to the setting of~\cite{CES19}, and 
recall some of the results on the Anisotropic KPZ equation 
obtained therein.  
\medskip

Notice at first that an immediate scaling argument guarantees that for any $N\in\N$, 
\begin{equ}[e:Scaling]
H_N(t,x)\overset{{\rm law}}{=} h^N(t/N^2,x/N)\,,\qquad t\geq 0\text{ and }x\in\T_N^2
\end{equ}
where $h^N$ is the solution of 
\begin{equ}[e:akpz:torus1]
\partial_t h^{N} = \tfrac{1}{2} \Delta h^N
+
 \lambda \tilde\cN^N(h^N) + \xi\,,\qquad  h^N(0)=\tilde\eta
\end{equ}
in which $h^N=h^N(x,t)$ for $t\geq 0$, $x\in\T^2\eqdef\T_1^2$, and all the other quantities are defined as in 
the discussion after~\eqref{e:kpz:reg}. 
Therefore, even though all the statements in the introduction as well as the results we aim for are formulated 
(and ultimately proved) for the solution $H_N$ of~\eqref{e:kpz:reg},~\eqref{e:Scaling} guarantees that 
we can focus instead on $h^N$ since whatever is shown for the latter can then be translated back to $H_N$. 

As in~\cite{CES19}, it turns out to be convenient to work with 
the Stochastic Burgers equation instead of AKPZ, which can be derived from~\eqref{e:akpz:torus1} by setting 
$u^N\eqdef(-\Delta)^{\frac{1}{2}}h^N$ 
so that $u^N$ solves 
\begin{equation}\label{e:AKPZ:u}
\partial_t u^N = \tfrac{1}{2} \Delta u^N
+
\lambda \cM^N[u^N] + (-\Delta)^{\frac{1}{2}}\xi, \quad u^N(0)=\eta\eqdef(-\Delta)^{\frac12}\tilde \eta
\end{equation}
where $u^N=u^N(t,x)$, $t\geq 0$, $x\in\T^2$, and the nonlinearity $\cM^N$ is given by 
\begin{equation}\label{e:nonlin}
\cM^N[u^N]\eqdef (-\Delta)^{\frac{1}{2}}\Pi_N \Big((\Pi_N \partial_1(-\Delta)^{-\frac{1}{2}} u^N)^2 - (\Pi_N \partial_2 (-\Delta)^{-\frac{1}{2}} u^N)^2\Big)\,.
\end{equation}

Note that, since $\tilde \eta$ is a standard Gaussian Free Field, $\eta$ is a (spatial) white noise on $\mathbb T^2$ 
whose basic properties are recalled in the next section (for more on it 
see~\cite[Chapter 1]{Nualart2006}, or~\cite{GPnotes,GPGen} and~\cite[Section 2]{CES19}). 

\subsection{Elements of Wiener space analysis}
\label{S:Malliavin}

Let $(\Omega,\cF,\P)$ be a complete probability space and 
$\eta$ be a mean-zero spatial white noise on the two-dimensional torus $\T^2$, i.e. 
$\eta$, defined in $\Omega$, is a centered isonormal Gaussian process 
(see~\cite[Definition 1.1.1]{Nualart2006}), on $H\eqdef L^2_0(\T^2)$, 
the space of square-integrable functions with $0$ total mass,
whose covariance function is given by 
\begin{equ}\label{eq:spatial:white:noise}
\E[\swn(\vphi)\swn(\psi)]=\langle \vphi, \psi\rangle_{L^2(\T^2)}
\end{equ}
where $\varphi,\psi\in H$ and $\langle\cdot,\cdot\rangle_{L^2(\T^2)}$ is the usual scalar product in $L^2(\T^2)$. 
For $n\in\N$, let $\SH_n$ be the {\it $n$-th homogeneous Wiener chaos}, i.e. the closed linear subspace of 
$L^2(\eta)\eqdef L^2(\Omega)$ generated by the random variables $H_n(\eta(h))$, 
where $H_n$ is the $n$-th Hermite polynomial, and 
$h\in H$ has norm $1$. By~\cite[Theorem 1.1.1]{Nualart2006}, $\wc_n$ and $\wc_m$ are orthogonal whenever 
$m\neq n$ and $L^2(\eta)=\bigoplus_{n}\SH_n$. 
Moreover, there exists a canonical contraction $\sint{}:\bigoplus_{n\ge 0} L^2(\T^{2n}_L) \to L^2(\eta)$, 
which restricts to an isomorphism $\sint{}:\fock \to L^2(\eta)$ on the Fock space $\fock:=\bigoplus_{n \ge 0} \fock_n$, 
where $\fock_n$ denotes the space $L_\sym^2(\T_L^{2n})$ of functions in $L^2(\T_L^{2n})$ which are symmetric
with respect to permutation of variables. The restriction of $\sint$ to $\fock_n$, $\sint_n$, 
called $n$-th (iterated) Wiener-It\^o integral with respect to $\eta$, is itself an isomorphism from $\fock_n$ to $\SH_n$
so that by~\cite[Theorem 1.1.2]{Nualart2006}, for every $F\in L^2(\eta)$ there exists $f=(f_n)_{n\in\N}\in\fock$ 
such that 
\begin{equ}[e:Isometry]
F=\sum_{n\geq 0} I_n(f_n)\qquad\text{and}\qquad \|F\|_\eta^2 = \sum_{n\geq 0} n!\|f_n\|_{L^2(\T^{2n})}^2
\end{equ}
and we take the right hand side as the definition of the scalar product on $\fock$, i.e. 
\begin{equ}[e:ScalarProdFock]
 \langle f,g\rangle_{\fock}\eqdef \sum_{n\geq 0} \langle f_n,g_n\rangle_{\fock_n}\eqdef\sum_{n\geq 0} n!\langle f_n,g_n\rangle_{L^2(\T^{2n})}.
\end{equ}

We conclude this paragraph by mentioning that we will mainly work with the Fourier representation $\{\hat \eta(k)\}_k$ of 
$\eta$, which is a family of complex valued, centered Gaussian random variables such that 
\begin{equs}[e:NoiseFourier]
\hat \eta(0)=0\,,\qquad \overline{\hat \eta(k)}=\hat\eta({-k})\qquad\text{and}\qquad \E[\hat \eta(k)\hat \eta(j)]=\1_{k+j=0}\,.
\end{equs}

\subsection{Stochastic Burgers equation and its Generator}
\label{S:Properties}

The properties of equation~\eqref{e:AKPZ:u} which will be important for us were
obtained in~\cite[Section 3]{CES19}. 
In order to fix the relevant notations, below we recall the Fourier representation of~\eqref{e:AKPZ:u} 
and summarise some of its features referring to~\cite{CES19} for the proofs. 
\medskip 

The Fourier representation of~\eqref{e:AKPZ:u} is equivalent to an infinite system of 
(complex-valued) SDEs given by 
\begin{equation}\label{e:kpz:u}
\dd \hat u^{N}(k) =
\Big(-\frac{1}{2}|k|^2 \hat u^{N}(k)
+
 \lambda\cM^N_k[u^{N}]\Big) \dd t +  |k|\dd B_k(t)\,,\qquad k\in\Z^2\setminus\{0\}\,.
\end{equation}
The $k$-th Fourier component of the nonlinearity is
\begin{align}
\cM_k^N[u^{N}]&\eqdef\cM^N[u^{N}](e_{-k})= |k|\sum_{\ell+m=k}\nonlin_{\ell,m}\hat  u^{N}(\ell) \hat u^{N}(m)\,,\label{e:nonlinF}\\
 \nonlin_{\ell,m} &\eqdef \frac{1}{2\pi} \frac{c(\ell,m)}{|\ell||m|} \indN{\ell,m}\,,\qquad c(\ell,m) \eqdef \ell_2m_2- \ell_1m_1 \label{e:nonlinCoefficient}
\end{align}
where $\ell=(\ell_1,\ell_2),\,m=(m_1,m_2)\in\Z^2$ and 
%\begin{eqnarray}
%  \label{eq:clm}
%  c(\ell,m) \eqdef \ell_2m_2- \ell_1m_1 ,
%\end{eqnarray}
\begin{equ}[eq:JN]
\indN{\ell,m}\eqdef
\mathds{1}_{0<|\ell|\le N, 0<|m|\le N, |\ell+m|\le N}
\end{equ}
and all the variables in the sum \eqref{e:nonlinF} range over $\Z^2\setminus\{0\}$ (the value $0$ is automatically excluded by the definition of $\indN{}$).

In~\eqref{e:kpz:u}, the $B_k$'s are complex valued Brownian motions defined via  
$B_k(t)\eqdef \int_0^t \hat \xi(s,k)\, ds$, $\hat \xi(k)=\xi(e_{-k})$, so that (recalling that $\xi$ is a space-time white noise)
\begin{equ}
\overline{B_k}= B_{-k}\,,\qquad\text{and}\qquad \dd\langle B_k, B_\ell \rangle_t =  \mathds{1}_{\{k+\ell=0\}}\, \dd t\,.
\end{equ}

Since eventually we are interested in $h^{N}$ rather than in $u^{N}$, note  that $(-\Delta)^{\frac{1}{2}}$ is an invertible 
linear bijection on functions with zero mass, so that we can recover all the non-zero Fourier components of 
$h^{N}$ via
\begin{equ}
  [e:hdak]
\hat h^{N}(k)=\frac{\hat u^{N}(k)}{|k|},\quad k\neq 0
\end{equ}
 On the other hand,  
the zero-mode $\hat h^{N}(0)$ is also a function of $u^{N}$ and of an independent Brownian motion, 
since it satisfies 
\begin{equ}[eq:modozero]
\dd \hat h^{N}(0)= \lambda \cN^N_0+ \dd B_0(t)
\end{equ}
where 
\begin{equ}
  \cN^N_0= \sum_{\ell+m=0}\nonlin_{\ell,m}\hat u^{N}(\ell)\hat u^{N}(m)\,.
\end{equ}

Proposition 3.4 of \cite{CES19} guarantees that, for any $N\in\N$, the process 
$t\mapsto \hat u^{N}(t)=\{\hat u^{N}(t,k)\}_{k\in\Z^2\setminus\{0\}}$
solution to~\eqref{e:kpz:u} is a 
strong Markov process 
and we denote its generator by $\gen$. If the initial condition is
white noise, the law of the process is also translation invariant.
Let $F$ be a cylinder function acting on the space of distributions $\CD'(\T^2)$ and 
depending only on finitely many Fourier components, i.e. $F$ is such that there exists a smooth function 
$f=f((x_k)_{k\in\Z^2\setminus\{0\}})$ with all derivatives growing at most polynomially and depending 
only on finitely many variables, 
for which $F(\eta)=f((\hat \eta(k))_{k\in\Z^2\setminus\{0\}})$. 
Then, $\gen$ can be written as
the sum of $\gensy$ and $\gena$, whose action on $F$ as above is given by
\begin{align}
(\gensy F)(v) &\eqdef \sum_{k \in \Z^2} \half|k|^2 (-\hat v({-k}) D_k  +  D_{-k}D_k )F(v) \label{e:gens}\\
(\gena F)(v) &= \lambda \sum_{m,\ell \in \Z^2\setminus\{0\}} |m+\ell|\nonlin_{m,\ell} \hat v(m) \hat v(\ell) D_{-m-\ell} F(v)\,. \label{e:gena}
\end{align}
Here, for $k\in\Z^2$ and $F$ as above, $D_k F$ is defined as\footnote{For more on the actual definition of cylinder function, Malliavin derivative and the formula below, we address the reader to~\cite[Section 2 and Lemma 2.1]{CES19}}
\begin{equ}
  \label{e:Malliavin}
D_{k} F\eqdef (\partial_{x_k} f)((\hat \eta({k}))_{k\in\Z^2\setminus\{0\}})\,.
\end{equ}

In the following lemma and throughout the remainder of the paper, we will slightly abuse notations and use the same symbol 
to denote an operator acting on (a subspace of) $L^2(\eta)$ and its Fock space version. 
%, 
%i.e. if $\CO$ is an operator in $L^2(\eta)$ 
%then for any $F\in L^2(\eta)$ such that $F=\sum_n\sint_n(f_n)$, $(f_n)_n\in\fock$, we will write
%$\CO F=\CO \sum_n\sint_n(f_n)=\sum_n\sint_n(\CO f_n)$. 
%
%
%
%In order to state the following result, let us recall that the Fourier transform 
%$\CF$ maps $\fock_n= L^2_\sym(\T_L^{2n})$ (isometrically) 
%into $ \ell^2((\Z_L^2)^n)$, i.e. $\CF(\cdot)=\hat\cdot : L^2_\sym(\T^{2n}) \to \ell^2((\Z^2)^n)$ and, 
%notation-wise, for any operator $\CO$ acting on (a subspace of) $L^2(\swn)$, we will denote by 
%$\mathfrak  O$ the operator corresponding to $\CO$ but acting on the Fock space $\fock$, i.e. such that for all 
%$\varphi\in \fock$ one has $\CO \sint(\varphi)= \sint(\mathfrak O\varphi)$. 
%
%\giuseppeText{I checked Gubinelli and Perkowski's paper and the abuse the notation by indicating with the same symbol 
%the operator acting on $L^2$ and that acting in Fock space...should we do the same? I think it would be prettier...}

\begin{lemma}\cite[Lemmas 3.1 and 3.5]{CES19}\label{lem:generator}
For any $N\in\N$, the spatial white noise $\eta$ on $\T^2$ defined in \eqref{eq:spatial:white:noise} is invariant 
for the solution $\hat u^{N}$ of~\eqref{e:kpz:u} and, with respect to $\eta$, 
the symmetric and anti-symmetric part of $\gen$ are given by 
$\gensy$ and $\gena$, respectively.

Moreover, for all $n \in \N$ the operator $\gensy$ leaves $\wc_n$ invariant, while $\gena$ can be written as the sum 
of two operators $\genap$ and $\genam$ which respectively map $\wc_n$ into $\wc_{n+1}$ and $\wc_{n-1}$ and 
are such that $-\genap$ is the adjoint of $\genam$. 

Finally, on the Fock space $\fock$, we have that $\gensy=-\half\Delta$ and,
in Fourier variables, the action $\gensy$, $\genam$ and $\genap$ on $\phi_n\in\fock_n$ is given by 
\begin{align}
\CF(\gensy \varphi_n) (k_{1:n}) &= \tfrac{1}{2} |k_{1:n}|^2 \hat\phi_n(k_{1:n}) \label{e:gens:fock}\\
\CF(\genap \phi_n)(k_{1:n+1}) 
&= n \lambda 
|k_1+k_2|\nonlin_{k_1, k_2} \hat\phi_n(k_1 + k_2, k_{3:n+1})  \label{e:genap:fock}
\\
\CF(\genam \phi_n)(k_{1:n-1})
&=
2n(n-1) \lambda
\sum_{\ell+m = k_1}|m|
\nonlin_{k_1, -\ell} \hat \phi_n(\ell,m, k_{2:n-1}), \label{e:genam:fock}
\end{align}
where the functions on the right hand side  need to be symmetrised with respect to all permutations of their arguments 
(see e.g.~\eqref{eq:Fourier_asym}). 
In~\eqref{e:gens:fock}-\eqref{e:genam:fock}, all the variables belong to $\Z^2\setminus\{0\}$ and 
we adopted the short-hand notations $k_{1:n}\eqdef (k_1,k_2,\dots,k_n)$ and 
$|k_{1:n}|^2\eqdef |k_1|^2+\dots+|k_n|^2$. 
\end{lemma}

For later purposes, let us introduce the following definition:
\begin{definition}\label{def:diagonal}
 An operator $\cZ$ on $\fock$ is said to be diagonal if for every $n$ it maps $\fock_n$ into itself, and it acts in Fourier space as a multiplier, that is there exists a sequence of symmetric functions $\zeta=(\zeta_n)_{n\ge1}$ such that for all $n$ and $\phi\in \fock_n$, one has $\cF(\cZ \phi)(k_{1:n})=\zeta_n(k_{1:n})\hat\phi(k_{1:n})$.
\end{definition}
In this sense, the operator $\gensy$ is diagonal while $\genap,\genam$ are clearly not.

\section{The variance of the nonlinearity}
\label{sec:core}
The present section represents the bulk of the paper and focuses on the term in~\eqref{e:kpz:reg} that 
distinguishes SHE and AKPZ, i.e. the nonlinearity. In particular, we aim at estimating  
from above and below the Laplace transform of the second moment of
the time integral of $\tilde\cN^N(h^{N})$ tested against a suitable test function, 
see Proposition~\ref{p:mainB} for the result we are after.
We will then see later in Section~\ref{sec:provathmmain} that the nonlinearity gives the dominant contribution to the bulk diffusivity $D_N$ (Lemma~\ref{l:BD}) and to $h^N(t)[\phi]-h^N(0)[\phi]$ (Proposition~\ref{p:mainAC}).
\medskip

Let $\phi$ be a sufficiently regular test function and, for $t\geq 0$, denote the time integral of the 
nonlinearity against $\phi$ by  
\begin{equ}[e:BN]
B^N_\phi(t)\eqdef\int_0^t\lambda \tilde{  \mathcal N}^N(h^N(s))[\phi]\dd s=\int_0^t \lambda\cN^N(u^N(s))[\varphi]\dd s
\end{equ}
where the second equality is an immediate consequence of~\eqref{e:nonlinCoefficient} and \eqref{e:hdak} 
once we set 
\begin{equ}[eq:ches]
\cN^N(u^N)[\varphi]\eqdef\sum_{\ell,m\in\mathbb Z^2} \nonlin_{\ell,m}\hat u^N(\ell)\hat u^N(m) \hat \varphi(-\ell-m)\,.
\end{equ} 
%In order to simplify the forthcoming analysis and exploit the results in the previous section, 
%we point out that $\tilde{  \mathcal N}^N(h^N)$ can be expressed 
%as a functional of $u^N$. Indeed, thanks to~\eqref{e:nonlinCoefficient} and~\eqref{eq:eventu}, we have 
%\begin{equ}[eq:ches]
%\tilde{  \mathcal N}^N(h^N)[\varphi]=\lambda\sum_{\ell,m\in\mathbb Z^2} \nonlin_{\ell,m}\hat u^N(\ell)\hat u^N(m) \hat \varphi(-\ell-m)=:\cN^N(u^N)[\varphi]
%\end{equ}
%which holds since $\inonlin_{\ell,m}$ prevents $\ell$ and $m$ from being $0$. 
%

In the stationary process, the random variable $B_\phi^N$ is centered.  This follows from \eqref{eq:ches} and from the
anti-symmetry of $\nonlin_{\ell,m}$ under $ \ell=(\ell_1,\ell_2)\mapsto(\ell_2,\ell_1),  m=(m_1,m_2)\mapsto(m_2,m_1)$. Its variance then coincides with its second moment.

By~\eqref{e:BN},~\eqref{e:Isometry} and~\cite[Lemma 5.1]{CES19}, 
for any $\mu>0$, the Laplace transform of the second moment of $B^N_\phi$ satisfies
\begin{equ}[e:Laplace]
\cB_\phi^N(\mu)\eqdef	\mu\int_0^\infty \,
	e^{-\mu t}\Exp\Big[ \big(B^N_\phi(t)\big)^2 \Big] \dd t
	= \frac{2}{\mu} \langle \nf_\phi,\,(\mu- \gen)^{-1}\nf_\phi\rangle_{\fock}
\end{equ} 
where $\gen$ is the generator of $u^{N}$ and $\nf_\phi$ is the representation in Fock space of 
$\lambda \cN^N[\eta](\phi)$, i.e.
\begin{equ}[eq:explFock]
\lambda\cN^N[\eta](\phi)=\sint_2(\nf_\phi)\qquad\text{with}\qquad \hat{\mathfrak n}^N_\phi(\ell,m)=\lambda\nonlin_{\ell,m}\hat \varphi(-\ell-m),\quad\ell,m\in\Z^2
\end{equ}
as can be read off~\eqref{eq:ches}. 
\medskip

Now, in order to control $\cB_\phi^N$ we need to improve our
understanding of the scalar product at the right hand side
of~\eqref{e:Laplace}, which in particular means that we need to invert $\mu-\gen$, 
which is though not feasible in view of the singularity induced by the antisymmetric part of operator $\gen$, i.e. $\gena$.  
To overcome this difficulty we will exploit a technique first
established in~\cite{Landim2004} and explored in full strength
in~\cite{Yau}, where the authors studied the superdiffusivity of the
asymmetric simple exclusion process in dimension $d=1,2$, 
and which essentially consists in truncating the resolvent equation. 
To be more precise for $n\in\N$, let $I_{\leq n}$ be the
projection onto $\fock_{\leq n}\eqdef\bigoplus_{k=0}^n \fock_k$ and
$\gen_n = \sint_{\leq n}\gen \sint_{\leq n}$. Then, let
$\hf^{N,n}\eqdef (\hf_j^{N,n})_{j\leq n}\in \fock_{\leq n}$ be the
solution of the {\it truncated generator equation} 
\begin{equ}[e:GenEqun]
(\mu-\gen_n)\hf^{N,n}=\nf_\phi\,
\end{equ}
(which will be given explicitly below), and further write $\hf^N=(\mu-\gen)^{-1}\nf_\phi$. %We omitted the dependence of $\hf^{N,n}$ on $\phi$ not to clutter the notation. 
The property of $\hf^{N,n}$ that allows one to reduce the analysis to that of the truncated resolvent equation is stated in the following lemma, 
derived in~\cite[Lemma 2.1]{Landim2004}. 

\begin{lemma}\label{l:Sandwich}
Let $\mu>0$. Then, for every $n\in\N$, we have that 
\begin{equs}
\langle \nf_\phi,\hf^{N,2n+1}\rangle_{\fock}%&\leq \nf_\phi,\hf^{N,2n+3}\rangle_{\fock}
\leq \langle \nf_\phi,(\mu-\gen)^{-1}\nf_\phi\rangle_{\fock}%\leq \langle \nf_\phi,\hf^{N,2m+2}\rangle_{\fock}\leq 
\leq\langle \nf_\phi,\hf^{N,2n}\rangle_{\fock}\,.
\end{equs}
Moreover, the sequence $\{\langle \nf_\phi,\hf^{N,2n+1}\rangle_{\fock}\}_n$ is increasing while 
$\{\langle \nf_\phi,\hf^{N,2n}\rangle_{\fock}\}_n$ is decreasing and they both converge to 
$\langle \nf_\phi,\hf^{N}\rangle_{\fock}$ as $n\to\infty$. 
\end{lemma}

Notice that, thanks to Lemma~\ref{l:Sandwich}, on the one hand we have reduced the problem of studying the solution of the 
full generator equation (which is the same as~\eqref{e:GenEqun} with $\gen$ replacing $\gen_n$) to that 
of its truncated version given in~\eqref{e:GenEqun}. On the other hand, by orthogonality 
\begin{equ}
\langle \nf_\phi,\hf^{N,n}\rangle_{\fock}=\langle \nf_\phi,\hf^{N,n}_2\rangle_{\fock_2},\qquad\text{ for all $n\in\N$,}
\end{equ}
so that we only need to determine the component of $\hf^{N,n}$ in $\fock_2$. 

Getting back to~\eqref{e:GenEqun}, by Lemma~\ref{lem:generator} 
$\gen$ can be decomposed in the sum of $\gensy$, $\genap$ and $\genam$, the first of which leaves the order 
of the Wiener chaos component invariant, whereas the others respectively increase and decrease it by $1$. 
Thus, the truncated generator equation coincides with the following hierarchical system 
\begin{equation}\label{e:System}
\begin{cases}
\big(\mu-\gensy\big)\hf^{N,n}_n-\genap\hf^{N,n}_{n-1}=0,\\
\big(\mu-\gensy\big)\hf^{N,n}_{n-1}-\genap\hf^{N,n}_{n-2}-\genam\hf^{N,n}_{n}=0,\\
\dots\\
\big(\mu-\gensy\big)\hf^{N,n}_2-\genap\hf^{N,n}_1-\genam\hf^{N,n}_3=\nf_\phi,\\
\big(\mu-\gensy\big)\hf^{N,n}_1-\genam\hf^{N,n}_2=0,
\end{cases}
\end{equation}
where, in the last equation we exploited the fact that $\genap$ is $0$ on constants by Lemma~\ref{lem:generator}.
Let us introduce the operators $\Op_k$, $k\geq 2$, which are recursively defined via 
\begin{equation}\label{def:OpH}
\begin{cases}
\Op_2\equiv 0\,,\\
%\Op_3= -\genam\big(\mu-\gensy\big)^{-1}\genap\,,\\
\Op_k = -\genam[\big(\mu-\gensy\big)+\Op_{k-1}]^{-1}\genap\,,\qquad k\geq 3
\end{cases}
\end{equation}
and which satisfy the properties stated in the following lemma. 

\begin{lemma}\label{lem:OpH}
For any $k\geq 3$, the operators $\Op_k$ in~\eqref{def:OpH} are positive definite and 
such that, for all $n\in\N$, $\Op_k(\wc_n)\subset\wc_n$. 
\end{lemma}
\begin{proof}
We first consider the case $k=3$. Since 
$\mu-\gensy$ is positive for every $\mu\geq0$, we have 
\begin{equs}
\langle \Op_3\psi,\psi\rangle_{\fock}&=\langle -\genam\big(\mu-\gensy\big)^{-1}\genap\psi,\psi\rangle_{\fock}\\
&= \langle (\mu-\gensy)^{-1}\genap\psi,\genap\psi\rangle_{\fock}=\|(\mu-\gensy)^{-\half}\genap\psi\|_{\fock}^2\geq 0
\end{equs}
while by Lemma~\ref{lem:generator}, $\Op_3(\wc_n)\subset\wc_n$. 
For $k>3$, the result can be proved inductively using the recursive definition of $\Op_k$. 
\end{proof}

Now, upon solving~\eqref{e:System} for $\hf^{N,n}$, starting from the first equation in~\eqref{e:System} and using the definition of $\Op_k$ in~\eqref{def:OpH}, 
we see that $\hf^{N,n}_2$ can be written as\footnote{the other components $\hf^{N,n}_j, j\ne 2$ can also be written down explicitly in terms of $   \hf^{N,n}_2$ and of the operators $\Op_n$, but we do not need their explicit expression.}
\begin{equ}\label{eq:h2}
\hf^{N,n}_2=\big(\big(\mu-\gensy\big)+\Op_n-\genap\big(\mu-\gensy\big)^{-1}\genam\big)^{-1}\nf_\phi\,,
\end{equ}
and consequently, for every $n\in\N$, we have 
\begin{equ}[e:Final]
\langle \nf_\phi,\hf^{N,n}\rangle_{\fock}=\langle \nf_\phi,\big(\big(\mu-\gensy\big)+\Op_n-\genap\big(\mu-\gensy\big)^{-1}\genam\big)^{-1}\nf_\phi\rangle_{\fock_2}\,.
\end{equ}
Therefore, to take advantage of Lemma~\ref{l:Sandwich}, it suffices to derive suitable bounds 
on the operators $\Op_n$ and $\genap(\mu-\gensy)^{-1}\genam$ and the rest of the section is indeed 
devoted to this purpose. 

\subsection{An iterative approach for the operators $\Op_n$'s}\label{sec:iteration}
The advantage of dealing with diagonal (positive) operators, in the sense of Definition \ref{def:diagonal}, is that their inverse is fully explicit 
and easily computable.
The difficulty in getting upper and lower bounds for the operators $\Op_n$ is that, even though they are diagonal with respect to the chaos, they are definitely not in Fourier space.
 %, and, according to the definition~\eqref{def:OpH}, 
%this is what we need to be able to do in order to control 
Hence, it is a priori hard to determine any bound on their inverse,  which is though essential
given that, by~\eqref{def:OpH}, for all $k$, $\Op_k$ 
is {\it defined} in terms of the inverse of $\Op_{k-1}$. 
Therefore, the goal of this section is to show that it is possible to recursively estimate the $\Op_k$'s 
in terms of diagonal operators $\cS_k$ (see Theorem~\ref{thm:Main}). 
Let us begin by giving some preliminary definitions, necessary to rigorously 
introduce the $\cS_k$'s. 
%\giuseppeText{which is what we need in order to to estimate them, given their definition. 
%Our goal in this section is to estimate $\{\Op_n\}_{n\ge3}$ with operators which are diagonal 
%and consequently easily invertible. 
%Let us begin by introducing the operators which will bound them. To do so we need some preliminary 
%defintiions. }
%
%
%Our goal is to estimate the operators $\{\Op_n\}_{n\ge3}$ defined in~\eqref{def:OpH} in terms of 
%simpler operators whose action in Fourier space can be expressed via explicit multipliers. 
%Before stating (and proving) the main results we need to give some preliminary definitions 
%and establish some useful conventions. 
\medskip 

Let $k\in\N$ and $\lambda>0$. 
For $x>0$ and $z\geq 1$, we define the functions $\Ll$, $\LB_{k}$ and $\UB_{k}$ (here $\LB$ stands for lower bound and $\UB$ for upper bound)
on $\R_+\times[1,\infty)$ as follows
\begin{gather}
\Ll(x,z)\eqdef \lambda^2(z+\log(1+x^{-1}))+1\,,\label{e:L}\\
\LB_{k}(x,z)\eqdef \sum_{j\leq k}\frac{(\frac12\log \Ll(x,z))^j}{j!}\,\quad\text{and}\quad\UB_{k}(x,z)\eqdef \frac{\Ll(x,z)}{\LB_{k}(x,z)}\label{e:LUBk}\,.
%\,,\label{e:UBk}
\end{gather}
For $n\in\N$ and any $\delta>0$, set 
\begin{equ}[e:Functions]
z_k(n)\eqdef K_1(\lambda^2\vee 1)(n+k)^{3+2\delta},\,\qquad f_k(n)\eqdef K_2\sqrt{z_k(n)}
\end{equ}
where $K_1,K_2$ are sufficiently large absolute positive constants which will be fixed below (they will be chosen 
in such a way that~\eqref{e:K2} and~\eqref{e:CondonK} hold). As we will need the following fact below, 
note that for any $n,\,k\in\N$, $z_k$ and $f_k$ trivially satisfy 
\begin{equ}[e:+1]
f_k(n+1)=f_{k+1}(n)\qquad\text{ and }\qquad z_k(n+1)=z_{k+1}(n)\,.
\end{equ}
We are now ready to introduce the operators $\cS_k$.

\begin{definition}\label{def:OpS}
Let $\lambda,\mu>0$. 
Let $\cS_2^N$ be the operator which is identically equal to $0$ and, for $k\in\N$, $k\geq 3$ and $N\in\N$, 
define $\cS_k^N$ via 
\begin{equ}[e:SN]
\cS_k^N\eqdef
\begin{cases}
f_k(\cN)\,\sigma_k^N(\mu-\gensy,z_k(\cN))\,, &\text{if $k$ is odd,}\\
\frac{1}{f_k(\cN)}\Big[\sigma_k^N(\mu-\gensy, z_k(\cN))-f_k(\cN)\Big]\,, &\text{if $k$ is even,}
\end{cases}
\end{equ}
where $f_k$ and $z_k$ are given in~\eqref{e:Functions}, $\cN$ is the {\it number operator} - 
the operator such that, for any $g\colon \N\to\R$, the action of $g(\cN)$ in Fock space 
is $g(\cN)\phi_n=g(n) \phi_n$, $\phi_n\in \fock_n$ - and $\gensy$ is given as in~\eqref{e:gens:fock}. 
At last, the Fourier multiplier 
\begin{equ}
\sigma_k^N(x,z)\eqdef \sigma_k(x/N^2, z)
\end{equ}
is such that 
\begin{equ}[e:sigmaN]
\sigma_k(x,z)\eqdef
\begin{cases}
\UB_{\frac{k-3}{2}}(x,z)\,, &\text{if $k$ is odd,}\\
\LB_{\frac{k}{2}-1}(x\vee1,z)\,, &\text{if $k$ is even.}
\end{cases}
\end{equ}
In what follows we will write $\LlN(x,\cdot)$, $\LBN_k(x,\cdot)$, $\UBN_k(x,\cdot)$ for 
$\Ll(x/N^2,\cdot)$, $\LB_k(x/N^2,\cdot)$, $\UB_k(x/N^2,\cdot)$, respectively.
\end{definition}

The following theorem is the main result of this section and establishes the previously 
announced bounds on the operators $\Op_k$. 

\begin{theorem}\label{thm:Main}
Let $\lambda>0$, $\mu>0$, $N\in\N$ and $\{\Op_n\}_{n\geq 2}$ be the family of operators on 
$L^2(\eta)$ recursively defined according to~\eqref{def:OpH}. 
Then, for all $k\in\N$ and any $\delta>0$ there exist constants $c^+_{2k+1}=c^+_{2k+1}(\delta)>1$ 
and $c^-_{2k+2}=c^-_{2k+2}(\delta)<1$ independent of $\mu$ and $N$ such that the following bounds hold
\begin{equs}
\Op_{2k+1}&\leq c^+_{2k+1} \,(-\gensy)\cS_{2k+1}^N\label{e:UltimateUB1}\,,\\
\Op_{2k+2}&\geq c^-_{2k+2} \,(-\gensy)\cS_{2k+2}^N\label{e:UltimateLB1}\,,
\end{equs}
where the operators $\{\cS_m\}_{m\geq 2}$ are given in Definition~\ref{def:OpS}\footnote{For any two operators $\cZ_1$ 
$\cZ_2$ on $\fock$, $\cZ_1\leq\cZ_2$ if and only if for all $\phi\in\fock$, $\langle \cZ_1\phi,\phi\rangle_{\fock}\leq \langle \cZ_2\phi,\phi\rangle_{\fock}$. }.
Further, the constants $c^+_{2k+1}$, $c^-_{2k+2}$ can be chosen as $c_2^-=\tfrac12$, 
\begin{equ}[e:Constants]
c^+_{2k+1}\eqdef \frac{1}{\pi c^-_{2k}}\Big(1+\frac{1}{2 k^{1+\delta}}\Big)  \qquad\text{and}\qquad   c^-_{2k+2}\eqdef \frac{1}{\pi (1+\frac{1}{f_{2k+2}(1)})c_{2k+1}^+}\Big(1-\frac{1}{2 k^{1+\delta}}\Big)
\end{equ}
so that, in particular, the sequences $\{c^+_{2k+1}\}_{k\ge1}$ and $\{c^-_{2k+2}\}_{k\ge1}$ 
tend to two positive and finite limits as $k\to\infty$.
\end{theorem}
\begin{remark}
To obtain a better understanding of the importance of the above result note that since $\lim_{k\to\infty} \LB_k(x,z)=\lim_{k\to\infty}\UB(x,z)= \sqrt{\Ll(x,z)}$ Equations~\eqref{e:UltimateUB1} and~\eqref{e:UltimateLB1} essentially state that for $k$ sufficiently large $\Op_k \approx (-\gens) \sqrt{\log(\mu-\gens)}$, which is very much in the form needed to conclude Theorems~\ref{thm:BD} and~\ref{thm:main2}.
	\end{remark}

The rest of the section is devoted to the proof of the previous statement and its application 
in estimating~\eqref{e:Laplace} (see Proposition~\ref{p:mainB}). 
We will first make some observations and prove some statements that will streamline the subsequent analysis. 

We will show Theorem~\ref{thm:Main} inductively on $k$, which, modulo the initial inductive step, 
reduces to prove that if~\eqref{e:UltimateUB1} 
holds then so does~\eqref{e:UltimateLB1} with $c_{2k+2}^-$ as in~\eqref{e:Constants} and viceversa. 
Note that thanks to the definition of $\{\Op_n\}_{n\geq 2}$ in~\eqref{def:OpH},~\eqref{e:UltimateUB1} implies 
\begin{equ}
\cH_{2k+2}^N \geq -\genam\left(\mu -\gensy\left(1+c_{2k+1}^+\cS_{2k+1}^N\right)\right)^{-1}\genap\,,
\end{equ}
while~\eqref{e:UltimateLB1} implies
\begin{equ}
\Op_{2k+3} \leq -\genam\left(\mu -\gensy\left(1+c_{2k+2}^+\cS_{2k+2}^N\right)\right)^{-1}\genap\,.
\end{equ}
Here we use the fact that for any two positive operators $A,B$,
\begin{equ}
  [e:AByau]
  0<A\le B \quad \text{if and only if} \quad 0<B^{-1}\le A^{-1}.
\end{equ}
Now, the operators at the right hand sides above 
are of the form $-\genam \cZ\genap$ for some diagonal operator $\cZ$. 
Therefore, the quantity we need to bound is 
\begin{equ}
\langle -\genam \cZ\genap\phi,\phi\rangle_{\fock}=\langle  \cZ\genap\phi,\genap\phi\rangle_{\fock}
\end{equ} 
where we used that $-\genam=(\genap)^*$ by Lemma~\ref{lem:generator}. 
Let us derive a decomposition of the latter which will be useful in highlighting the relevant contributions 
to the scalar product. 

\begin{lemma}\label{l:DiagOffDiag}
Let $\cZ$ be a diagonal operator on $\fock$ with Fourier multiplier $\zeta=(\zeta_n)_{n\in\N}$. 
Then, for every $\phi\in\fock_n$, the following decomposition holds
\begin{equ}
\langle \cZ\genap\phi,\genap\phi\rangle_{\fock_{n+1}}=\langle \cZ\genap\phi,\genap\phi\rangle_{\Di} +\sum_{i=1}^2 \langle \cZ\genap\phi,\genap\phi\rangle_{\oD_i}
\end{equ}
where the first summand will be referred to as the ``diagonal part'' and is given by 
\begin{equation}\label{e:Diag}
%\begin{aligned}
\langle \cZ\genap\phi,\genap\phi\rangle_{\Di}\eqdef n!\, n \,2 \lambda^2 \sum_{k_{1:n}}|k_1|^2|\hat\phi(k_{1:n})|^2\sum_{\ell+m=k_1}\zeta_{n+1}(\ell,m,k_{2:n}){(\nonlin_{\ell,m})^2}
%\end{aligned}
\end{equation}
while the other two terms will be referred to as the ``off-diagonal part of type $1$ and $2$'' and are 
respectively given by 
\begin{align}
\langle \cZ\genap\phi,&\genap\phi\rangle_{\oDi}\eqdef n!\,c_{\oDi}(n)\,\lambda^2\times\nonumber\\
&\times\sum_{k_{1:n+1}} \zeta_{n+1}(k_{1:n+1})   |k_1+k_2|\nonlin_{k_1,k_2}|k_1+k_3|\nonlin_{k_1,k_3}\times\label{e:OffDiag1}\\
&\qquad\qquad\times
\hat{\phi}(k_1+k_2,k_3,k_{4:n+1})\overline{\hat{\phi}(k_1+k_3,k_2,k_{4:n+1})}\,,\nonumber\\
\langle \cZ\genap\phi,&\genap\phi\rangle_{\ooDi}\eqdef n!\,c_{\ooDi}(n)\,\lambda^2\times\nonumber\\
&\times \sum_{k_{1:n+1}} \zeta_{n+1}(k_{1:n+1})   |k_1+k_2|\nonlin_{k_1,k_2}|k_3+k_4|\nonlin_{k_3,k_4} \times\label{e:OffDiag2}\\
&\qquad\qquad\times\hat{\phi}(k_1+k_2,k_3,k_4,k_{5:n+1})\overline{\hat{\phi}(k_3+k_4,k_1,k_2,k_{5:n+1})}\nonumber
\end{align}
where, for $i=1,\,2$, $c_{\oD_i}(n)$ is an explicit positive constant only depending on $n$ and such that $c_{\oD_i}(n)= O(n^{i+1})$. 
\end{lemma}
\begin{proof}
As stated in Lemma~\ref{lem:generator}, the right hand side of~\eqref{e:genap:fock} still needs to be symmetrised, 
so that, to be precise, we have 
\begin{equs}[eq:Fourier_asym]
\cF&(\genap\phi)(k_{1:n+1})=\frac{n\lambda}{(n+1)!}\sum_{s\in S_{n+1}}|k_{s(1)}+k_{s(2)}|\nonlin_{k_{s(1)},k_{s(2)}}\hat\phi(k_{s(1)}+k_{s(2)},k_{s(3):s(n+1)})\\
&=\frac{2\lambda}{n+1}\sum_{i<j}|k_i+k_j|\nonlin_{k_{i},k_{j}}\hat\phi(k_{i}+k_{j},k_{\{1:n+1\}\setminus\{i,j\}})=: \frac{2\lambda}{n+1}\sum_{i<j}(\genapc\phi_n)_{i,j}(k_{1:n+1})\,
\end{equs}
where $S_{n+1}$ is the set of permutations of $\{1,\dots,n+1\}$. 
Then, by the definition of $\langle\cdot,\cdot\rangle_{\fock_n}$ in~\eqref{e:ScalarProdFock}, 
\begin{equ}
\langle \cZ\genap\phi,\genap\phi\rangle_{\fock_{n+1}}=n!\frac{4\lambda^2 }{n+1}\sum_{k_{1:n+1}}\zeta_{n+1}(k_{1:n+1})\sum_{\underline{i},\underline{j}\in I}\prod_{\ell=1}^2(A_+^N\phi_n)_{i_\ell,j_\ell}(k_{1:n+1})
\end{equ}
where $\underline{i}=(i_1,i_2),\,\underline{j}=(j_1,j_2)\in\{1,\dots,n+1\}^2$ and $I$ is the set 
$\{(\underline{i},\underline{j})\,:\,i_1<j_1\text{ and }i_2<j_2\}$. We now split $I$ into 
the disjoint union of $I_m$, $m=0,1,2$, containing those $(\underline{i},\underline{j})\in I$ 
such that $|\{i_1,j_1\}\cap\{i_2,j_2\}|=m$. 
By using basic combinatorics (see also~\cite[Section 4.1]{Yau}), it is not hard to see that 
\begin{equ}[e:Cardinality]
|I_2|={n+1\choose 2}\,,\quad |I_1|=2  {n+1\choose 2} (n-1)\,,\quad\text{and}\quad |I_0|=\tfrac{(n+1)n(n-1)(n-2)}{4}\,.
\end{equ}
Then, 
\begin{equ}
\langle \cZ\genap\phi,\genap\phi\rangle_{\fock_{n+1}}=n!\frac{4\lambda^2 }{n+1}\sum_{m=0}^2\sum_{k_{1:n+1}}\zeta_{n+1}(k_{1:n+1})\sum_{\underline{i},\underline{j}\in I_m}\prod_{\ell=1}^2(A_+^N\phi_n)_{i_\ell,j_\ell}(k_{1:n+1})
\end{equ} 
which, by relabelling the variables and using~\eqref{e:Cardinality} to derive the value of the prefactor of~\eqref{e:Diag} 
and of the constants $c_{\oDi}$ and $c_{\ooDi}$, 
gives the decomposition we were after - the diagonal 
term~\eqref{e:Diag} is the summand corresponding to $m=2$, the off-diagonal term of the first type~\eqref{e:OffDiag1} 
that with $m=1$ and the off-diagonal of the second type~\eqref{e:OffDiag2} that with $m=0$. 
\end{proof}

The following lemma provides a condition under which 
we can easily deduce operator bounds for the diagonal term. 

\begin{lemma}\label{l:Diag}
Let $\cZ_1$ and $\cZ_2$ be diagonal operators on 
$\fock$ with Fourier multipliers $\zeta^i=(\zeta^i_n)_n$, $i=1,2$. 
If for every $k_{1:n}\in\Z^{2n}$ 
\begin{equ}[e:BoundSym]
\sum_{\ell+m=k_1}(\nonlin_{\ell,m})^2\zeta^1_{n+1}(\ell,m,k_{2:n})\leq \zeta^2_n(k_{1:n})\,,
\end{equ}
where the sum is over $\ell,\,m\in\Z^2$, then for every $\phi\in\fock_n$
\begin{equ}[e:GenBoundDiagonal]
\langle \cZ_1\genap\phi,\genap\phi\rangle_\Di\leq 4\lambda^2\langle (-\gensy)\cZ_2\phi,\phi\rangle_{\fock_n}\,.
\end{equ}
If instead %$\sigma_2$ is is allowed to be negative and that
~\eqref{e:BoundSym} holds with the opposite inequality,   
then so does~\eqref{e:GenBoundDiagonal}. % with $\geq$ replacing $\leq$.
\end{lemma}
\begin{proof}
By~\eqref{e:Diag}, we have 
\begin{equs}
\langle \cZ_1\genap\phi,&\genap\phi\rangle_{\Di}\eqdef n!\, n \,2 \lambda^2 \sum_{k_{1:n}}|k_1|^2|\hat\phi(k_{1:n})|^2\sum_{\ell+m=k_1}\zeta^1_{n+1}(\ell,m,k_{2:n}){(\nonlin_{\ell,m})^2}\\
&\leq n!\, n \,2 \lambda^2 \sum_{k_{1:n}}|k_1|^2\zeta^2_n(k_{1:n})|\hat\phi(k_{1:n})|^2=n! 4\lambda^2\sum_{k_{1:n}}\tfrac12 |k_{1:n}|^2\zeta^2_n(k_{1:n})|\hat\phi(k_{1:n})|^2
\end{equs}
where the first inequality is a consequence of~\eqref{e:BoundSym}, while in the second equality we 
simply symmetrised the arguments. 
Then,~\eqref{e:GenBoundDiagonal} is an immediate consequence of the definition of the scalar 
product on $\fock_n$ (see~\eqref{e:ScalarProdFock}). 
\end{proof}

In view of the results above, we are ready to state and prove 
the three lemmas which represent the core of the proof of Theorem~\ref{thm:Main}. 
In the first two we respectively show a lower and an upper bound for the diagonal term, while in the last we 
focus on the off-diagonal terms. 

\begin{lemma}\label{l:UBtoLB}
Let $\lambda,\mu>0$, $n,\,k\in\N$ and $c>1$. 
Then, for any $\phi\in\fock_n$
\begin{equ}[e:UBtoLB]
\langle (\mu-\gensy(1+c\cS^N_{2k+1}))^{-1}\genap\phi,\genap\phi\rangle_\Di
\geq \frac{1}{\pi c(1+\frac1{f_{2k+2}(1)}) }\langle (-\gensy)\tilde\cS^N_{2k+2} \phi,\phi\rangle_{\fock_n} 
\end{equ}
where the operator $\tilde\cS^N_{2k+2}$ is defined as 
\begin{equ}[e:tildeSN]
\tilde\cS^N_{2k+2}\eqdef \frac{1}{f_{2k+2}(\cN)}\Big[\sigma^N_{2k+2}(\mu-\gensy, z_{2k+2}(\cN))\Big(1-\frac{4\pi \lambda C_{\Di}}{\sqrt{z_{2k+2}(1)}}\Big)-\frac{f_{2k+2}(\cN)}{2}\Big]
\end{equ}
and the constant $C_\Di>0$ depends just on the constant $K$ of Lemmas~\ref{l:Approx} and~\ref{lemma:RiemannNew}. 
% and $f(z)$ is any function such that 
%\begin{equ}[e:function]
%f(z)\geq 2((\lambda\sqrt{z})\vee1)
%\end{equ}
%for all $z\geq 5$
\end{lemma}
\begin{proof}
Thanks to Lemma~\ref{l:Diag} applied to the operator $\cZ_1\eqdef (\mu-\gensy(1+c\cS^N_{2k+1}))^{-1}$, it is sufficient to focus on 
\begin{equ}
\sum_{\ell+m=k_1}\frac{(\nonlin_{\ell,m})^2}{\mu+\Gamma(\ell,m,k_{2:n})(1+cf_{2k+2}\UB^N_{k-1}(\mu+\Gamma(\ell,m,k_{2:n}),z_{2k+2}))}
\end{equ}
where we recall the definition of $\sigma^N_{2k+1}$ in~\eqref{e:sigmaN}, and 
we set 
\begin{equ}[e:Conv1]
\Gamma(\ell,m,k_{2:n})\eqdef\tfrac12(|\ell|^2+|m|^2+|k_{2:n}|^2)\,,\qquad \ell,m,k_2, \dots,k_n\in\Z^2\,.
\end{equ}
We omitted the dependence on $n$ of $f_{2k+2}$ and $z_{2k+2}$ since $n$ will be fixed throughout, 
and used~\eqref{e:+1} to justify the subscript of $f_{2k+2}$ and $z_{2k+2}$ (note that $\cZ_1$ is applied to 
$\genap\phi$ which is in the $n+1$-th chaos). 
Since $c,\,f_{2k+2}$ and $\UB_{k-1}$ are all bigger than $1$, the previous is 
lower bounded by 
\begin{equ}[e:MainQnty]
\frac{1}{cf_{2k+2}(1+\frac1{f_{2k+2}})}\sum_{\ell+m=k_1}\frac{(\nonlin_{\ell,m})^2}{\mu+\Gamma(\ell,m,k_{2:n})\UB^N_{k-1}(\mu+\Gamma(\ell,m,k_{2:n}),z_{2k+2})}.
\end{equ}
Now, upon choosing $F^N(x,z)=\UB^N_{k-1}(x,z)$ and introducing the short-hand notation 
\begin{equ}[e:alphaN]
\alpha_N\eqdef \mu/N^2+\tfrac12|k_{1:n}/N|^2\,,
\end{equ}
Lemmas~\ref{l:Approx} and~\ref{lemma:RiemannNew}
imply that there exists a constant $K>0$ such that the sum in~\eqref{e:MainQnty} is bounded from below by 
\begin{equation}\label{e:MainQnty1}
\begin{aligned}
&\int_{\R^2}\frac{(\nonlin_{xN,-xN})^2}{(|x|^2+\alpha_N)(1+|x|^2+\alpha_N)\UB_{k-1}(|x|^2+\alpha_N,z_{2k+2})}\dd x-K\frac{\LB_{k-1}(\alpha_N\vee\tfrac{1}{N^2},z_{2k+2})}{\lambda\sqrt{z_{2k+2}(n)}}\\
&\geq\int_{\R^2}\frac{(\nonlin_{xN,-xN})^2}{(|x|^2+\alpha_N)(1+|x|^2+\alpha_N)\UB_{k-1}(|x|^2+\alpha_N,z_{2k+2})}\dd x-K\frac{\LB_{k}(\alpha_N\vee\tfrac{1}{N^2},z_{2k+2})}{\lambda\sqrt{z_{2k+2}(1)}}
\end{aligned}
\end{equation}
where we used the definition of $\UB_{k-1}$ and $\LB_{k-1}$ in~\eqref{e:LUBk}, 
the monotonicity of the latter in $k$ and that $z_{2k+2}(n)\geq z_{2k+2}(1)$. 
To control the first term above, notice that, in polar coordinates, i.e. 
setting $x=r(\cos\theta,\sin\theta)$, the Fourier coefficient of the non-linearity~\eqref{e:nonlinCoefficient} reads
\begin{equ}[e:NonlinPolar]
\nonlin_{xN,-xN}=-\frac{1}{2\pi}\cos(2\theta)\1_{r\in[1/N,1]}
\end{equ}
so that the integral in~\eqref{e:MainQnty1} factorises and we get
\begin{equs}[e:MainStepLB]
\int_{\R^2}&\frac{(\nonlin_{xN,-xN})^2}{(|x|^2+\alpha_N)(1+|x|^2+\alpha_N)\UB_{k-1}(|x|^2+\alpha_N,z_{2k+2})}\dd x\\
&=\int_0^{2\pi}\Big(\frac{\cos(2\theta)}{2\pi}\Big)^2\dd \theta\int_{\tfrac1N}^1\frac{r\dd r}{(r^2+\alpha_N)(1+r^2+\alpha_N)\UB_{k-1}(r^2+\alpha_N,z_{2k+2})}\\
&=\frac{1}{8\pi}\int_{\tfrac1{N^2}+\alpha_N}^{1+\alpha_N}\frac{\dd\rho}{\rho(\rho+1)\UB_{k-1}(\rho,z_{2k+2})}\\
&=\frac{1}{4\pi\lambda^2}\Big[\LB_k(\tfrac1{N^2}+\alpha_N, z_{2k+2})-\LB_k(1+\alpha_N, z_{2k+2})\Big]
\end{equs}
where in the last step we exploited~\eqref{e:IntUBtoLB}. Notice now that, by using the same identity, 
\begin{equs}
\LB_k(\tfrac1{N^2}&+\alpha_N, z_{2k+2})=\LB_k(\tfrac1{N^2}\vee\alpha_N, z_{2k+2})-\frac{\lambda^2}{2}\int_{\tfrac1{N^2}\vee\alpha_N}^{\tfrac1{N^2}+\alpha_N}\frac{\dd\rho}{\rho(\rho+1)\UB_{k-1}(\rho,z_{2k+2})}\\
&=\LB_k(\tfrac1{N^2}\vee\alpha_N, z_{2k+2})-\frac{\lambda^2}{2}\int_{\tfrac1{N^2}\vee\alpha_N}^{\tfrac1{N^2}+\alpha_N}\frac{\LB_{k-1}(\rho,z_{2k+2})\dd\rho}{\rho(\rho+1)\Ll(\rho,z_{2k+2})}\\
&\geq \LB_k(\tfrac1{N^2}\vee\alpha_N, z_{2k+2})\Big(1-\frac{\lambda^2}{2}\int_{\tfrac1{N^2}\vee\alpha_N}^{\tfrac1{N^2}+\alpha_N}\frac{\dd\rho}{\rho\Ll(\rho,z_{2k+2})}\Big)\\
&\geq \LB_k(\tfrac1{N^2}\vee\alpha_N, z_{2k+2})\Big(1-\frac{\lambda^2}{2z_{2k+2}}\int_{\tfrac1{N^2}\vee\alpha_N}^{\tfrac1{N^2}+\alpha_N}\frac{\dd\rho}{\rho}\Big)\\
&\geq  \LB_k(\tfrac1{N^2}\vee\alpha_N, z_{2k+2})\Big(1-\frac{\lambda^2}{2z_{2k+2}}\frac{\frac1{N^2}+\alpha_N-\tfrac1{N^2}\vee\alpha_N}{\tfrac1{N^2}\vee\alpha_N}\Big)\\
&\geq \LB_k(\tfrac1{N^2}\vee\alpha_N, z_{2k+2})\Big(1-\frac{\lambda^2}{\sqrt{z_{2k+2}(1)}}\Big)
\end{equs}
where in the first step we exploited the definition of $\UB_{k}$ in~\eqref{e:LUBk}, in the second 
the fact that $\LB_{k-1}(\rho,z)$ is decreasing in $\rho$ and increasing in $k$ by Lemma~\ref{l:MainIntegrals} 
and in the last that the fraction involving $\alpha_N$ is bounded above by $2$ and $z_{2k_2}(n)$ is increasing 
in $n$ and greater or equal to $1$. 
Moreover, 
\begin{equs}[e:K2]
\LB_k(1+\alpha_N, z_{2k+2})&\leq \LB_k(1, z_{2k+2}) \leq\sqrt{\Ll(1,z_{2k+2})}\\
&=\sqrt{\lambda^2(z_{2k+2}+\log 2)+1}\leq \tfrac12 f_{2k+2}
\end{equs}
which in turn is a consequence of the fact that $\LB_k$ is decreasing 
in the first variable (see Lemma~\ref{l:MainIntegrals}),~\eqref{e:LBExp} and 
a choice of a sufficiently large $K_2$ in~\eqref{e:Functions}. 

As a consequence,~\eqref{e:MainStepLB} is lower bounded by 
\begin{equ}
\frac{1}{4\pi\lambda^2}\Big[\LB_k(\tfrac1{N^2}\vee\alpha_N, z_{2k+2})- \frac{f_{2k+2}}{2}\Big]
\end{equ}
so that, in conclusion, there exists a $C_{\Di}$ so that~\eqref{e:MainQnty} is lower bounded by 
\begin{equs}
\frac{1}{4\lambda^2\pi c(1+\frac1{f_{2k+2}(1)})}\frac{1}{f_{2k+2}(n)}\Big[\LB_k(\tfrac1{N^2}\vee\alpha_N, z_{2k+2}(n))\Big(1-\frac{4\pi \lambda C_{\Di}}{\sqrt{z_{2k+2}(1)}}\Big)-\frac{f_{2k+2}(n)}{2}\Big]\,,
\end{equs}
where we additionally used that $n\mapsto f_{2k+2}(n)$ is increasing to replace $f_{2k+2}(n)$ by $f_{2k+2}(1)$.
Applying Lemma~\ref{l:Diag} is sufficient to conclude the proof. 
\end{proof}

\begin{lemma}\label{l:LBtoUB}
Let $\lambda,\mu>0$, $n,\,k\in\N$ and $c<1$. 
Then, for any $\phi\in\fock_n$
\begin{equs}[e:LBtoUB]
\langle (\mu-\gensy(1+&c\cS^N_{2k+2}))^{-1}\genap\phi_n,\genap\phi_n\rangle_\Di\\
&\leq  \frac{1}{\pi c}\Big(1+\frac{4\lambda\pi C_{\Di}}{\sqrt{z_{2k+2}(1)}}\Big)\langle (-\gensy)\cS^N_{2k+3}\phi,\phi\rangle_{\fock_n}\,.
\end{equs}
\end{lemma}
\begin{proof}
Applying Lemma~\ref{l:Diag} to the operator $\cZ_1\eqdef (\mu-\gensy(1+c\cS^N_{2k+2}))^{-1}$, 
we see that we can focus on 
\begin{equ}[e:MainQntyUB]
\sum_{\ell+m=k_1}\frac{(\nonlin_{\ell,m})^2}{\mu+\Gamma(\ell,m,k_{2:n})(1+\frac{c}{f_{2k+3}}[\LB^N_{k}((\mu+\Gamma(\ell,m,k_{2:n}))\vee 1,z_{2k+3})-f_{2k+3}])}
\end{equ}
where, as in the proof of Lemma~\ref{l:UBtoLB}, we recalled the definition of $\sigma^N_{2k+1}$ in~\eqref{e:sigmaN}, 
defined $\Gamma$ as in~\eqref{e:Conv1}, 
omitted the dependence on $n$ of $f_{2k+2}$ and $z_{2k+2}$ since $n$ will be fixed throughout, 
and used~\eqref{e:+1}. 
Note that by assumption $1-c>0$ so that 
the sum above can be upper bounded by
\begin{equ}[e:MainQntyUBLB]
\frac{f_{2k+3}}{c}\sum_{\ell+m=k_1}\frac{(\nonlin_{\ell,m})^2}{\mu+\Gamma(\ell,m,k_{2:n})\LB^N_{k}(\mu+\Gamma(\ell,m,k_{2:n}),z_{2k+3})}
\end{equ}
where we removed the $\vee$ at the denominator of~\eqref{e:MainQntyUB} as, 
by the definition of $\nonlin_{\ell,m}$ in~\eqref{e:nonlinCoefficient} and 
of $\indN{\ell,m}$ in~\eqref{eq:JN}, both $\ell$ and $m$ are such that $|\ell|,|m|\geq 1$ thus ensuring 
$\mu+\Gamma(\ell,m,k_{2:n})\geq 1$. 
Now, upon choosing $F^N(x,z)=\LB^N_{k}(x,z)$ and defining $\alpha_N$ according to~\eqref{e:alphaN}, 
Lemmas~\ref{l:Approx} and~\ref{lemma:RiemannNew}, 
imply that the sum in~\eqref{e:MainQntyUBLB} is bounded from above by 
\begin{equation}\label{e:MainQnty11}
\int_{\R^2}\frac{(\nonlin_{xN,-xN})^2}{(|x|^2+\alpha_N)(1+|x|^2+\alpha_N)\LB_{k}(|x|^2+\alpha_N,z_{2k+3})}\dd x+C_{\Di}\frac{\UB_{k}(\alpha_N,z_{2k+3})}{\lambda\sqrt{z_{2k+3}(1)}}
\end{equation}
where we used the definition of $\LB_{k}$ and $\UB_{k}$ in~\eqref{e:LUBk} and again that $z_{2k+3}(n)\geq z_{2k+3}(1)$. 
To control the integral, we proceed as in the proof of Lemma~\ref{l:UBtoLB} - we pass to 
polar coordinates and use~\eqref{e:NonlinPolar}, so that we obtain 
\begin{equs}
\int_{\R^2}&\frac{(\nonlin_{xN,-xN})^2}{(|x|^2+\alpha_N)(1+|x|^2+\alpha_N)\LB_{k}(|x|^2+\alpha_N,z_{2k+3})}\dd x\\
&=\int_0^{2\pi}\Big(\frac{\cos(2\theta)}{2\pi}\Big)^2\dd \theta\int_{\tfrac1N}^1\frac{r\dd r}{(r^2+\alpha_N)(1+r^2+\alpha_N)\LB_{k}(r^2+\alpha_N,z_{2k+3})}\\
&=\frac{1}{8\pi}\int_{\tfrac1{N^2}+\alpha_N}^{1+\alpha_N}\frac{\dd\rho}{\rho(\rho+1)\LB_{k}(\rho,z_{2k+3})}\\
&\leq \frac{1}{4\pi\lambda^2}\Big[\UB_k(\tfrac1{N^2}+\alpha_N,z_{2k+3})-\UB_k(1+\alpha_N, z_{2k+3})\Big]
\leq \frac{1}{4\pi\lambda^2}\UB_k(\alpha_N,z_{2k+3})
\end{equs}
where in the second to last step we exploited~\eqref{e:IntLBtoUB} and in the last the monotonicity of $\UB$ 
as stated in Lemma~\ref{l:MainIntegrals}. 

Summarising what done so far, we have showed that~\eqref{e:MainQntyUBLB} is bounded above by 
\begin{equ}
\frac{1}{4\lambda^2 \pi c}\Big(1+C_{\Di}\frac{4\lambda\pi}{\sqrt{z_{2k+2}(1)}}\Big)f_{2k+3}(n)\UB_k(\alpha_N,z_{2k+3}(n))
\end{equ}
so that~\eqref{e:LBtoUB} follows at once by Lemma~\ref{l:Diag} and the definition of $\cS_{2k+3}$ in~\eqref{e:SN}. 
\end{proof}

\begin{lemma}\label{l:OffDiagBound}
For $m\in\N$, $m\geq 2$, let $c_m$ be a constant such that if $m$ is odd then $c_m>1$, while if $m$ is even $c_m<1$. 
Then, there exists a constant $C_\mathrm{off}>1$ independent of $\mu,\,\lambda,\,k$ and $N$ 
such that for all $m\in\N$, $m> 2$,  and any $\phi\in\fock_n$
\begin{equ}[e:OffDiagBound]
\sum_{i=1,2}\Big|\langle (\mu-\gensy(1+c_m \cS^N_{m}))^{-1}\genap\phi,\genap\phi\rangle_{\mathrm{off}_i}\Big|
%+\Big|\langle& \genam (\mu-\gens+\cH_{m}^N)^{-1}\genap\phi_n,\phi_n\rangle_\oDi\Big|\\
\leq \lambda^2C_\mathrm{off}  \langle(-\gensy)\cS^{N,\oD}_{m+1}\phi,\phi\rangle_{\fock_n} 
\end{equ}
where $\cS^{N,\oD}_{m+1}$ is the operator defined by 
\begin{equ}[e:OffSN]
\cS^{N,\oD}_{m+1}\eqdef \frac{\cN^2}{\tilde c_m z_{m+1}(\cN)}\sigma_{m+1}(\mu-\gensy , z_{m+1}(\cN))
\end{equ}
and $\tilde c_m\eqdef c_m f_{m+1}(\cN)$ if $m$ is odd, while $\tilde c_m\eqdef c_m/f_{m+1}(\cN)$ if $m$ is even. 
If $m=2$,~\eqref{e:OffDiagBound} still holds but with $\cS^{N,\oD}_{m+1}$ replaced by $\cN^2$. 
\end{lemma}

\begin{proof}
%Throughout the proof the constant $C>0$ will change from line to line and is independent of $\lambda,\,n,\,m$ and $N$. 
%
We begin with the off-diagonal term of the first type, which, by~\eqref{e:OffDiag1} and~\eqref{e:nonlinCoefficient}, is 
\begin{equs}[e:OffDiagFirst]
&n!\frac{\,c_{\oDi}(n)\,\lambda^2}{4\pi^2}\sum_{j_{1:3},k_{3:n}}\frac{c(j_1,j_2)}{|j_1||j_2|}\frac{c(j_1,j_3)}{|j_1||j_3|} \hat\varphi(j_1+j_2,j_3,k_{3:n})\hat\varphi(j_1+j_3,j_2,k_{3:n})\\
&\qquad\qquad\qquad\times \frac{|j_1+j_2||j_1+j_3| \indN{j_1,j_2}\indN{j_2,j_3}}{\mu +\Gamma(j_{1:3},k_{3:n})(1+\tilde c_{m}[\sigma^N_{m}(\mu+\Gamma(j_{1:3},k_{3:n}),z_{m+1})-a_m])}
\end{equs}
where $a_m$ is $0$ if $m$ is odd and $f_{m+1}$ otherwise, and,
as in the proof of Lemmas~\ref{l:UBtoLB} and~\ref{l:LBtoUB}, $\sigma^N_{m}$ is as in~\eqref{e:sigmaN}, 
we adopted the same convention for $\Gamma$ as in~\eqref{e:Conv1}, 
omitted the dependence on $n$ of $z_{m+1}$ since $n$ will be fixed throughout, 
and used~\eqref{e:+1}. 
In the rest of the proof we will omit the subscript of $z$ 
since it will not change. 

In order to control the absolute value of the previous, we first bound
the factors $c(j_1,j_2)/(|j_1||j_2|)$, $c(j_1,j_3)/(|j_1||j_3|)$ by
$1$ in absolute value, neglect $1-\tilde c_m a_m>0$ inside the
parenthesis in the denominator (recall that for $m$ even $c_m<1$) and
denote the product of the indicators by $\indN{j_{1:3}}$.  Further,
define
$\Phi(\ell_{1:n})\eqdef\hat\varphi(\ell_{1:n})\prod_{i=1}^n|\ell_i|$,
so that the sum in~\eqref{e:OffDiagFirst} can be upper bounded by
\begin{equs}
&\sum_{j_{1:3},k_{3:n}}\frac{|\Phi(j_1+j_2,j_3,k_{3:n})||\Phi(j_1+j_3,j_2,k_{3:n})| \indN{j_{1:3}}}{|j_2||j_3|(\mu+\tilde c_m\Gamma(j_{1:3},k_{3:n})\sigma^N_{m}(\mu+\Gamma(j_{1:3},k_{3:n}),z))\prod_{i=3}^n|k_i|^2}\\
&\leq \sum_{j_{1:3},k_{3:n}}\frac{|\Phi(j_1+j_2,j_3,k_{3:n})|^2}{\prod_{i=3}^n|k_i|^2}\frac{\indN{j_{1:2}}}{|j_2||j_3|(\mu+\tilde c_m\Gamma(j_{1:3},k_{3:n}))\sigma^N_{m}(\mu+\Gamma(j_{1:3},k_{3:n}),z))}\\
&=\sum_{k_{1:n}}|\hat\varphi(k_{1:n})|^2|k_1|^2|k_2|\sum_{j_1+j_2=k_1}\frac{\indN{j_{1:2}}}{|j_2|(\mu+\tilde c_m\Gamma(j_{1:2},k_{2:n})\sigma^N_{m}(\mu+\Gamma(j_{1:2},k_{2:n})),z))}
\end{equs}
where in the first passage we estimated the product of the $\Phi_n$'s by half the sum of their squares 
and in the second we renamed the variables ($j_3=k_2$). 
Let us point out that in case $m=2$, in all the inequalities above 
there is no $\sigma_m^N$ in the denominator. 

Concerning the inner sum, %let us denote by $\alpha_N\eqdef |k_{1:n}|^2/N^2$ and $\mu_N\eqdef\mu/N^2$. 
by~\eqref{e:lm}, $\tfrac18(|j_2|^2+|k_{1:n}|^2)\leq \Gamma(j_{1:2},k_{2:n})\leq 2(|j_2|^2+|k_{1:n}|^2)$, 
and, since $\sigma_m$ is monotonically decreasing in the first variable, we have 
\begin{equation}\label{e:InnerSum}
\begin{aligned}
&\sum_{j_1+j_2=k_1}\frac{\indN{j_{1:2}}}{|j_2|(\mu+\tilde c_m\Gamma(j_{1:2},k_{2:n})\sigma^N_{m}(\mu+\Gamma(j_{1:2},k_{2:n})),z))}\\
&\leq \frac{1}{N}\sum_{|j_2/N|\leq 1}\frac{1}{N^2}\frac{1}{|j_2/N|} \frac{1}{(\mu_N+\frac{\tilde c_m}{8}(|j_2/N|^2+|k_{1:n}/N|^2)\sigma^N_{m}(2|j_{2}/N|^2+4\alpha_N,z))}\\
&\lesssim \frac{1}{N}  \int_0^1 \frac{\dd\rho}{\mu_N+\frac{\tilde c_m}{8}(\rho^2+|k_{1:n}/N|^2)\sigma_{m}(2\rho^2+4\alpha_N,z)}
\end{aligned}
\end{equation}
where the last bound follows by Riemann sum approximation 
and the use of polar coordinates. 

Now, if $m=2$, modulo constants,~\eqref{e:InnerSum} is bounded above by 
\begin{equ}
\frac{1}{N}\int_0^1\frac{\dd\rho}{\rho^2+\alpha_N}\lesssim\frac{1}{N}\frac{1}{\sqrt{\alpha_N}} \leq\frac{1}{\sqrt{|k_{1:n}|^2}}\,.
\end{equ}

For $m>2$, notice that, if $\mu_N\leq |k_{1:n}/N|^2$, then the denominator in the integral can be trivially 
bounded from below by
\begin{equs}[e:BoundDenom]
\mu_N+\frac{\tilde c_m}{8}(\rho^2+|k_{1:n}/N|^2)\sigma_{m}\geq\frac{\tilde c_m}{8}(\rho^2+|k_{1:n}/N|^2)\sigma_{m} \gtrsim \tilde c_m(2\rho^2 + 4\alpha_N)\sigma_{m}
\end{equs}
where we omitted the arguments of $\sigma_m$ since they do not change. Thus, modulo constants,
~\eqref{e:InnerSum} can be bounded from above by 
\begin{equation}\label{e:BoundInnerSum1}
\begin{aligned}
\frac{1 }{\tilde c_mN}\int_0^1& \frac{\dd\rho}{(2\rho^2+4\alpha_N)\sigma_{m}(2\rho^2+4\alpha_N,z)}\leq \frac{1 }{\tilde c_m N} \int_0^{\infty}  \frac{\dd\rho}{(\rho^2+4\alpha_N)\sigma_{m}(\rho^2+4\alpha_N,z)}\\
&\lesssim \frac{1}{\tilde c_mN}\frac{1}{\sqrt{\alpha_N}}\frac{1}{\sigma_m(8 \alpha_N,z)}\lesssim\frac{1}{\tilde c_m}\frac{1}{\sqrt{ |k_{1:n}|^2}}\frac{1}{\sigma^N_m(8(\mu+\half |k_{1:n}|^2),z)}
\end{aligned}
\end{equation}
where we applied Lemma~\ref{l:OffDiag}, which holds since $\sigma_m$ satisfies its assumptions. 

If instead $\mu_N> \alpha_N$, then 
we split the integral in~\eqref{e:InnerSum} according to $\rho>\sqrt{\mu_N}$ and $\rho\leq\sqrt{\mu_N}$. 
In the former case, we first bound the denominator as in~\eqref{e:BoundDenom} and then extend the integral to the 
interval $[0,1]$ so that we can exploit~\eqref{e:BoundInnerSum1}. For the latter, we exploit the monotonicity 
of $\sigma_m$ which ensures that 
\begin{equ}[e:BoundDenom2]
\sigma_{m}(4(\mu_N+\tfrac12(\rho^2+|k_{1:n}/N|^2)),z)\geq\sigma_{m}(8\alpha_N,z)\,.
\end{equ}
Hence,
\begin{equs}
\frac{1}{N}  \int_0^{\sqrt{\mu_N}} &\frac{\dd\rho}{\mu_N+\frac{\tilde c_m}{8}(\rho^2+|k_{1:n}/N|^2)\sigma_{m}(4(\mu_N+\tfrac12(\rho^2+|k_{1:n}/N|^2)),z)}\\
 &\lesssim \frac{1}{\tilde c_mN}\frac{1}{\sigma_m(8\alpha_N,z)} \int_0^{\infty}\frac{\dd\rho}{\rho^2+|k_{1:n}/N|^2}
 %&\lesssim \frac{1}{\tilde c_mN}\frac{1}{|k_{1:n}/N|}\frac{1}{\sigma_m(8\alpha_N,z)}=
 \lesssim\frac{1}{\tilde c_m}\frac{1}{\sqrt{|k_{1:n}|^2}}\frac{1}{\sigma_m(8\alpha_N,z)}
\end{equs}
where in the first passage we neglected $\mu_N$ and extended the integral to the positive real line.
Overall, we have shown that there exists a constant $C>0$ such that, for any $m\geq 2$,
~\eqref{e:OffDiagFirst} is bounded above by
\begin{equs}[b:OffDiag1]
\Big|\langle &(\mu-\gensy(1+c_m \cS^N_{m}))^{-1}\genap\phi,\genap\phi\rangle_{\mathrm{off}_i}\Big|\\
&\leq n!\frac{\,c_{\oDi}(n)\,\lambda^2}{4\pi^2}\frac{C}{\tilde c_m}\sum_{k_{1:n}}|\hat\varphi(k_{1:n})|^2|k_1|^2\frac{|k_2|}{\sqrt{|k_{1:n}|^2}}\frac{1}{\sigma^N_m(8(\mu+\half |k_{1:n}|^2),z)}\\
&\qquad\qquad \leq n!  C\sum_{k_{1:n}} |\hat\varphi(k_{1:n})|^2|k_{1:n}|^2\frac{\lambda^2 n}{\tilde c_m \sigma^N_m(8(\mu+\half |k_{1:n}|^2),z)}\,,
\end{equs}
and the $n$ at the numerator follows by $c_{\oDi}(n)=O(n^2)$ (see Lemma~\ref{l:DiagOffDiag}) and by replacing $|k_1|^2$ by $|k_{1:n}|^2$.
\medskip

For the off-diagonal term of second type, we proceed similarly. 
Adopting the same notations as in~\eqref{e:OffDiagFirst},~\eqref{e:OffDiag2} equals
\begin{equs}\label{e:OffDiagSecond}
\frac{n!\,c_{\ooDi}(n)\,\lambda^2}{4\pi^2}&\sum_{j_{1:4},k_{4:n}}\frac{c(j_1,j_2)}{|j_1||j_2|}\frac{c(j_3,j_4)}{|j_3||j_4|} \hat\varphi(j_1+j_2,j_{3:4},k_{4:n})\hat\varphi(j_{1:2},j_3+j_4,k_{4:n})\\
&\qquad\qquad\times \frac{|j_1+j_2||j_3+j_4| \indN{j_1,j_2}\indN{j_3,j_4}}{\mu +\Gamma(j_{1:4},k_{4:n})(1+\tilde c_{m}[\sigma^N_{m}(\mu+\Gamma(j_{1:4},k_{4:n}),z)-a_m])}\,.
\end{equs}
By retracing the same steps as in the proof of the bound on the off-diagonal term of the first type, 
we see that the sum above is upper bounded by
\begin{equs}
&\sum_{k_{1:n}}|\hat\varphi(k_{1:n})|^2|k_1|^2|k_2||k_3|\sum_{j_1+j_2=k_1}\frac{\indN{j_1,j_2}}{|j_1||j_2|(\mu+\tilde c_m\Gamma(j_{1:2},k_{2:n})\sigma^N_{m}(\mu+\Gamma(j_{1:2},k_{2:n}),z))}\\
&\leq 2\sum_{k_{1:n}}|\hat\varphi(k_{1:n})|^2|k_1||k_2||k_3|\sum_{j_1+j_2=k_1}\frac{\indN{j_1,j_2}}{|j_2|(\mu+\tilde c_m\Gamma(j_{1:2},k_{2:n})\sigma^N_{m}(\mu+\Gamma(j_{1:2},k_{2:n}),z))}
\end{equs}
where we used that, since $j_1+j_2=k_1$, 
the modulus of at least one of the two must be bigger than $|k_1|/2$. Hence, we can proceed as in~\eqref{e:InnerSum} 
so that, again, modulo constants, the previous is bounded above by 
\begin{equs}[b:OffDiag2]
\frac{1}{\tilde c_m}& \sum_{k_{1:n}}|\hat\varphi(k_{1:n})|^2\frac{|k_1||k_2||k_3|}{\sqrt{|k_{1:n}|^2}}\frac{1}{\sigma^N_m(8(\mu+\half |k_{1:n}|^2),z)}\\
&\leq \frac{1}{\tilde c_m}\frac{1}{n} \sum_{k_{1:n}}|\hat\varphi(k_{1:n})|^2|k_{1:n}|^2\frac{1}{\sigma^N_m(8(\mu+\half |k_{1:n}|^2),z)}\,.
\end{equs}
Hence, it follows that there exists a constant $C>0$ such that 
\begin{equs}[e:OffDiag2Final]
\Big|\langle &\genam (\mu-\gensy(1+\tilde c_m\cS_m^N))^{-1}\genap\phi,\phi\rangle_\ooDi\Big|\\
&\leq \frac{n!\,c_{\ooDi}(n)\,\lambda^2}{4\pi^2n} \frac{C}{\tilde c_m}\sum_{k_{1:n}} |\hat\varphi(k_{1:n})|^2|k_{1:n}|^2\frac{1}{\sigma^N_m(8(\mu+\half|k_{1:n}|^2),z)}\\
&\leq n!\, C \sum_{k_{1:n}} |\hat\varphi(k_{1:n})|^2|k_{1:n}|^2\frac{\lambda^2 n^2}{\tilde c_m \sigma^N_m(8(\mu+\half |k_{1:n}|^2),z)}
\end{equs}
and, again, the $n^2$ comes from $c_{\ooDi}(n)=O(n^3)$ (see Lemma~\ref{l:DiagOffDiag}). 
\medskip

By~\eqref{b:OffDiag1} and~\eqref{e:OffDiag2Final}, we obtained
\begin{equs}[e:OffDiagBound]
\sum_{i=1,2}\Big|\langle (\mu-\gensy&(1+c_m \cS^N_{m}))^{-1}\genap\phi,\genap\phi\rangle_{\mathrm{off}_i}\Big|\\
%+\Big|\langle& \genam (\mu-\gens+\cH_{m}^N)^{-1}\genap\phi_n,\phi_n\rangle_\oDi\Big|\\
&\leq n!\, C \sum_{k_{1:n}} |\hat\varphi(k_{1:n})|^2|k_{1:n}|^2\frac{\lambda^2 n^2}{\tilde c_m \sigma^N_m(8(\mu+\half |k_{1:n}|^2),z_{m+1}(n))}
\end{equs}
for some constant $C>0$ independent of $m, N, \mu$. 
Now, to conclude the proof it suffices to control the fraction inside the sum. 
If $m$ is even by~\eqref{e:sigmaN}, we have 
\begin{equs}
&\frac{\lambda^2 n^2}{\tilde c_m\sigma^N_m(8(\mu+\half |k_{1:n}|^2),z_{m+1}(n))}=\frac{\lambda^2 n^2}{\tilde c_m\LBN_{\frac{m}{2}-1}(8(\mu+\half |k_{1:n}|^2)\vee 1,z_{m+1}(n))}\\
&=\frac{n^2\lambda^2}{\tilde c_m\LlN(8(\mu+\frac{1}{2}|k_{1:n}|^2)\vee 1,z_{m+1}(n))}\UBN_{\frac{m}{2}-1}(8(\mu+\tfrac{1}{2}|k_{1:n}|^2)\vee 1,z_{m+1}(n))\\
&\leq \frac{n^2}{\tilde c_m z_{m+1}(n)}\UBN_{\frac{m}{2}-1}(\mu+\tfrac{1}{2}|k_{1:n}|^2,z_{m+1}(n))=\frac{n^2}{\tilde c_m z_{m+1}(n)}\sigma_{m+1}(\mu+\tfrac{1}{2}|k_{1:n}|^2,z_{m+1}(n))\,,
\end{equs}
while il $m$ is odd 
\begin{equs}
&\frac{\lambda^2 n^2}{\tilde c_m\sigma^N_m(8(\mu+\half |k_{1:n}|^2),z_{m+1}(n))}=\frac{\lambda^2 n^2}{\tilde c_m\UBN_{\frac{m-3}{2}}(8(\mu+\half |k_{1:n}|^2),z_{m+1}(n))}\\
&\leq \frac{\lambda^2 n^2}{\tilde c_m\UBN_{\frac{m-3}{2}}(8(\mu+\half |k_{1:n}|^2)\vee 1,z_{m+1}(n))}\\
&=\frac{n^2\lambda^2}{\tilde c_m\Ll^N(8((\mu+\frac{1}{2}|k_{1:n}|^2)\vee 1),z_{m+1}(n))}\LBN_{\frac{m-3}{2}}(8((\mu+\tfrac{1}{2}|k_{1:n}|^2)\vee 1),z_{m+1}(n))\\
&\leq \frac{n^2 \LBN_{\frac{m-1}{2}}((\mu+\tfrac{1}{2}|k_{1:n}|^2)\vee 1,z_{m+1}(n))}{\tilde c_m z_{m+1}(n)}=\frac{n^2}{\tilde c_m z_{m+1}(n)}\sigma_{m+1}(\mu+\tfrac{1}{2}|k_{1:n}|^2,z_{m+1}(n))
\end{equs}
where, in both cases, we exploited the definition of $\UBN,\LBN,\LlN$ in~\eqref{e:LUBk},~\eqref{e:L} and their monotonicity 
properties in Lemma~\ref{l:MainIntegrals}. 
\end{proof}

We have now all the ingredients we need for the proof of Theorem~\ref{thm:Main}. 

\begin{proof}[of Theorem~\ref{thm:Main}]
As mentioned above, the proof goes by induction. Since by definition~\eqref{def:OpH} $\Op_2\equiv 0\equiv \cS^N_2$, 
~\eqref{e:UltimateLB1} clearly holds for $k=0$ with an arbitrary $c_2^-$ that we can pick to be equal to $\tfrac12$. 
Now, we show that if~\eqref{e:UltimateUB1} holds for $2k+1$, then~\eqref{e:UltimateLB1} holds for $2k+2$ 
with $c_{2k+2}^-$ as in~\eqref{e:Constants}.  
Using the definition of $\Op_{2k+2}$ in~\eqref{def:OpH}, the inductive hypothesis~\eqref{e:UltimateUB1} 
and~\eqref{e:AByau}, we have 
\begin{equ}
\Op_{2k+2}\geq -\genam(\mu-\gensy(1+c^+_{2k+1}\cS^N_{2k+1}))^{-1}\genap\,.
\end{equ}
Let $\phi=(\phi_n)_{n\in\N}$. Thanks to the decomposition of Lemma~\ref{l:DiagOffDiag}, we see that 
for all $n\in\N$ we have 
\begin{equs}[e:SemiFinalUBLB]
\langle \Op_{2k+2} \phi_n,\phi_n\rangle_{\fock_n}&\geq 
\langle \big(\mu -\gensy\left(1+c^+_{2k+1}\cS_{2k+1}^N\right)\big)^{-1}\genap \phi_n,\genap\phi_n\rangle_{\Gamma L^2_{n+1}}\\
&\geq \langle \left(\mu -\gensy\left(1+c^+_{2k+1}\cS_{2k+1}^N\right)\right)^{-1}\genap \phi_n,\genap\phi_n\rangle_{\Di}\\
&\qquad -\sum_{i=1,2}\Big|\langle  (\mu-\gensy(1+c^+_{2k+1} \cS^N_{2k+1}))^{-1}\genap\phi_n,\genap\phi_n\rangle_{\mathrm{off}_i}\Big|\,.
\end{equs}
We are now in a position to apply Lemma~\ref{l:UBtoLB} for the diagonal term and Lemma~\eqref{l:OffDiagBound} 
to the off-diagonal, which together give the lower bound for~\eqref{e:SemiFinalUBLB} 
\begin{equs}
\frac{1}{\pi c^+_{2k+1}(1+\frac{1}{f_{2k+2}(1)})}\Big\langle (-\gensy)\Big[\tilde\cS^N_{2k+2}-\lambda^2 C_{\oD}\pi c^+_{2k+1}(1+\frac{1}{f_{2k+2}(1)}) \cS^{N,\oD}_{2k+2} \Big] \phi_n,\phi_n\Big\rangle_{\fock_n}
\end{equs}
where $\tilde\cS^N_{2k+2}$ and $\cS^{N,\oD}_{2k+2}$ are respectively defined in~\eqref{e:tildeSN} and~\eqref{e:OffSN}. 
It remains to look at the difference of the operators in brackets, which is 
\begin{equs}
&\tilde\cS^N_{2k+2}-\lambda^2 C_{\oD}\pi c^+_{2k+1}(1+\frac{1}{f_{2k+2}(1)}) \cS^{N,\oD}_{2k+2} \\
&=\frac{1}{f_{2k+2}(\cN)}\Big[\sigma^N_{2k+2}(\mu-\gensy, z_{2k+2}(\cN))\Big(1-\frac{4\pi \lambda C_{\Di}}{\sqrt{z_{2k+2}(1)}}-\frac{\lambda^2 C_{\oD}\pi (1+\frac{1}{f_{2k+2}(1)})\cN^2}{z_{2k+2}(\cN)}\Big)\\
&\qquad\qquad\qquad\qquad\qquad\qquad\qquad\qquad\qquad\qquad\qquad\qquad\qquad\qquad-\frac{f_{2k+2}(\cN)}{2}\Big]
\end{equs}
We now choose $K_1$ in~\eqref{e:Functions} big enough so that 
\begin{equs}[e:CondonK]
&\frac{4\pi \lambda C_{\Di}}{\sqrt{z_{2k+2}(1)}}+\frac{\lambda^2\pi(1+\frac{1}{f_{2k+2}(1)}) C_{\oD} \cN^2}{z_{2k+2}(\cN)}\\
&\qquad=\frac{4\pi \lambda C_{\Di}}{\sqrt{K_1(\lambda^2\vee 1)}(2k+3)^{\frac32 +\delta}}+\frac{\pi(1+\frac{1}{f_{2k+2}(1)}) C_{\oD} \cN^2}{K_1(\lambda^2\vee 1)(\cN+2k+2)^{3+2\delta}}\\
&\qquad\leq \frac{4\pi C_{\Di}}{\sqrt{K_1}}k^{-\frac32 -\delta}+\frac{2\pi C_{\oD}}{K_1}k^{-1-2\delta}\leq \frac{1}{2 k^{1+\delta}}\,.
\end{equs} 
where we recalled that the operator $\cN$ takes values in $\N$, so that in particular it is positive.  
Since the right hand side of~\eqref{e:CondonK} is smaller than 
$1/2$, by defining the constant $c^-_{2k+2}$ according to~\eqref{e:Constants} 
it is immediate to see that~\eqref{e:SemiFinalUBLB} is lower bounded by 
\begin{equs}
&\frac{1}{\pi c^+_{2k+1}(1+\frac{1}{f_{2k+2}(1)})}\Big\langle (-\gensy)\frac{1}{f_{2k+2}(\cN)}\Big[\sigma^N_{2k+2}(\mu-\gensy, z_{2k+2}(\cN))\Big(1-\frac{1}{2k^{1+\delta}}\Big)\\
&\qquad\qquad\qquad\qquad\qquad\qquad\qquad\qquad\qquad\qquad\qquad\qquad-\frac{f_{2k+2}(\cN)}{2}\Big] \phi_n,\phi_n\Big\rangle_{\fock_n}\\
&\geq c^-_{2k+2}\langle (-\gensy)\frac{1}{f_{2k+2}(\cN)}\Big[\sigma^N_{2k+2}(\mu-\gensy, z_{2k+2}(\cN))-f_{2k+2}(\cN)\Big] \phi_n,\phi_n\Big\rangle_{\fock_n}\\
&=c^-_{2k+2}\langle (-\gensy)\cS_{2k+2} \phi_n,\phi_n\Big\rangle_{\fock_n}\,,
\end{equs}
which is what we wanted to show. 

We now assume~\eqref{e:UltimateLB1} and prove that~\eqref{e:UltimateUB1} holds for $2k+3$. 
Exploiting~\eqref{e:AByau}, the above and the decomposition in diagonal and off diagonal parts in 
Lemma~\ref{l:DiagOffDiag}, for $\phi=(\phi_n)_{n\in\N}\in\fock$ and any $n\in\N$, we have 
\begin{equs}[e:SemiFinalLBUB]
\langle \Op_{2k+3} \phi_n,\phi_n\rangle_{\fock_n}&\leq 
\langle \big(\mu -\gensy\left(1+c^-_{2k+2}\cS_{2k+2}^N\right)\big)^{-1}\genap \phi_n,\genap\phi_n\rangle_{\Gamma L^2_{n+1}}\\
&\leq \langle \left(\mu -\gensy\left(1+c^-_{2k+2}\cS_{2k+2}^N\right)\right)^{-1}\genap \phi_n,\genap\phi_n\rangle_{\Di}\\
&\qquad +\sum_{i=1,2}\Big|\langle  (\mu-\gensy(1+c^-_{2k+2} \cS^N_{2k+2}))^{-1}\genap\phi_n,\genap\phi_n\rangle_{\mathrm{off}_i}\Big|\,.
\end{equs}
for which Lemmas~\ref{l:LBtoUB} and~\ref{l:OffDiagBound} provide an upper bound of the form
\begin{equs}
 \frac{1}{\pi c^-_{2k+2}}\Big\langle (-\gensy)\Big[\Big(1+\frac{4\lambda\pi C_{\Di}}{\sqrt{z_{2k+2}(1)}}\Big)\cS^N_{2k+3}+\lambda^2 C_{\oD}\pi c^-_{2k+2} \cS^{N,\oD}_{2k+3} \Big] \phi_n,\phi_n\Big\rangle_{\fock_n}\,.
\end{equs}
Note that, by~\eqref{e:OffSN} and~\eqref{e:SN}, we have 
\begin{equ}
\cS^{N,\oD}_{2k+3}=\frac{1}{c^-_{2k+2}}\frac{\cN^2}{z_{2k+2}(\cN)}\cS^N_{2k+3}
\end{equ}
which, together with the equation above, implies that~\eqref{e:SemiFinalLBUB} is upper bounded by 
\begin{equs}
 \frac{1}{\pi c^-_{2k+2}}\Big\langle (-\gensy)\Big(1+\frac{4\lambda\pi C_{\Di}}{\sqrt{z_{2k+2}(1)}}+\frac{\lambda^2 C_{\oD}\cN^2}{z_{2k+2}(\cN)}\Big)\cS^N_{2k+3} &\phi_n,\phi_n\Big\rangle_{\fock_n}\\
 &\leq c_{2k+3}^+\langle (-\gensy)\cS_{2k+3} \phi_n,\phi_n\Big\rangle_{\fock_n}
\end{equs}
where in the last step we exploited both~\eqref{e:CondonK} and the definition of $c_{2k+3}^+$ in~\eqref{e:Constants}, 
and the proof is completed. 
\end{proof}

\subsection{The operator $-\genap(\mu-\gensy)^{-1}\genam$}

We now come to the other operator we need to estimate in order to
control \eqref{e:Final}, namely $-\genap(\mu-\gensy)^{-1}\genam$. 
In view of \eqref{e:AByau} we need an
upper bound on this operator (as lower bound we will simply use that
$-\genap(\mu-\gensy)^{-1}\genam$ is positive). Note, however, that we
only need to analyse its action on elements of the second Wiener
chaos. This is because, in \eqref{e:Final}, $ \nf_\phi$ belongs to the
second chaos, and $\mathcal L_0$ and $\Op_n$ (and also $-\genap(\mu-\gensy)^{-1}\genam$) leave the order of the
chaos unchanged.  We have the following lemma.

\begin{lemma}\label{l:H3'} 
There exists a constant $c>0$ such that for any $\phi\in \fock_2$ and any function $G:\R_+\mapsto [1,\infty)$, 
\begin{equs}\label{e:H3'}
  -\langle \genap&(\mu-\gensy)^{-1}\genam\phi,\phi\rangle_{\Gamma L^2_2}
  &\leq c\lambda^2 \langle (-\gensy) \cS^{+-}\phi,\phi\rangle\,,
\end{equs}
where the operator $\cS^{+-}$ acts in Fourier space on $\phi\in \fock_2$ as
\begin{equ}
\cF(\cS^{+-}\phi)(\ell,m)= J^N_{\ell,m} g(\ell+m) G(\mu+\frac12(|\ell|^2+|m|^2))\hat{\phi}(\ell,m)\,,
	\end{equ}
where
\begin{equ}[e:fN]
g(k)\eqdef \frac{|k|^2}{\mu+\tfrac12 |k|^2}\sum_{\ell+m=k}\frac{\mathbb J^N_{\ell,m}}{\tfrac12(|\ell|^2+|m|^2)G(\mu+\frac12(|\ell|^2+|m|^2))}\,.
\end{equ}
\end{lemma}
\begin{proof}
Notice that, by Lemma \ref{lem:generator},
\begin{equs}[e:Op]
\langle (\mu-\gensy)^{-1}&\genam\phi_2,\genam\phi_2\rangle_{\Gamma L^2_2}\\
&=\frac{4\lambda^2}{\pi^2}\sum_{k}\frac{1}{\mu +\tfrac12|k|^2}\left(\sum_{\ell+m=k}|m|\frac{c(k,-\ell)}{|\ell||k|}\mathbb J^N_{\ell,m}\hat \phi(\ell,m)\right)^2.  
\end{equs}
We begin by analysing the inner sum and we treat differently the small and the large values of $\ell$. 
For lightness of notation, we write
\begin{equ}
  \sum'_{\ell+m=k}\dots\eqdef\sum_{\ell+m=k}\mathbb J^N_{\ell,m}\dots.
\end{equ}

We consider first the case $|\ell|\leq 2|k|$. Note that  $|m|=|\ell-k|\leq3|k|$, hence
\begin{equs}[e:ellsmall]
		\Big|\sum'_{\substack{\ell+m=k\\|\ell|\leq 2|k|}}|m|\frac{c(k,-\ell)}{|\ell||k|}\hat \phi(\ell,m)\Big|\leq
		%\leq 3|k|\sum'_{\substack{\ell+m=k\\|\ell|\leq 2|k|}}|\hat \phi(\ell,m)|=
		 3|k|\sum'_{\substack{\ell+m=k\\|\ell|\leq 2|k|}}|\hat \phi(\ell,m)|\,.
\end{equs}
To continue, we multiply and divide by $(\tfrac12(|\ell|^2+|m|^2))^{\tfrac12}G(\mu+\tfrac12(|\ell|^2+|m|^2))^{\tfrac12}$
and apply the Cauchy-Schwarz inequality. This readily gives the desired contribution. 
%Note that the sum over
%$\ell$ such that $\ell+m=k$ and $|\ell|\leq 2|k|$ of
%$(|\ell|^2+|m|^2)^{-1}$ is upper bounded by a constant.  Hence, the
%square of the right hand side in~\eqref{e:ellsmall} is bounded above by a (multiple
%of)
%\begin{equs}\label{e:ellsmallfinal}
%|k|^2\sum'_{\ell+m=k}(|\ell|^2+|m|^2)|\hat \phi(\ell,m)|^2\,.
%\end{equs}
%Using that $G\ge1$, we see that the contribution from $|\ell|\le 2|k|$ to \eqref{e:Op} is taken care by the term ``$1+$'' in the definition \eqref{e:fN} of $g$.

Next we consider the case $|\ell| > 2|k|$. 
Since $\phi$ is symmetric, so is $\Phi(\ell,m)\eqdef |\ell||m|\hat\phi(\ell,m)$, hence, the summand in the 
inner sum at the right hand side of~\eqref{e:Op} can be rewritten as 
\begin{equ}
\sum_{\substack{\ell+m=k\\ |\ell|>2|k|}} \frac{1}{|k|}% \sum'_{\ell+m=k}
	\frac{c(k,-\ell)}{|\ell|^2}\Phi(\ell,k-\ell)
	=\sum_{\substack{\ell+m=k\\ |\ell|>2|k|}}\frac{1}{2|k|}% \sum'_{\ell+m=k}
	\left(\frac{c(k,-\ell)}{|\ell|^2}+\frac{c(k,\ell-k)}{|k-\ell|^2}\right)\Phi(\ell,k-\ell)\,.
\end{equ}
A direct computation using the definition of $c(k,\ell)$ shows that the summand equals
\begin{equ}
	-\frac{c(k,k)}{2|k|}% \sum'_{\ell+m=k}
        \frac{|\ell|}{|k-\ell|}\hat\phi(\ell,k-\ell)% \\
	% &\qquad\qquad
        +\frac{1}{2|k|}% \sum'_{\ell+m=k}
        c(k,\ell)\left(\frac{1}{|k-\ell|^2}-\frac{1}{|\ell|^2}\right)\Phi(\ell,k-\ell)\,.
\end{equ}
Since $|\ell| >2|k|$, $3|\ell|/2\geq |k-\ell|\geq |\ell|/2$.
Therefore, 
\begin{equ}\label{eq:inner1}
	\Big|\frac{c(k,k)}{2|k|}\sum'_{\substack{\ell+m=k\\|\ell|> 2|k|}}\frac{|\ell|}{|k-\ell|}\hat\phi(\ell,k-\ell)\Big|\leq |k|\sum'_{\ell+m=k, |\ell|\ge 2|k|}|\hat\phi(\ell,k-\ell)|
\end{equ}
that can be estimated as~\eqref{e:ellsmall}.
To estimate the second summand above we note that, since $|c(k,\ell)|\leq|k||\ell|$, we have 
\begin{equ}
|c(k,\ell)||k-\ell||\ell|\Big|\frac{1}{|k-\ell|^2}-\frac{1}{|\ell|^2}\Big|
\leq |k|\frac{\big||\ell|^2-|k-\ell|^2\big|}{|k-\ell|}\leq \frac{|k|^2|2\ell-k|}{|k-\ell|}\lesssim|k|^2\,.
\end{equ}
Thus, 
\begin{equ}\label{eq:inner2}
	\Big|\frac{1}{2|k|}\sum'_{\substack{\ell+m=k\\|\ell|> 2|k|}}c(k,\ell)\left(\frac{1}{|k-\ell|^2}-\frac{1}{|\ell|^2}\right)\Phi(\ell,k-\ell)\Big|\lesssim |k|\sum'_{\substack{\ell+m=k\\|\ell|> 2|k|}}|\hat\phi(\ell,k-\ell)|\,,
\end{equ}
which once again can be bounded as~\eqref{e:ellsmall}. Hence, the result follows.
%Using~\eqref{eq:inner1} and~\eqref{eq:inner2} we see that for any $G$ as in the statement
%\begin{equs}[e:ellbig]
%\Big(&\sum'_{\substack{\ell+m=k\\|\ell|> 2|k|}}|m|\frac{c(k,-\ell)}{|\ell||k|}\hat \phi(\ell,m)\Big)^2\lesssim |k|^2\Big(\sum'_{\substack{\ell+m=k\\|\ell|> 2|k|}}|\hat\phi(\ell,k-\ell)|\Big)^2\\
%	%&\lesssim |k|^2\Big(\sum'_{\substack{\ell+m=k\\|\ell|> 2|k|}}\frac{((|\ell|^2+|m|^2)G(\mu+\frac12(|\ell|^2+|m|^2)))^{\half}}{((|\ell|^2+|m|^2)G(\mu + \frac12(|\ell|^2+|m|^2)))^{\half}}|\hat \phi(\ell,m)|\Big)^2\\
%	&\leq (\mu+ \tfrac12|k|^2)g(|k|)\Big(\sum'_{\substack{\ell+m=k\\|\ell|> 2|k|}}(|\ell|^2+|m|^2)G(\mu+\frac12(|\ell|^2+|m|^2))|\hat \phi(\ell,m)|^2\Big)\\
%	&\lesssim (\mu+ \tfrac12|k|^2)g(|k|) \sum'_{\ell+m=k}(|\ell|^2+|m|^2)G(\mu+\frac12(|\ell|^2+|m|^2))|\hat \phi(\ell,m)|^2\,
%\end{equs}
%where $g(x)$ is as in~\eqref{e:fN} and in the passage from the first to the second line we multiplied and divided 
%by $((|\ell|^2+|m|^2)G(\mu+\frac12(|\ell|^2+|m|^2)))^{\half}$ and applied the Cauchy-Schwarz inequality. Plugging this bound back into \eqref{e:Op} 
%%Since $G\geq 1$ it follows that~\eqref{e:ellsmallfinal} and~\eqref{e:ellbig} provide the required upper bound on~
% concludes the proof. 
\end{proof}

In view of Theorem~\ref{thm:Main} and the definition of the operators $\{\cS^N_{2k+1}\}_k$ in~\eqref{e:SN}, 
a special role will be played by the case in which the function $G$ is chosen to 
depend on $k\in\N$ and is of the form $G(x)\eqdef \UBN_{k-1}(x,z_{2k+1}(2))$. 
%The following lemma provides an upper bound on $g$ in this scenario. 

\begin{lemma}\label{l:g}
In the setting of Lemma~\ref{l:H3'}, choose $G$ as $G(x)\eqdef \UBN_{k-1}(x,z_{2k+1}(2))$. Then, 
there exists a constant $c_0>0$ independent of $\mu,\,k$ and $N$ for which 
\begin{equ}
  \label{e:lg}
g(j)\leq c_0\frac{|j|^2}{\mu+\tfrac12|j|^2}\times
\begin{cases}
\frac{\log(\mu/|j|^2)}{\UBN_{k-1}(8\mu, z_{2k+1}(2))}+\frac{1+\mu_N}{\lambda^2} \LBN_k(8\mu,z_{2k+1}(2))\,,& \text{if $|j|^2\leq \mu$}\\
\frac{1+\mu_N}{\lambda^2} \LBN_k(4(\mu +\tfrac12|j|^2),z_{2k+1}(2))\,, & \text{if $|j|^2> \mu$.}
\end{cases}
\end{equ}
\end{lemma}
\begin{proof}
Note that with our choice of $G$ it is enough to estimate
\begin{equs}
\sum_{\ell+m=j}\frac{\mathbb J^N_{\ell,m}}{(|\ell|^2+|m|^2)\UBN_{k-1}(\mu+\frac12(|\ell|^2+|m|^2), z_{2k+1}(2))}\,.
\end{equs}
Thanks to~\eqref{e:lm}, and by an immediate extension of Lemma~\ref{lemma:RiemannNew},
the previous is upper bounded by 
\begin{equs}
% \indN{k}
\sum_{1/N\leq |\ell/N|\leq 1}&\frac{1}{N^2}\frac{1}{\tfrac14(|\ell/N|^2+|j/N|^2)\UB_{k-1}(4(\mu_N+\frac12(|\ell/N|^2+|j/N|^2), z_{2k+1}(2))}\\
&\lesssim % \indN{k}
\int_{0}^1\frac{\rho\,\dd\rho}{(\rho^2+|j/N|^2)\UB_{k-1}(4(\mu_N+\frac12(\rho^2+|j/N|^2)), z_{2k+2}(2))}\\
&\lesssim % \indN{k}
\int_{|j/N|^2}^{1}\frac{\dd\rho}{\rho\UB_{k-1}(4(\mu_N+\rho), z_{2k+1}(2))}
\end{equs}
where in the last line we enlarged the integration interval by using
that $|j/N|\leq 1$, which holds for all values of $j$ appearing above, because by definition ${\mathbb J^N_{\ell,m}}$ is zero
if $|j|=|\ell+m|>N$. We now distinguish two cases, depending on the
relation between $\mu$ and $|j|^2$. If $|j|^2\leq\mu$, then we
split the integral as
\begin{equ}
  \Big(\int_{|j/N|^2}^{\mu_N}+\int_{\mu_N}^{1}\Big)\frac{\dd\rho}{\rho\UB_{k-1}(4(\mu_N+\rho),
    z_{2k+1}(2))}=: I_1+I_2\,.
\end{equ}
For $I_1$ we exploit the fact that $\UB$ is decreasing, so that 
\begin{equ}
I_1\leq \frac{1}{\UB_{k-1}(8\mu_N, z_{2k+1}(2))}\int_{|j/N|^2}^{\mu_N}\frac{\dd\rho}{\rho}=\frac{\log(\mu/|j|^2)}{\UBN_{k-1}(8\mu, z_{2k+1}(2))}.
\end{equ}
For the other integral we have 
\begin{equs}
I_2&\lesssim \int_{\mu_N}^{1}\frac{\dd\rho}{4(\mu_N+\rho)\UB_{k-1}(4(\mu_N+\rho), z_{2k+1}(2))}\\
&\lesssim (1+\mu_N)\int_{8\mu_N}^{4(1+\mu_N)}\frac{\dd\rho}{(\rho^2+\rho)\UB_{k-1}(\rho, z_{2k+1}(2))}\lesssim
\frac{1+\mu_N}{\lambda^2} \LBN_k(8\mu,z_{2k+1}(2))\,,
\end{equs}
where we used \eqref{e:IntUBtoLB}.
If instead $|j/N|^2>\mu_N$, then, proceeding as in the bound for $I_2$ we get 
\begin{equs}
\int_{|j/N|^2}^{1}\frac{\dd\rho}{\rho\UB_{k-1}(4(\mu_N+\rho), z_{2k+1}(2))}&\lesssim \int_{|j/N|^2}^{1}\frac{\dd\rho}{4(\mu_N+\rho)\UB_{k-1}(4(\mu_N+\rho), z_{2k+1}(2))}\\
&\lesssim \frac{1+\mu_N}{\lambda^2} \LBN_k(4(\mu +\tfrac12|j|^2),z_{2k+1}(2))\,,
\end{equs}
and the statement follows. % the proof is concluded. 
\end{proof}

\subsection{Estimating $\cB_\phi^N(\mu)$}

Based on the results obtained above, we are ready to formulate and
prove the main result of this section. In the next proposition, we
provide both an upper and a lower bound on the Laplace transform of
the second moment of $B^N_\phi(t)$ given in~\eqref{e:BN}.  We will
adopt the same notations and conventions introduced at the beginning
of Section~\ref{sec:iteration}.
% and for $\phi$ a test function and $a>0$ we set 
% \begin{equ}[e:phi]
% \|\phi^{(\le a)}\|_{L^2(\T^2)}^2\eqdef\sum_{|j|^2\le a}|\hat\phi(j)|^2\,.
% \end{equ}

\begin{proposition}\label{p:mainB}
Let $\lambda>0$ and, for $N\in\N$, let $h^N$ be the solution of~\eqref{e:akpz:torus1} and $\phi\in L^2(\T^2)$ be a test function. 
Let $\cB^N_\phi$ be defined as in~\eqref{e:Laplace}.  
\begin{enumerate}[label=$\mathrm{(UB)}$]
\item\label{i:UB} For every $\delta>0$ there exists a constant $c_\delta>0$ (depending also on $\lambda$) such that for all $N\in\N$ and $\mu>0$ 
\begin{equs}[e:Bupper]
\cB_\phi^N(\mu)\leq \frac{c_\delta}{\,\mu}\Big[\LlN(\mu,0)\Big]^{\frac12}\Big(\log \LlN(\mu,0)\Big)^{5+\delta}\|\phi\|_{L^2(\T^2)}^2
\end{equs}
where $\Ll^N(\mu,0)$ is defined according to~\eqref{e:L} (see also Definition \ref{def:OpS});
\end{enumerate}
\begin{enumerate}[label=$\mathrm{(LB)}$]
%\begin{equ}[e:p]
%p(N,\mu,\lambda) = 1+\lambda^2 \log\Big(1+\frac{N^2}{\mu}\Big)\,,
%\end{equ}
\item\label{i:LB} There exists a constant $C$ such that
   for all $N\in\N, k\in\N$ and $0<\mu\le N^2$, 
\begin{equation}\label{e:Blower}
\begin{aligned}
  \cB_\phi^N(\mu)\geq \frac{1}{C\,\mu} \sum_{|j|^2\le \mu}&\frac{|\hat\phi(j)|^2}{c_{2k+1}^+f_{2k+1}(2)+\frac{|j|^2}{\mu}{\LBN_k(4\mu,z_{2k+1}(2))}}\times\\
  &\times\Big[{\LBN_k((\mu+\tfrac12|j|^2)\vee 1,z_{2k+2}(2))}-f_{2k+1}(2)\Big]\,.
 \end{aligned}
\end{equation}
\end{enumerate}

\end{proposition}

\begin{proof}
We use~\eqref{e:Laplace} and we focus on the scalar product at its right hand side. 
Throughout the proof, in order to lighten the notation we omit the subscript
``${\fock}$'' in all scalar products appearing below. 

Let us begin with~\ref{i:UB}. By Lemma~\ref{l:Sandwich} together with 
the fact that $-\genap(\mu-\gensy)^{-1}\genam$ is a positive operator, for any $k\in\N$ we have 
\begin{equs}[e:Reduction]
\langle \nf_\phi&, (\mu-\gen)^{-1}\nf_\phi\rangle\leq \langle \nf_\phi, (\mu-\gensy+\Op_{2k+2}-\genap(\mu-\gensy)^{-1}\genam)^{-1}\nf_\phi\rangle\\
&\leq\langle \nf_\phi, (\mu-\gensy+\Op_{2k+2})^{-1}\nf_\phi\rangle\leq \langle \nf_\phi, (\mu-\gensy(1+c_{2k+2}^-\cS^N_{2k+2}))^{-1} \nf_\phi\rangle\,,% \\
\end{equs}
where in the last passage we applied Theorem~\ref{thm:Main}. 
Recalling the definition of $\nf_\phi$ in~\eqref{eq:ches} and~\eqref{eq:explFock}, and observing that the operator
$(\mu-\gensy(1+c_{2k+2}\cS^N_{2k+2}))^{-1} $ is diagonal in Fourier space,
the right hand side equals 
\begin{equs}
% \langle (-2\gensy)^{-\frac12}\phi, (-\Delta/2)\cS^N_{2k+3}(-\Delta)^{-\frac12}\phi\rangle&=
  2&\lambda^2\sum_{j\in\Z^2} |\hat\phi(j)|^2\times\\
  &\times\sum_{\ell+m=j}\frac{(\nonlin_{\ell,m})^2}{\mu+\Gamma(\ell,m)\Big(1+\frac{c_{2k+2}^-}{f_{2k+2}(2)}[\LBN_k(\mu+\Gamma(\ell,m), z_{2k+2}(2))-f_{2k+2}(2)]\Big)}\,,
  % \UBN_{k}(|j|^2/2+\mu, M(2k+4)^2)
% &\leq \frac12 \UB_{k}(\mu_N , M(2k+4)^2)  \sum_{j\in\Z^2} |\hat\phi(j)|^2=\frac12\UB_{k}(\mu_N, M(2k+4)^2)\|\phi\|_{L^2(\T^2)}^2
\end{equs} 
where we adopted the same convention as in~\eqref{e:Conv1} for $\Gamma$ and used that, by the definition of $\nonlin$ and $\indN{}$, $|\ell|,|m|\geq 1$, to remove the $\vee 1$ that appears in the definition of $\cS^M_{2k+2}$. 
Note that the inner sum (for fixed $j$) is precisely of the form \eqref{e:MainQntyUB}, except that
$|k_{2:n}|$ is set to zero. Therefore, proceeding as in the proof of Lemma~\ref{l:LBtoUB}, we obtain
\begin{equation}\label{e:FinalUPPER}
\begin{aligned}
  \langle \nf_\phi&, (\mu-\gen)^{-1}\nf_\phi\rangle\\
  &\leq \frac{f_{2k+2}(2)}{2 \pi c_{2k+2}^-}\Big(1+\frac{4\lambda\pi C_{\Di}}{\sqrt{z_{2k+2}(1)}}\Big)\sum_{j\in\Z^2} |\hat\phi(j)|^2 \UBN_k(\mu+\tfrac12|j|^2,z_{2k+3}(2))\\
  &\leq \frac{c_{2k+3}^+}{2} f_{2k+2}(2) \UBN_k(\mu, z_{2k+3}(2)) \|\phi\|^2_{L^2(\T^2)}\lesssim  \frac{k^{\frac92 +3\delta}}{\LBN_k(\mu, 0)} \LlN(\mu,0)\|\phi\|^2_{L^2(\T^2)}
  \end{aligned}
 \end{equation}
where the second inequality holds in view of~\eqref{e:CondonK} and 
the monotonicity of $\UBN$ in the first variable, while the last by the definition of $\UBN$, the monotonicity of $\LBN$ in its second argument, the definition of $\LlN$ and $f_{2k+3}, z_{2k+3}$ 
in~\eqref{e:LUBk} and in~\eqref{e:Functions} respectively,
and the fact that, in view of Theorem~\ref{thm:Main} 
the sequence $\{c_{2k+1}^+\}_k$ converges to a finite constant depending on $\delta$, so that 
the constant hidden in $\lesssim$ 
is an absolute positive constant depending only on $\delta,\lambda$ but on neither $k$ nor $N$.

At this point, it remains to optimise over $k$ in order to obtain the smallest possible upper bound. 
By Stirling's formula 
\begin{equ}[e:Stirling]
 \frac{k^{\frac92 +3\delta}}{\LBN_k(\mu, 0)} \leq  \frac{k^{\frac92 +3\delta} k!}{(\frac12 \log \LlN(\mu,0))^k}\lesssim e k^{5+3\delta} \exp\Big[k \log\Big(\frac{2k}{e \log \LlN(\mu,0)}\Big)\Big]\,.
\end{equ}
We choose then $k=k(\mu/N^2)$  as
\begin{equ}
k(\mu/N^2)\eqdef \left\lfloor\frac12 \log \Ll^N(\mu,0)\right\rfloor=\left\lfloor\frac12 \log \Ll(\mu/N^2,0)\right\rfloor\,\label{e:k}\,.
\end{equ}
With this choice of $k$ and~\eqref{e:Stirling}, we obtain  
\begin{equ}[e:ULB]
\frac{k^{\frac92 +3\delta}}{\LBN_k(\mu, 0)} \lesssim \Big[\LlN(\mu,0)\Big]^{-\frac12}\Big(\log \LlN(\mu,0)\Big)^{5+3\delta}
\end{equ}
from which, plugging~\eqref{e:ULB} into~\eqref{e:FinalUPPER} and recalling~\eqref{e:Laplace},
~\eqref{e:Bupper} follows (with $\delta$ replaced by $3\delta$).
\medskip

We now turn to~\ref{i:LB}. Arguing as in~\eqref{e:Reduction}, for any $k\in\N$ we have 
\begin{equ}[e:UltimateLB]
\langle \nf_\phi,(\mu-\gen)^{-1}\nf_\phi\rangle\ge \langle \nf_\phi,(\mu-\gensy+\Op_{2k+1}-\genap(\mu-\gensy)^{-1} \genam)^{-1}\nf_\phi\rangle\,.
\end{equ}
For $\Op_{2k+1}$ we use the upper bound provided by
Theorem~\ref{thm:Main} while for $-\genap(\mu-\gensy)^{-1} \genam$ we
use Lemma~\ref{l:H3'} with the choice
$G(x)\eqdef \UBN_{k-1}(x,z_{2k+1}(2))$. 
Hence,~\eqref{e:UltimateLB}
is bounded below by
\begin{equs}[e:BigSumLB]
{2\lambda^2}&\sum_j|\hat\phi(j)|^2\sum_{\ell+m=j}\frac{(\nonlin_{\ell,m})^2}{\mu+\Gamma(\ell,m)(1+F_k(j))\UBN_{k-1}(\mu+\Gamma(\ell,m),z_{2k+1}(2)))}\\
&\geq {\lambda^2}\sum_j\frac{|\hat\phi(j)|^2}{F_k(j)}\sum_{\ell+m=j}\frac{(\nonlin_{\ell,m})^2}{\mu+\Gamma(\ell,m)\UBN_{k-1}(\mu+\Gamma(\ell,m),z_{2k+1}(2))}
\end{equs}
where we introduced $F_k(j)\eqdef c_{2k+1}^+f_{2k+1}(2)+c\lambda^2g(j)\geq 1$
for $k\in\N$ and $j\in\Z^2$, in which $c$ is the constant that
appears in \eqref{e:H3'} and $g$ is defined in \eqref{e:fN}. 
Also in
this case, the inner sum in~\eqref{e:BigSumLB} has the same structure
as in~\eqref{e:MainQnty}, so that, proceeding as in the proof of
Lemma~\ref{l:UBtoLB}, we
obtain
\begin{equs}[e:UltimateSumLB1]
\sum_{\ell+m=j}&\frac{(\nonlin_{\ell,m})^2}{\mu+\Gamma(\ell,m)\UBN_{k-1}(\mu+\Gamma(\ell,m),z_{2k+1}(2))}\\
&\geq \frac{1}{4\lambda^2\pi}\Big(1-\frac{4\pi \lambda C_{\Di}}{\sqrt{z_{2k+1}(1)}}\Big)\Big[\LBN_k(1+\mu+\tfrac12|j|^2 ,z_{2k+1}(2))-f_{2k+1}(2)\Big]\\
&\geq \frac{1}{8\lambda^2\pi}\Big[\LBN_k(1+\mu+\tfrac12|j|^2,z_{2k+1}(2))-f_{2k+1}(2)\Big]
\end{equs}
where we further exploited that by our choice of $z_{2k+1}$ in~\eqref{e:Functions} with $K_1$ large, the quantity the constant 
in the rounded brackets in the second line is bigger than $1/2$ (see~\eqref{e:CondonK}). 

Restricting to $|j|^2\leq \mu$ and plugging the expression for $F_k(j)$ back in, 
we see that, for all $k\in\N$ 
(modulo constants independent of $\mu$, $k$ and $N$)~\eqref{e:BigSumLB} is lower bounded by 
\begin{equs}[e:BigSumLB2]
\sum_{|j|^2\le \mu}\frac{|\hat\phi(j)|^2}{c_{2k+1}^+f_{2k+1}(2)+{\lambda^2g(j)}}\Big[{\LBN_k(1+\mu+\tfrac12|j|^2,z_{2k+1}(2))}-f_{2k+1}(2)\Big]\,. 
\end{equs}
We are left to deal with the denominator in the sum in~\eqref{e:BigSumLB2} for which we need Lemma~\ref{l:g}.
Since $|j|^2\leq \mu\leq N^2$ (in particular $\mu_N\le 1$) for any $k\in\N$ we have 
\begin{equs}
g(j)&\lesssim \frac{|j|^2}{\mu}\Big(\frac{\log(\mu/|j|^2)}{\UBN_{k-1}(8\mu, z_{2k+1}(2))}+\frac{1}{\lambda^2} \LBN_{k}(8\mu,z_{2k+1}(2))\Big)\\
&\lesssim \frac{|j|^2}{\lambda^2\mu}\Big(\LBN_{k-1}(8\mu,z_{2k+1}(2))+ \LBN_{k}(8\mu,z_{2k+1}(2))\Big)\lesssim \frac{|j|^2}{\lambda^2\mu}{\LBN_{k}(4\mu,z_{2k+1}(2))}
\end{equs}
where, in the passage from the first to the second line we exploited the definition of $\UB^N_{k-1}$ 
and of $\Ll^N$, while in the last the monotonicity of $\LB_k$ with respect to its first argument and 
the fact that for all $k$, $\LB_{k-1}\leq\LB_k$.  
Recalling \eqref{e:Laplace}, we deduce \eqref{e:Blower}.
\end{proof}

\section{Proof of the main results}
\label{sec:provathmmain}

This section is devoted to the proofs of the main theorems. We begin with the bulk diffusivity, since, as we will see, 
the bounds we aim at follow directly from Proposition~\ref{p:mainB}. 

\subsection{The bulk diffusivity: proof of Theorem~\ref{thm:BD}}

At first we provide an equivalent formulation of the bulk diffusivity $D_N$ defined in~\eqref{e:BD}, which 
shows that $D_N$ represents the average speed at which the mass of the solution $H_N$ spreads in time. 

\begin{lemma}\label{l:BD}
For $N\in\N$, let $D_N$ be the bulk diffusivity defined in~\eqref{e:BD}. Then, for all $N\in\N$ and $t>0$, 
the following equality holds
\begin{equ}[e:BDNew]
t\, D_N(t)=t + N^2\Exp[B_{e_0}^N(t/N^2)^2]
\end{equ}
where $B^N$ was defined in~\eqref{e:BN} and $e_0\equiv\frac{1}{2\pi}$ is the $0$-th Fourier basis element.
\end{lemma}
\begin{proof}
Notice at first, that in view of the scaling relation~\eqref{e:Scaling} and the definition of $\cN^N$ in~\eqref{eq:ches}, 
it is immediate to see that for any $N\in\N$ and $t\geq 0$
\begin{equ}[e:BDScaling]
t D_N(t)=t + 2\lambda^2N^2\int_0^{t/N^2}\int_0^s \int_{\T^2}\Exp\Big[\cN^{N}[u^{N}](r,0)\cN^{N}[u^{N}](0,x)]\Big]\dd x\dd r\dd s\,.
\end{equ}
Now, the process $u^N$ with white noise initial condition is
translation invariant in law, which implies that, for every $r\geq 0$, 
the spatial integral in the second summand on the right hand equals
\begin{equs}
\frac{1}{4\pi^2}\int_{\T^2}\int_{\T^2}\Exp\Big[\cN^{N}[u^{N}](r,y) &\cN^{N}[u^{N}](0,x+y)]\Big]\dd x\dd y\\
&=\Exp\Big[\cN^{N}[u^{N}(r)](e_0) \cN^{N}[u^{N}(0)](e_0)]\Big]
\end{equs}
where the last passage can be obtained by integrating first in $x$ and then in $y$. To conclude it is sufficient to note that for any $t\geq 0$
\begin{equ}
\int_0^t\int_0^s \Exp\Big[\cN^{N}[u^{N}(r)](e_0) \cN^{N}[u^{N}(0)](e_0)]\Big]\dd r\dd s=\frac{1}{2}\Exp\Big[\Big(\int_0^t \cN^{N}[u^{N}(s)](e_0)\dd s\Big)^2\Big]\,.
\end{equ}  
\end{proof}

The advantage of~\eqref{e:BDNew} is that the bulk diffusivity $D_N$ is expressed in terms of an observable of the 
form~\eqref{e:BN} so that the results in the previous section are directly applicable. 
We are now ready to prove Theorem~\ref{thm:BD}. 

\begin{proof}[of Theorem~\ref{thm:BD}]
Thanks to~\eqref{e:BDNew} and~\eqref{e:Laplace}, it is immediate to 
show that for every $N\in\N$
\begin{equ}
\cD_N(\mu)=\mu\int_0^\infty e^{-\mu t} t\, D_N(t)\dd t =\frac{1}{\mu}+{N^2}\cB_{e_0}^N(\mu N^2)\,.
\end{equ}
Therefore, it remains to bound the second summand, for which we exploit Proposition~\ref{p:mainB}. 
We begin with the upper bound. Notice that~\eqref{e:Bupper} gives
\begin{equs}
N^2\cB_{e_0}^N(\mu N^2)&\leq \frac{C}{\mu} \Big[\LlN(\mu N^2,0)\Big]^{\frac12}\Big(\log \LlN(\mu N^2,0)\Big)^{5+3\delta}=\frac{C}{\mu} \Big[\Ll(\mu,0)\Big]^{\frac12}\Big(\log \Ll(\mu,0)\Big)^{5+3\delta}
\end{equs}
from which~\eqref{e:BDUBound} follows.  

For the lower bound instead,~\eqref{e:Blower} implies that for all $k\in\N$, $\mu>0$ we have
\begin{equ}[e:nonop]
\liminf_{N\to\infty} N^2\cB_{e_0}^N(\mu N^2)\geq \frac{1}{C\,\mu} \Big[\frac{\LB_k(\mu,z_{2k+2}(2))}{c^+_{2k+1}f_{2k+1}(2)}-\frac{1}{c^+_{2k+1}}\Big]\,,
\end{equ}
and we are left to optimise over $k$. But we have already done so in the proof of Proposition~\ref{p:mainB}\ref{i:UB}. 
Upon choosing $k=k(\mu)$ as in~\eqref{e:k}, by~\eqref{e:ULB}, we have 
\begin{equs}[e:Estimate]
\frac{\LB_k(\mu,z_{2k+2}(2))}{c^+_{2k+1}f_{2k+1}(2)}\gtrsim \Big[\Ll(\mu,0)\Big]^{\frac12}\Big(\log \Ll(\mu,0)\Big)^{-5-3\delta}%\quad\text{and}\quad\frac{1}{c_{2k+1}}\leq \Ll(\mu,0)^{-\frac{\log C_\fin}{R}}
\end{equs}
and, since the right hand side diverges as $\mu\to 0$, at the price of an absolute constant, we can 
reabsorb the $-1/c^+_{2k+1}$ in~\eqref{e:nonop}, so that the proof is complete. 
\end{proof}

\subsection{The diffusive scaling: proof of Theorem~\ref{thm:main2} and Corollary \ref{cor:t+t-}}

In order to prove Theorem~\ref{thm:main2}, 
we first consider the weak formulation of AKPZ on the torus of side length $1$ 
and separately analyse 
each of the three summands appearing on the right hand side of~\eqref{e:akpz:torus1}. 
More precisely, let $\phi$ be a smooth test function on $\T^2$ and $N$ fixed, then, for all $t\geq 0$, $h^N$ satisfies
\begin{equ}\label{eq:ABC}
h^{N}(t)[\varphi]-h^{N}(0)[\varphi]=\underbrace{\frac12\int_0^t h^{N}(s)[\Delta\varphi]\dd s}_{A^N_\varphi(t)}+B_\phi^N(t)+\underbrace{\int_0^t\xi(\dd s)[\varphi] }_{C^N_\varphi(t)}\,,
\end{equ}
where $B_\phi^N(t)$, the integral in time of the nonlinearity, was defined in~\eqref{e:BN}. We recall that $B^N_\phi$ is a centered random variable, and the same can be easily verified for $A^N_\phi,C^N_\phi$.
Now, $B^N_\phi$, or more precisely the Laplace transform of its second moment, has been thoroughly studied 
in the previous section. In the following proposition, %which is based on a direct application of the so-called It\^o trick 
%and basic properties of the space-time white noise, 
we derive suitable bounds on the second moments of $A^N_\varphi$ and $C^N_\varphi$ 
and on their Laplace transforms. To that end we define  
\begin{equ}[e:ACLaplace]
\cA_\phi^N(\mu)\eqdef \mu \int_0^\infty e^{-\mu t} \Exp A_\phi^N(t)^2\, \dd t\quad\text{and}\quad\cC_\phi^N(\mu)\eqdef\mu  \int_0^\infty e^{-\mu t} \Exp C_\phi^N(t)^2\, \dd t
\end{equ}
for $\mu>0$. 

\begin{proposition}\label{p:mainAC}
Let $N\in\N$, $h^N$ be the solution of~\eqref{e:akpz:torus1} and $\phi\in L^2(\T^2)$ be a test function. 
Let $A^N_\varphi(t)$ and $C^N_\varphi(t)$ be defined according to~\eqref{eq:ABC}. 
Then, there exists a constant $c>0$ independent of $N$ and $\phi$ such that for every $t,\,\mu>0$ we have 
\begin{equs}
\Exp A_\phi^N(t)^2\leq ct\|\phi\|_{L^2(\T^2)}^2\quad&\text{and}\quad\cA_\phi^N(\mu) \leq \frac{c}{\mu}\|\phi\|_{L^2(\T^2)}^2\,,\label{e:A}\\
\Exp C_\phi^N(t)^2=t \|\phi\|_{L^2(\T^2)}^2\quad&\text{and}\quad\cC_\phi^N(\mu)= \frac{1}{\mu}\|\phi\|_{L^2(\T^2)}^2\label{e:C}\,.
\end{equs}
\end{proposition}
\begin{proof}
The result on the Laplace transform is an immediate consequence of that on $A^N_\varphi$ and $C^N_\varphi$. 
The first identity in~\eqref{e:C} is straightforward and 
follows from an explicit computation that uses the correlation structure of the white noise $\xi$.
To estimate $A_\phi^N(t)$ we make use of~\cite[Lemma 4.3, Eq. (4.11)]{CES19}, which states that for any $T>0$
\begin{equ}[e:Energy]
\E\Big[\sup_{t\leq T}\Big|\int_0^t u^N(s)[\Delta \phi]\dd s\Big|^2\Big]^{1/2}\lesssim T^{1/2}\|\phi\|_{1,2}\,,
	\end{equ}
where 
\begin{equ}
\|\phi\|_{1,2}^2 =\sum_{k\in\Z^2}(1+|k|^2)|\hat{\phi}(k)|^2\,.
	\end{equ}
Here, we recall that $u^N$ is the solution to the stochastic Burgers equation~\eqref{e:AKPZ:u} so that, upon noting
\begin{equ}[e:Arewrite]
A^N_\varphi(t)=\frac12\int_0^t u^{N}(s)[(-\Delta)^{1/2}\varphi]\dd s\,,
% =\frac12 \int_0^t u^{N}(s)[\Delta [(-\Delta)^{-1/2}(-\phi)]\dd s\,.
\end{equ}
the result follows at once.
% Note that $(-\Delta)^{-1/2}\phi$ is not well defined if $\hat{\phi}(0)\neq 0$. However, in that case we simply \emph{define} $\cF((-\Delta)^{1/2} \phi)(0)=0$, and since the zero Fourier mode of $\Delta \phi$ is zero the second identity in~\eqref{e:Arewrite} holds. The claim then follows from~\eqref{e:Energy} and the fact that $\|(-\Delta)^{1/2}\phi\|_{1,2}= \|\phi\|_{L^2(\T^2)}$.
\end{proof}

The previous statement was the missing ingredient needed in the proof of 
Theorem~\ref{thm:main2}. 

\begin{proof}[of Theorem~\ref{thm:main2}]
Notice at first that, given any test function $\phi$, by the definition of $H^\eps_N[\phi]$ in~\eqref{e:PerHRescaled} 
and the equality in law~\eqref{e:Scaling}, it follows that 
\begin{equ}
\cV^{\eps,N}_{\phi}(\mu)=\cV^{N^{-1},1}_{\phi^{(\eps N)}}(\eps^2 N^2 \mu)
\end{equ}
where, for any $a>0$, $\phi^{(a)}$ is given as in the introduction, i.e. $\phi^{(a)}(\cdot)=a^2\phi(a \cdot)$. 
In view of the decomposition~\eqref{eq:ABC}, we have 
\begin{equ}
\cV^{N^{-1},1}_{\phi^{(\eps N)}}(\eps^2 N^2 \mu)\lesssim \cA_{\phi^{(\eps N)}}^N(\eps^2N^2\mu) + \cB_{\phi^{(\eps N)}}^N(\eps^2N^2\mu) + \cC_{\phi^{(\eps N)}}^N(\eps^2N^2\mu)
\end{equ}
so that, since $\|\phi^{(N\eps)}\|^2_{L^2(\T^2)}= (N\eps)^2\|\phi\|^2_{L^2(\R^2)}$, 
the bound~\eqref{e:UBVL} follows immediately from Propositions~\ref{p:mainAC} and~\ref{p:mainB}\ref{i:UB}.  

We now turn to the lower bounds. To that end note that for any $\mu>0$,  we have 
\begin{equ}[e:lowerV]
\cV^{N^{-1},1}_{\phi^{(\eps N)}}(\eps^2 N^2 \mu) \gtrsim \cB^N_{\phi^{(\eps N)}}(\eps^2 N^2 \mu) - \frac{\|\phi\|_{L^2(\R^2)}}{\sqrt{\mu}} \sqrt{\cB^N_{\phi^{(\eps N)}}(\eps^2 N^2 \mu)}-\frac{\|\phi\|^2_{L^2(\R^2)}}{\mu}
\end{equ}
where we exploited once more the decomposition~\eqref{eq:ABC}, the Cauchy-Schwarz inequality 
to control the cross products as well as Proposition~\ref{p:mainAC}, 
to bound the occurrences of $\cA_\phi^N$ and $\cC_\phi^N$. 
Now, by Proposition~\ref{p:mainB}\ref{i:LB}, for any integer $k\in\N$, we have 
\begin{equs}[tent2]
\liminf_{N\to\infty}&\,\cB^N_{\phi^{(\eps N)}}(\eps^2 N^2 \mu)\\
&\gtrsim\left[
1-\frac{1}{ c_{2k+1}^+Y_k}\right]
\lim_{N\to\infty}\frac1{(N\eps)^2}\sum_{|j|^2\le (N\eps)^2\mu}\frac{|\widehat{\phi^{(N\eps)}}(j)|^2}{\mu/Y_k+ |j|^2/(N\eps)^2}
\end{equs}
where
\begin{equ}[eq:Yk12]
  Y_k\eqdef \frac{\LB_k(4\eps^2\mu,z_{2k+1}(2))}{c_{2k+1}^+f_{2k+1}(2)}
\end{equ}
and we used the fact that $\LB_k$ is continuous and decreasing in its first argument. 
Since $\widehat{\phi^{(N\eps)}}(j)= \hat \phi(p)$ for $p\eqdef j/(N\eps)$, the limit in~\eqref{tent2} reduces to
  \begin{equ}[sognounpanino]
\lim_{N\to\infty}    \frac1{(N\eps)^2}\sum_{\substack{p\in (\mathbb Z/{N\eps})^2\\ |p|^2\le \mu}}\frac{|\hat \phi(p)|^2}{\mu/Y_k+{|p|^2}}=\int_{\R^2}\mathds{1}_{ |p|^2\le \mu}\frac{|\hat \phi(p)|^2}{\mu/Y_k+{|p|^2}}\dd p.
  \end{equ}
%  Let us consider first the case $\int_{\R^2}\phi(x)\dd x=0$.
Now we proceed similarly to the proof of the lower bound in Theorem \ref{thm:BD}. 
Namely, we fix $k(\eps^2\mu)$ as in~\eqref{e:k} and we note that, since we are assuming 
$\mu\leq \eps^{-1}$, $k(\eps^2\mu)$ diverges as $\eps\to 0$. 
Moreover, arguing once more as in the proof of Proposition~\ref{p:mainB}\ref{i:UB} we see that also $Y_{k(\eps^2\mu)}$ 
diverges since 
\begin{equ}[pranzo]
Y_{k(\eps^2\mu)}\gtrsim\sqrt{\Ll(\eps^2\mu,0)} (\log \Ll(\eps^2\mu,0))^{-5-3\delta}% \ge \left[1+\lambda^2\log\left(1+\frac1{\eps^2\mu}\right)\right]^\half (\log \left[1+\lambda^2\log\left(1+\frac1{\eps^2\mu}\right)\right])^{-5-3\delta}
.
\end{equ}
In particular, the negative term in~\eqref{tent2} can be neglected.
Also, since we are assuming $\mu\le (\log (1/\eps))^\half (\log\log(1/\eps))^{-5-3\delta}$, the ratio $\mu/Y_k$ tends to zero as $\eps\to 0$. Altogether, we get 
\begin{equ}\label{tent3bis}
\liminf_{\eps\to0}  \liminf_{N\to\infty}\,\cB^N_{\phi^{(\eps N)}}(\eps^2 N^2 \mu)\gtrsim
\int_{\R^2}\mathds{1}_{ |p|^2\le \mu}\frac{|\hat \phi(p)|^2}{{|p|^2}}\dd p.
\end{equ}
If $\phi$ has non-zero average then the integral is infinite. If instead the (smooth) function
 $\phi$ has zero average, the integral is finite: in this case, assuming that $\mu$ is sufficiently large (larger than some constant $a_\phi$), the integral is arbitrarily close to $\|\phi\|_{-1}^2$. In either case,  \eqref{e:UBVLmodo1} is proven.
\end{proof}
\begin{proof}[of Corollary \ref{cor:t+t-}]
Throughout the proof we fix $\delta>0$ as in the formulation of Corollary~\ref{cor:t+t-}. We start with \eqref{e:nat} so that
  \begin{equ}
    0\le 1-\frac{{\rm Cov}(H^{\eps}_N(t)[\phi], H^{\eps}_N(0)[\phi])}{{\rm
        Var}(H^{\eps}_N(0)[\phi])}= \frac{V_\phi^{\eps,N}(t)}{2{\rm
      Var}(H^{\eps}_N(0)[\phi])}
\label{eq:nat2}
\end{equ}
  and we recall that the denominator is uniformly positive and finite: since the stationary measure is the GFF on the plane and $\int_{\R^2}\phi(x)\dd x=0$,
  \begin{equ}
\lim_{N\to\infty}{\rm
      Var}(H^{\eps}_N(0)[\phi])=\frac{1}{2\pi^2}\int_{\R^2}\frac{|\hat\phi(k)|^2}{|k|^2}\dd k<\infty\,.
\label{e:nat3}    
  \end{equ}

To prove \eqref{e:cor1} we need to show that $V_\phi^{\eps,N}(t)$ is uniformly small for $t<t_-(\eps)$.
Recall the definition \eqref{e:Vi} of $V_\phi^{\eps,N}(t)$, together with
$H_N^\eps(t)[\phi]=h^N(t/(N\eps)^2)[\phi^{(\eps N)}]$.
One has then (with the notations of \eqref{eq:ABC})
\begin{equs}
  V_\phi^{\eps,N}(t)\lesssim \Exp A_{\phi^{(\eps N)}}(t/(\eps N)^2)^2+
  \Exp B_{\phi^{(\eps N)}}(t/(\eps N)^2)^2+\Exp C_{\phi^{(\eps N)}}(t/(\eps N)^2)^2.
\end{equs}
Thanks to Proposition \ref{p:mainAC} and $\|\phi^{(\eps N)}\|^2_{L^2(\T^2)}=(N\eps)^2\|\phi\|^2_{L^2(\R^2)}$, the terms involving $A$ and $C$ are upper bounded by a constant times $t \|\phi\|^2_{L^2(\R^2)}$ which is uniformly small in  $\eps\to0$ if $t\le t_-(\eps)\ll1$.
As for the term involving $B$, recall the definition of $\mathcal B_\phi^N(\mu)$ in~\eqref{e:Laplace}. With the same argument as in the proof of
\cite[Lemma 1]{QV} one can show that there exists a universal positive constant $C$ such that, for $t>0$ and letting $\mu_t\eqdef 1/t$, 
\begin{equs}
 \Exp[ B_\phi^N(t)^2]\le C \mathcal B_\phi^N(\mu_t).
\end{equs}
Thanks to \eqref{e:Bupper} and recalling that $\Ll^N(x,z)=L(x/N^2,z)$, we deduce that for any $\delta'>0$ there exists a constant $c_{\delta'}>0$ such that
\begin{equ}
  \Exp[ B_{\phi^{(\eps N)}}^N(t)^2]\le c_{\delta'} t\sqrt{\Ll(\eps^2/t,0)}(\log \Ll(\eps^2/t,0))^{5+\delta'}\|\phi\|^2_{L^2(\R^2)}.
\end{equ}
Choosing $\delta' < \delta$, the claim then follows from the fact that 
\begin{equ}
\lim_{\eps\to 0}t_-(\eps) \sqrt{\Ll(\eps^2/t_-(\eps),0)}(\log \Ll(\eps^2/t_-(\eps),0))^{5+\delta'}=0.	
	\end{equ}
        Finally, let us prove \eqref{e:cor2}. We choose $\mu=\mu(\eps)$ that satisfies
        $1/t_+(\eps)\ll \mu\ll 1/t_-(\eps)$ as $\eps\to0$.
        By definition,
        \[
          \cV^{\eps,N}_\phi(\mu)=\mu\int_0^{t_-(\eps)}e^{-\mu t} V^{\eps,N}_\phi(t)\dd t +\mu
            \int_{t_-(\eps)}^{t_+(\eps)}e^{-\mu t} V^{\eps,N}_\phi(t)\dd t +\mu\int_{t_+(\eps)}^\infty e^{-\mu t} V^{\eps,N}_\phi(t)\dd t\,.
          \]
          Since $ V^{\eps,N}_\phi$ is uniformly bounded and $\mu t_+(\eps)\gg1$, the third integral is negligible. The first integral is also negligible, as follows recalling \eqref{eq:nat2} and \eqref{e:cor1}. On the other hand, from \eqref{e:UBVLmodo1}, we have that $\liminf_{\eps}\liminf_N   \cV^{\eps,N}_\phi(\mu)\ge b \|\phi\|^2_{-1}$. Since the function $t\mapsto \mu e^{-\mu t}$ has integral $1$, we deduce that
          \[
\liminf_{\eps\to0}\liminf_{N\to\infty}\sup_{t\in [t_-(\eps),t_+(\eps)]} V^{\eps,N}_\phi(t)\ge b \|\phi\|^2_{-1}.
            \]
 The conclusion now follows recalling          \eqref{eq:nat2} and \eqref{e:nat3}.
\end{proof}

\subsection{Large time behaviour: proof of Theorem \ref{thm:logt}}
\label{sec:logt}

In order to prove Theorem~\ref{thm:logt}, we need a refined version of the bound obtained in~\cite[Lemma 4.3]{CES19} 
on observables of the form in~\eqref{e:BN}. 

\begin{lemma}\label{lemmaIto}
Let $\phi$ be a test function and, for $N\in\N$, let $B_\phi^N$ be defined according to~\eqref{e:BN}. 
Then, for all $t\geq 0$ the following bound holds
\begin{equ}[harrypotter]
\Exp[ B_\phi^N(t)]^2\lesssim \lambda^2 t \sum_{|k|\leq N}|\hat\phi(k)|^2 \log \Big(\frac{1}{|k/N|^2\vee N^{-2}}\Big)\,.
\end{equ}
\end{lemma}
\begin{proof}
The proof of~\eqref{harrypotter} is extremely close to that of~\cite[Lemma 4.3]{CES19} so we will adopt the same notations 
and conventions therein and we will limit ourselves to sketch the 
main steps, addressing the reader to the above mentioned reference for more details. 
Let $\cG^N$ be the solution of the Poisson equation $\gensy \cG^N(\eta)[\phi]=\lambda \cN^N(\eta)[\phi]$, which is 
explicitly given by 
\begin{equ}
\cG^N(\eta)[\phi]=\lambda \sum_{\ell,m\in\mathbb Z^2} \frac{\nonlin_{\ell,m}}{|\ell|^2+|m|^2}\hat \eta(\ell)\hat \eta(m) \hat \varphi(-\ell-m)\,.
\end{equ}
Then, defining $\energy{}$ as in~\cite[eq. (4.4)]{CES19}, a simple Gaussian computation shows that  
\begin{equ}
\E[\energy{(\cG^N(\eta)[\phi])}]= 8\lambda^2 \sum_{\ell,\,m\in\Z^2}\frac{(\nonlin_{\ell,m})^2}{|\ell|^2+|m|^2}| \hat \varphi(-\ell-m)|^2\,.
\end{equ}
The last sum can be bounded as follows
\begin{equs}
&\sum_{\ell,\,m\in\Z^2}\frac{(\nonlin_{\ell,m})^2}{|\ell|^2+|m|^2}| \hat \varphi(-\ell-m)|^2\lesssim \sum_{|k|\leq N}|\hat\phi(k)|^2\sum_{\ell+m=k}\frac{\indN{\ell,m}}{|\ell|^2+|m|^2}\\
&\lesssim \sum_{|k|\leq N}|\hat\phi(k)|^2\sum_{\substack{\ell+m=k\\ |\ell|\geq |m|}}\frac{\indN{\ell,m}}{|\ell|^2}\lesssim \sum_{|k|\leq N}|\hat\phi(k)|^2\sum_{ N\geq |\ell|> |k|/2\vee 1}\frac{1}{|\ell|^2}\\
&\lesssim \sum_{|k|\leq N}|\hat\phi(k)|^2\int_{ 1\geq |\rho|> |k/N|/2\vee N^{-2}}\frac{\dd\rho}{|\rho|^2}\lesssim \sum_{|k|\leq N}|\hat\phi(k)|^2 \log \Big(\frac{1}{|k/N|^2\vee N^{-2}}\Big)
\end{equs}
where we exploited the symmetry of the summand and the fact that if $\ell+m=k$ and $|\ell|\ge |m|$, 
then necessarily $|\ell|\ge |k|/2$. 
The conclusion follows by~\cite[Lemma 4.1]{CES19}. 
\end{proof}

We are now ready to prove Theorem~\ref{thm:logt}. 

\begin{proof}[of Theorem~\ref{thm:logt}]
Let $\phi$ be a smooth test function on $\R^2$ and 
$g\colon\R^2\to [0,1]$ be positive, smooth and such that $\int_{\R^2} g(y)\,\dd y=1$. 
For $n\in \N$ define $g_n(y)\eqdef g(y/n)$ so that $\int_{\R^2} g_n(y)\, dy= n^2$, and $\psi_n(y)\eqdef (\phi\ast g_n)(y)$. 
Throughout the proof, we will denote by $c_\phi$ a positive constant that may change from line to line 
and that will depend on $\phi$ and possibly $g$. 

Notice at first that 
\begin{equ}
H_N(t)[\psi_n]=\int_{\R^2} g_n(y)\left\{H_N(t)[\phi(\cdot-y)]-H_N(t)[\phi]\right\}\dd y+H_N(t)[\phi]n^2
\end{equ}
which implies 
\begin{equ}[fwr]
    H_N(t)[\phi]-H_N(0)[\phi]= \frac1{n^2}\left(
H_N(t)[\psi_n]-H_N(0)[\psi_n]- v^{(n)}(t)+v^{(n)}(0)
    \right)
\end{equ}
where 
\begin{equ}
v^{(n)}(t)=\int_{\mathbb R^2}g_n(y)\left(H_N(t)[\phi(\cdot-y)]-H_N(t)[\phi]\right)\dd y\,.
\end{equ}
Thanks to the scaling~\eqref{e:Scaling}, it is immediate to see that 
\begin{equ}
H_N(t)[\psi_n]-H_N(0)[\psi_n]\overset{{\rm law}}{=} h^N(t/N^2)[\psi_n^{(N)}]-h^N(0)[\psi_n^{(N)}]
\end{equ}
where $\psi_n^{(N)}$ is given as in the introduction, i.e. $\psi_n^{(N)}(\cdot)\eqdef N^2 \psi_n(N\cdot)$, 
so that we can focus on the right hand side. Applying the decomposition~\eqref{eq:ABC}, we write
\begin{equ}[abc]
h^N(t/N^2)[\psi_n^{(N)}]-h^N(0)[\psi_n^{(N)}] = A_{\psi_n^{(N)}}^N(t/N^2)+B_{\psi_n^{(N)}}^N(t/N^2)+C_{\psi_n^{(N)}}^N(t/N^2)\,.
\end{equ}
By Proposition~\ref{p:mainAC} and the fact that $\|\psi_n^{(N)}\|^2_{L^2(\T^2)}=N^2\|\psi_n\|^2_{L^2(\R^2)}$, the variances 
of $A^N$ and $C^N$ (which are centred) can be bounded by 
\begin{equ}
\Exp[A_{\psi_n^{(N)}}^N(t/N^2)]^2\lesssim t\|\psi_n\|_{L^2(\R^2)}^2\qquad \Exp[C_{\psi_n^{(N)}}^N(t/N^2)]^2\leq t\|\psi_n\|_{L^2(\R^2)}^2
\end{equ}
and, using the fact that $\|\psi_n\|^2_{L^2(\R^2)}\leq c_\phi n^2$, we conclude 
\begin{equ}[ac]
\frac{1}{n^4}\Big(\Exp[A_{\psi_n^{(N)}}^N(t/N^2)]^2\vee \Exp[C_{\psi_n^{(N)}}^N(t/N^2)]^2\Big)\leq c_\phi \frac{t}{n^2}\,.
\end{equ}
Concerning $B^N$, we exploit Lemma~\ref{lemmaIto}, which gives
\begin{equs}[e:b]
&\limsup_{N\to\infty}\Exp[B_{\psi_n^{(N)}}^N(t/N^2)]^2\lesssim \lambda^2 t \limsup_{N\to\infty}\sum_{|k/N|\leq 1} \frac{1}{N^2} \hat\psi_n(k/N)\log \Big(\frac{1}{|k/N|^2\vee N^{-2}}\Big)\\
&\lesssim \lambda^2 t \int |\hat\psi_n(p)|^2\log\left(\frac1{|p|} \right) \mathds{1}_{|p|\le 1} \dd p\leq \lambda^2 t n^2\|\phi\|^2_\infty \int |\hat{g}(p)|^2\log\left(\frac{n}{|p|} \right) \mathds{1}_{|p|\le n} \dd p\\
&\leq c_\phi \lambda^2 t n^2 \log n
\end{equs}
where we used that $\hat \psi_n(k)=\hat\phi(k)\hat g_n(k)=n^2 \hat\phi(k)\hat g(k n)$. 
Getting back to~\eqref{fwr}, it remains to control $\Exp [v^{(n)}(t)]^2=\E [v^{(n)}(0)]^2$, 
the equality being due to the stationarity of $H_N$. 
Note that 
\begin{equs} \label{b}
\E [v^{(n)}(0)]^2&= \int g_n(y)g_n(y')\\
   &\times\nonumber \E[(H_N(0)[\phi(\cdot-y)]-H_N(0)[\phi])(H_N(0)[\phi(\cdot-y')]-H_N(0)[\phi])]\dd y \dd y'\,.
\end{equs}
Thus, using the Cauchy Schwarz inequality, that $H_N$ is distributed according to a Gaussian Free Field and that $g_n$ integrates to $n^2$,
it is not hard to see that 
\begin{equ}[d]
\limsup_{N\to\infty} \E [v^{(n)}(0)]^2\leq c_\phi n^4 (\log n\vee 1)\,.
\end{equ}

We are now ready to put the bounds~\eqref{ac}~\eqref{e:b} and~\eqref{d} together and, 
suitably applying the Cauchy-Schwarz inequality to the various terms in~\eqref{fwr}, deduce
\begin{equ}
\limsup_{N\to\infty}  V_\phi^{1,N}(t)\le c_\phi\left[\frac t{n^2}+\lambda^2 t\frac{\log n\vee 1}{n^2}+\log n\vee 1\right]
\end{equ}
Therefore, choosing $n=\lceil \sqrt t\rceil$ concludes the proof. 
\end{proof}

\begin{appendix}

\section{The bulk diffusivity and the Green-Kubo formula: a heuristic}
\label{sec:heuristics}
In this section we want to provide a heuristic justification for the choice of the definition of the bulk diffusivity
given in~\eqref{e:BD}. We consider the Stochastic Burgers equation (obtained by~\eqref{e:kpz:reg} by formally setting 
$U_N\eqdef (-\Delta)^{1/2}H_N$) on the full space 
($N=\infty$) with cut-off $1$ and initial condition given by a regularised spatial white noise that is independent of $\xi$, i.e. 
\begin{equ}[e:akpz:smooth]
\partial_t U = \frac{1}{2} \Delta U
+
\lambda \cM^1[U] + (-\Delta)^{\frac{1}{2}}\Pi_a\xi\,,\qquad U(0)=\eta^a\eqdef \Pi_a\eta\,,
\end{equ}
where $a\in(1,\infty)$ and $\cM$ is defined as in~\eqref{e:nonlin}. 
Compared to~\eqref{e:AKPZ:u}, in~\eqref{e:akpz:smooth} also the space-time white noise $\xi$ is smoothened out. 
The main properties of the solution $U$ remain unaltered, and, with the 
same techniques adopted in~\cite{CES19}, it can be shown that the unique solution $U$ 
exists globally in time and is a space-time translation invariant strong Markov process with invariant measure 
$\eta^a$. The advantage of~\eqref{e:akpz:smooth} is that $U$ is smooth so that space-time point evaluation 
is allowed and well-posed. 

The bulk diffusivity serves as a way to measure the spread of the correlations of $U$ and it is classically defined 
(see for instance~\cite{BQS11} for the definition in the context of the $1$d KPZ equation) as 
\begin{equ}[e:CBD]
D^{(a)}(t)\eqdef \frac{1}{2t}\int_{\R^2} |x|^2 S(t,x)\dd x\,,
\end{equ}
where, for $t\geq 0$ and $x\in\R^2$, $S$ denotes the two-point correlation function
\begin{equ}[e:Corr]
S(t,x)\eqdef \Exp[U(t,x) U(0,0)]\,.
\end{equ}
See for instance \cite[Ch. II.2.2]{spohn2012large} for the analogous definition for
interacting particle systems 
(we put the prefactor $1/2$ simply to ensure that the diffusion
coefficient of the linear equation is $1$ and, with respect to the
interacting particle system references, we omit a prefactor related to
the so-called ``compressibility''). We now want to formally manipulate the expression on the right hand side 
of~\eqref{e:CBD} in order to connect it 
to~\eqref{e:BD}. Note that if $\lambda=0$, $S(t,\cdot)$ is explicit 
and it can be easily shown to integrate to $1$ and 
to decay at $\infty$ exponentially fast. For the purpose of this section, 
we will \emph{assume} that $S(t,x)$ decays fast (say, faster than $1/|x|^2$)
for $|x|\to\infty$ also for $t>0$. Using integration by parts and
that $\cM^1[U]= (-\Delta)^{1/2}\mathcal \cN^1[U]$, one then sees
that $S(t,\cdot)$ also integrates to $1$.
%\begin{enumerate}[noitemsep, label=(a\arabic*)]
%\item\label{i:a1} for all $t> 0$, $\int S(t,x)\dd x=1$\,
%\item\label{i:a2} for all $t> 0$ and $|x|\gg t$, $|S(t,x)|\lesssim |x|^{-p}$ for some $p>2$\,.
%\end{enumerate} 
Now, upon integrating~\eqref{e:akpz:smooth} in time and plugging it into the definition of $S$ we see that 
\begin{equ}
S(t,x)= S(0,x)+\half\int_0^t \Delta S(s,x)\dd s+\lambda \int_0^t \Exp[\cM^1[U](s,x) u(0,0)]\dd s\,,
\end{equ}
where the term containing the noise drops out because the initial condition is independent of $\xi$.
Since $\int S(t,x)\dd x=1$ and $|S(t,\cdot)|$ decays sufficiently fast, a simple 
integration by parts gives
\begin{equ}
\frac14 \int_0^t\int |x|^2  \Delta S(s,x)\dd x\dd s={ t}\,.
\end{equ}
For the term containing the nonlinearity instead, recall $(-\Delta)^{\tfrac{1}{2}}\cN^1[U]=\cM^1[U]$. 
Then, integrating once more by parts, we get
\begin{equs}
\half \int |x|^2 \Exp[\cM^1[U](s,x) &U(0,0)]\dd x\\
&=-\half\int (-\Delta)^{\tfrac12} |x|^2 \Exp[\cN^1[U](s,x) (U(0,0)-U(s,0))]\dd x\\
&=-\half\int (-\Delta)^{\tfrac12} |x|^2 \E[\cN^1[\eta^a](0) \tilde \Exp_{\eta^a}[\tilde U(s,x)-\tilde U(0,x)]]\dd x
\end{equs}
where the first passage is a consequence of the fact that $U(s)$ is distributed according to $\eta^a$, the latter is Gaussian 
and $\cN^1$ is quadratic, while for the second we further exploited the space-time translation invariance of $U$ and 
denoted by $\tilde\Exp_{\eta^a}$ the expectation with respect to the process $\tilde U$ starting from $\eta^a$ and 
running backward in time, i.e. $\tilde U(r,\cdot)=U(s-r,\cdot)$. We point out that $\tilde U$ has the 
same properties as $U$ and solves~\eqref{e:akpz:smooth} but with $-\lambda$ replacing $\lambda$. 
Therefore, arguing as above, we write 
\begin{equs}
\E[\cN^1[\eta^a](0)\tilde \Exp_{\eta^a}[\tilde U(s,x)-&\tilde U(0,x)]]\\
&=\int_0^s  \E[\cN^1[\eta^a](0)\tilde\Exp_{\eta^a}[\Delta\tilde U(r,x)-\lambda\cM^1[\tilde U](r,x)]]\dd r
\end{equs}
so that, integrating against $\half(-\Delta)^{\tfrac12} |x|^2$, we see that the summand containing the Laplacian 
vanishes while the other becomes
\begin{equs}
2\lambda \int_0^s \int \E[\cN^1[\eta^a](0)&\tilde \Exp_{\eta^a}[\cN^1[\tilde U](r,x)]]\dd x\dd r\\
&=2\lambda \int_0^s \int \E[\cN^1[\eta^a](0) e^{r\tilde\cL}\cN^1[\eta^a](x)]\dd x\dd r
\end{equs}
with $\tilde \cL$ the generator of the time reversed process.
In conclusion, $ D^{(a)}_\bulk$ can be rewritten as 
\begin{equs}[e:ModBD]
D^{(a)}(t)=&\frac{1}{2t}\int|x|^2 S(0,x)\dd x +1 \\
&+2\frac{\lambda^2}{t}\int_0^t\int_0^s\int \E[\cN^1[\eta^a](0) e^{r\tilde\cL}\cN^1[\eta^a](x)]\dd x\dd r\dd s\,.
\end{equs}
If we let $a\to\infty$, $\eta^a$ converges to a spatial white noise so
that the first term vanishes and \eqref{e:ModBD} reduces to
\begin{equs}[e:ModBDfin]
D(t)=1
+2\frac{\lambda^2}{t}\int_0^t\int_0^s\int \Exp[\tilde \cN^1[H](0,0) \tilde
\cN^1[H](r,x)]\dd x\dd r\dd s\,
\end{equs}
where we used the relation between $\cN^1$ and $\tilde \cN^1$, see \eqref{e:BN}.
Now, in case $\lambda=0$, we recover the well-known result concerning the bulk diffusivity 
of the linear stochastic heat equation, which is constant in time. On the other hand, for $\lambda>0$, 
taking $N\to\infty$, the bulk diffusivity $D_N(t)$ defined in~\eqref{e:BD} formally 
converges to $D(t)$ given as in~\eqref{e:ModBDfin}.

\section{Mode-coupling and $\sqrt{\log t}$ superdiffusivity}
\label{app:aheuri}
The ansatz and the calculations in this section are inspired by~\cite{BKS85} and Appendix C.2 of~\cite{S14}. 
As in the previous appendix we will work with the solution $U$ (and its Fourier transform $\hat U$) 
of~\eqref{e:akpz:smooth} on the full space and 
non-regularised noise, i.e. $a=\infty$, and we start from a white noise initial condition $\eta=\eta^\infty$. 
Let $\hat S$  be the Fourier transform of the two-point correlation 
function $S$ in~\eqref{e:Corr}, which by translation invariance is given by 
$\hat S(t,k)=\frac{1}{(2\pi)^2}\Exp[\hat U(t,k) \hat U(0,-k)]$. 
Formally $\hat S$ solves 
\begin{equs}\label{eq:Sequation}
\partial_t \hat S(t,k) + \frac12 |k|^2 \hat S(t,k) &= \frac{\lambda}{(2\pi)^2} \Exp[\hat U(0,-k) \cM_k^1(U(t))]	\\
&= \frac{\lambda}{(2\pi)^2} \E[\hat \eta(-k) e^{\cL t}\cM_k^1(\eta)]\,.
	\end{equs}
The generator $\cL$ of the Markov process $U$ can be written as the sum of $\gensy$ and $\cA$, whose definition 
can be read off~\eqref{e:gens:fock}-\eqref{e:genam:fock} (the variables now take values in $\R^2$ instead of $\Z^2$ and the sum is replaced by an integral) and 
whose properties are analogous to those in Lemma~\ref{lem:generator}.  
The semigroup associated to $\cL$ satisfies
\begin{equ}
e^{\cL t}= e^{\cL_0 t} + \int_0^t e^{\cL_0 (t-s)}\cA \,e^{\cL}\, \dd s \,.
	\end{equ}
Moreover, 
\begin{equ}
\E[\hat \eta(-k)e^{\cL_0 t}\cM_k^1(\eta)]=0\,	
	\end{equ}
which follows since $e^{\cL_0 t}$ corresponds to taking expectation
with respect to the Ornstein-Uhlenbeck process and $\cM_k^1$ is quadratic.
Getting back to~\eqref{eq:Sequation}, since the adjoint of $\cA$ is $-\cA$,
the term on the right hand side equals
\begin{equ}
 -\frac{\lambda}{(2\pi)^2} \int_0^te^{-\tfrac12 |k|^2(t-s)}\E[(\cA \hat \eta)({-k})e^{\cL s}\cM_k^1(\eta)]\, \dd s\, .
\end{equ}
Using that $\cA\hat  \eta ({-k})=\lambda \cM_{-k}^1(\eta)$, and the Fourier representation \eqref{e:nonlinF}
of the non-linearity  (with sums replaced by integrals) we see that the above equals
\begin{equs}\label{eq:Sintegral}
-\frac{\lambda^2}{(2\pi)^4}|k|^2&\int_0^t e^{-\tfrac12|k|^2(t-s)}\int\dd\ell\int\dd\ell'
\cK_{\ell,k-\ell}^1\cK_{\ell',-k-\ell'}^1\\
&\times\Exp[\hat U (s,\ell) \hat U(s,k-\ell) \hat U(0,\ell') \hat U(0,-k-\ell')]\,\dd s.
\end{equs}
Now, the ``mode-coupling
approximation'' (see e.g. \cite{S14}) consists in doing a Gaussian
approximation of the average of the product of four $U$ variables, which allows to apply Wick's rule. 
By translation invariance $\E[\hat U(s,\ell) \hat U(0,m)]=0$ unless
$\ell=-m$. Note also that the Wick contraction $\E[\hat U(s,\ell) \hat U(s,k-\ell)]\E[\hat U(0,\ell') \hat U(0,-k-\ell')]$ can be ignored because it vanishes unless $k=0$,  in which case however~\eqref{eq:Sintegral} is multiplied by $|k|^2=0$. 
Therefore, summing up the above computations and considerations we see that
\begin{equs}\label{eq:approxeqS}
\Big(\partial_t +\frac12 &|k|^2\Big)\hat S(t,k)\\
&\approx -2|k|^2\frac{\lambda^2}{(2\pi)^4} \int_0^t e^{-\tfrac12 |k|^2(t-s)}\int\dd\ell (\cK_{\ell,k-\ell}^1)^2\hat S(s,\ell)\hat S(s,k-\ell)\, \dd s\,.
	\end{equs}
We now make the ansatz 
\begin{equ}[e:ansatz]
\hat S(t,k) = \hat S(0,0)e^{-\tfrac12 |k|^2t
  -c|k|^2t(\log t)^\delta}
  \end{equ}
for small $k$ and large $t$, corresponding to a diffusion coefficient of order $(\log t)^\delta$. 
Our goal is to determine $\delta$ such that the left and right hand sides above coincide. 
According to~\eqref{e:ansatz}, in the regime considered, the left hand side of~\eqref{eq:approxeqS} equals
\begin{equation}\label{eq:lhs}
\Big(\partial_t +\frac12 |k|^2\Big)\hat S(t,k) \approx -c|k|^2 (\log t)^\delta \hat S(0,0)\,.
\end{equation}
Regarding the right hand side instead, we approximate $e^{-\tfrac12 |k|^2 (t-s)}$ by one and 
$k-\ell$ by $-\ell$ so that for $k\to 0$ and 
$t\to \infty$, it gives
\begin{equ}\label{eq:rhs}
-|k|^2\lambda^2 \int_0^t \dd s\int \dd \ell (\cK_{\ell,-\ell}^N)^2 e^{-2c|\ell|^2s(\log s)^\delta}	\approx -|k|^2\lambda^2 (\log t)^{1-\delta},
	\end{equ}
where we used the explicit form of 
$\cK_{\ell,-\ell}^N$ as in \eqref{e:NonlinPolar}.
Equating~\eqref{eq:lhs} and~\eqref{eq:rhs} yields that $\delta=\frac12$ as desired.

\section{Some technical results}\label{a:Technical}

In this section we will state and prove some technical bounds that are needed in Section~\ref{sec:iteration}. 

\begin{lemma}\label{l:lm}
	For any $\ell,m\in\Z^2$ and $k_{1:n}\in(\Z^2)^n$ such that $\ell+m=k_1$ we have 
	\begin{equation}\label{e:lm}
	\frac{1}{4}(|\ell|^2+|k_{1:n}|^2)\leq |\ell|^2+|m|^2+|k_{2:n}|^2\leq 4(|\ell|^2+|k_{1:n}|^2)\,.
	\end{equation}	
\end{lemma}
\begin{proof}
	The proof is an application of the triangular inequality, we omit the details.	
\end{proof}

%% Recall the definition of the sector $\CC^\theta_k$ in~\eqref{e:sector}. One then has the following  result.
%% \begin{lemma}\label{l:sector}
%% For any $\theta>0$ sufficiently small, any  $\ell\in\CC^\theta_k$ and $m=k-\ell$, there exists ${\delta}_\theta>0$ such that 
%% \begin{equation}\label{eq:sectorinequ}
%% \frac{|c(\ell,m)|}{|\ell||m|}\geq \sqrt{\delta_\theta}>0\,.
%% \end{equation}
%%  Moreover, let $\varphi_\theta = \sup\{\varphi\in(0,2\pi):\, \theta |\cos\varphi|\geq |\sin\varphi|\, \forall\, \varphi\in [0,\varphi_\theta]\}$, then the sector $\CC^\theta_k$ can be written as 
%%  \begin{equation}\label{e:sectorpolar}
%%   \CC^\theta_k=\{\ell=\varrho(\cos\varphi,\sin\varphi)\,: 2|k|\leq\varrho\leq N/3\text{ and }\varphi\in[0,\varphi_\theta]\cup[2\pi-\varphi_\theta,2\pi]\}\,.
%%  \end{equation}
%% \end{lemma}
%% %\giuseppeText{A simple observation. In order to get a lower bound on the quantity above simply notice the following. 
%% %Let $\ell=|\ell|(\cos\theta_\ell,\sin\theta_\ell)$ and let $\theta_m$ be defined similarly. Then
%% %\begin{equ}
%% %c(\ell,m)=-|\ell||m|(\cos\theta_\ell\cos\theta_m-\sin\theta_\ell\sin\theta_m)=-|\ell||m| \cos(\theta_\ell+\theta_m)
%% %\end{equ}
%% %which implies 
%% %\begin{equ}
%% %\frac{|c(\ell,m)|}{|\ell||m|}=|\cos(\theta_\ell+\theta_m)|\,.
%% %\end{equ}
%% %}

%% \begin{proof}
%% The proof of~\eqref{eq:sectorinequ} is based on various applications of the triangular inequality. We omit the details. %The last part of the statement is a direct consequence of the definition of the sector.  %Fix $\ell\in \CC^\theta_k$ and let $m=k-\ell$. We 
%% \end{proof}

In the following lemma, which is used  in Lemmas~\ref{l:UBtoLB} and~\ref{l:LBtoUB} we analyse the functions $\Ll,\,\LB,\,$ and $\UB$ introduced in~\eqref{e:L} and~\eqref{e:LUBk} 
respectively.

\begin{lemma}\label{l:MainIntegrals}
For $k\in\N$, let $\Ll$, $\LB_{k}$ and $\UB_{k}$ be the functions on $\R_+\times[1,\infty)$ 
defined in~\eqref{e:L} and~\eqref{e:LUBk}. 
Then, $\Ll$, $\LB_{k}$ and $\UB_{k}$ are monotonically decreasing in the first variable and increasing in the second.
For any $x>0$ and $z\geq 1$, we have that $\LB_{k}(x,z),\,\UB_{k}(x,z)\geq 1$ and the following inequalities hold
\begin{gather}
1\leq \LB_{k}(x,z)\leq \sqrt{\Ll(x,z)}\,,\label{e:LBExp}\\
1\vee \lambda \sqrt{z}\leq\sqrt{\Ll(x,z)}\leq\UB_{k}(x,z)\leq \Ll(x,z)\,.\label{e:UBExp}
\end{gather}
Moreover, for any $0<a<b$, we have 
\begin{equation}\label{e:IntUBtoLB}
	\lambda^2\int_a^b \frac{\dd x}{(x^2+x)\UB_{k}(x,z)}= 2\left[\LB_{k+1}(a,z)-\LB_{k+1}(b,z)\right]
\end{equation}
and 
\begin{equation}\label{e:IntLBtoUB}
	\lambda^2\int_a^b \frac{\dd x}{(x^2+x)\LB_{k}(x,z)}\leq 2 \left[\UB_{k}(a,z)-\UB_{k}(b,z)\right].
      \end{equation}
      Finally, one has
       \begin{equ}
   \label{e:addass}
  |\partial_x(x F(x,z))|= |F(x,z)+x\partial_x F(x,z)|\le (1+\lambda^2) F(x,z) \;\text{for every }\;
   x\ge0,
 \end{equ}
      when $F$ is either $\LB_k$ or $\UB_k$.
\end{lemma}
\begin{proof}
The two chains of inequalities in~\eqref{e:LBExp} and~\eqref{e:UBExp} 
are a direct consequence of the respective definitions and Taylor's approximation. A computation of the partial derivative with respect to the second variable yields the desired monotonicity.
Furthermore we have that 
\begin{equ}[e:Der1]
\partial_x\Ll(x,z)=-\frac{\lambda^2}{x^2+x},\,\qquad\partial_x\LB_{k}(x,z)=-\frac{\lambda^2}2\frac{\LB_{k-1}(x,z)}{(x^2+x)\Ll(x,z)}
\end{equ}
and 
	\begin{equs}[e:Der2]
		\partial_x\UB_{k}(x,z)&=-\lambda^2\frac{\LB_{k}(x,z)-\frac12\LB_{k-1}(x,z)}{(x^2+x)(\LB_{k}(x,z))^2}\\
		&=-\frac{\lambda^2}{2(x^2+x)\LB_{k}(x,z)}\Big[1+\frac{\frac{(\frac12\log \Ll(x,z))^k}{k!}}{\LB_{k}(x,z)}\Big]\,,
	\end{equs}
	which are all strictly negative for any $x>0$ and $z\geq 1$. 
The above computation of the partial derivatives moreover reveals that
	\begin{equs}
		\lambda^2\int_a^b \frac{\dd x}{(x^2+x)\UB_{k}(x,z)} = 2\int_b^a \partial_x\LB_{k+1}(x,z)\dd x
		&=2\left[\LB_{k+1}(a,z)-\LB_{k+1}(b,z)\right]\,,
	\end{equs}
	which is~\eqref{e:IntUBtoLB}. 
	For~\eqref{e:IntLBtoUB}, notice that 
	\begin{equs}
		&\lambda^2\int_{a}^{b}\frac{\dd x}{(x^2+x)\LB_{k}(x,z)}
		=\int_{b}^{a}\partial_x\UB_{k}(x,z)\dd x +\frac{\lambda^2}2\int_{a}^{b}\frac{ \LB_{k-1}(x,z)}{(x^2+x)\LB_{k}(x,z)^2}\dd x \\
		&\leq \int_{b}^{a}\partial_x\UB_{k}(x,z)\dd x  +\frac{\lambda^2}2\int_{a}^{b} \frac{ 1}{(x^2+x)\LB_{k}(x,z)}\dd x,
	\end{equs}
where the last inequality follows from the fact that all the terms are positive and 
for all $x$ we have $\LB_{k-1}(x,z)\leq \LB_{k}(x,z)$. 
Bringing the last term to the left hand side gives the required estimate.
%and obtain
%	\begin{equation}
%	\lambda^2\int_{a}^{b}\frac{\dd x}{(x^2+x)\LB_{k}(x,z)}\leq 2\int_{b}^{a}\partial_x\UB_{k}(x,z)\dd x
%	=2 \left[\UB_{k}(a,z)-\UB_{k}(b,z)\right],
%	\end{equation}
%	which concludes the proof.

Finally, \eqref{e:addass} follows immediately from \eqref{e:Der1}-\eqref{e:Der2}, recalling hat $\Ll(x,z)\ge1$.
\end{proof}

\begin{remark}
For notational convenience, the next three lemmas are formulated for a generic function $F$ satisfying Assumption \ref{ass:FG} below. In practice, we will always apply the results when $F(\cdot,z)$ is of the form  $a+b \UB_k(\cdot,z)$ or $a+b \LB_k(\cdot,z)$, for some positive constants $a,b$, possibly depending on $k$ and on $z$. In this case, the validity of  the assumption follows from the definition of $\UB_k$, $\LB_k$ and from Lemma \ref{l:MainIntegrals} above.
  
\end{remark}
\begin{assump}
  \label{ass:FG}
  	 $F=F(x,z)$ is a function on $\R_+\times [1,\infty)$
	monotonically decreasing in the first variable and such that for all $(x,z)\in\R_+\times [1,\infty)$, $F(x,z)\geq 1$. 
	We assume further that the function $G=G(x,z)$ given by 
	\begin{equ}[e:FG]
		G(x,z)=\frac{\Ll(x,z)}{F(x,z)}\,,
	\end{equ}
	where $\Ll$ is defined as in~\eqref{e:L}, is also monotonically decreasing in the first variable and satisfies 
	$G(x,z)\geq 1$ for all $(x,z)\in\R_+\times [1,\infty)$.
        Finally, we assume that \eqref{e:addass} holds.
\end{assump}

\begin{lemma}\label{l:OffDiag} Under Assumption \ref{ass:FG},
	 there exists $K>0$ (independent of $F$) such that 
	\begin{equ}[e:IntOffDiag]
		\int_0^\infty\frac{\dd\rho}{(\rho^2+\alpha)F(\rho^2+\alpha,z)}\leq \frac{K}{\sqrt{\alpha}} \frac{1}{F(2\alpha,z)}\,,
	\end{equ}
	for all $\alpha>0$, $\lambda>0$ and $z\ge1$. 
      \end{lemma}
      
      \begin{proof}
        We write the integral on the left hand side of~\eqref{e:IntOffDiag} as the sum of $I_1(\alpha,z)$ and $I_2(\alpha,z)$, where
	\begin{equ}
	I_1(\alpha,z)= \int_0^{\sqrt{\alpha}}\frac{\dd \rho}{(\rho^2+\alpha)F(\rho^2+\alpha,z)}\,,\qquad
	I_2(\alpha,z)= \int_{\sqrt{\alpha}}^\infty\frac{\dd \rho}{(\rho^2+\alpha)F(\rho^2+\alpha,z)}\,.
	\end{equ}
	For $I_1$, we use monotonicity of $F$ w.r.t. its first argument to write
        \begin{equ}
	  I_1(\alpha,z)\le \frac1{F(2\alpha,z)}\frac{\sqrt{ \alpha}}\alpha=\frac1{\sqrt{\alpha}F(2\alpha,z)}.
        \end{equ}
	Using~\eqref{e:FG}, and the fact that  $\Ll$ is decreasing w.r.t. its first argument, $I_2$ can be written as 
	\begin{equs}
		I_2(\alpha,z)&=\int_{\sqrt{\alpha}}^\infty\frac{G(\rho^2+\alpha,z)}{(\rho^2+\alpha)\Ll(\rho^2+\alpha,z)}\dd\rho\leq
                G(2\alpha,z)\int_{\sqrt{\alpha}}^\infty\frac{\dd\rho}{\rho^2\Ll(2\rho^2,z)}
		%% %&=\frac{G(2\alpha,z)}{\sqrt{\alpha}}\int_1^{\frac{a}{\sqrt{\alpha}}}\frac{\dd \rho}{(\rho^2+1)\Ll(\alpha(\rho^2+1),z)}\\
		%% &=\frac{G(2\alpha,z)}{2\sqrt{\alpha}}\int_2^{\frac{a^2}{\alpha}+1}\frac{1}{u\sqrt{u-1}[1+\lambda^2(z+\log(1+\frac{1}{\alpha u}))]}\dd u\\
		%% &\leq \frac{G(2\alpha,z)}{\sqrt{\alpha}}\int_1^{\frac{2a^2}{\alpha}}\frac{1}{u^{\frac{3}{2}}[1+\lambda^2(z+\log(1+\frac{1}{\alpha u}))]}\dd u%=\frac{G(2\alpha,z)}{\sqrt{\alpha}}\int_1^{\frac{2a^2}{\alpha}}\frac{1}{\rho^{\frac{3}{2}}[1+\lambda^2(z+\log\frac{1}{\alpha}-\log\rho)]}\dd\rho,
	\end{equs}
        and it remains to prove that
        \begin{equs}
          \label{e:itrema}
\int_{\sqrt{\alpha}}^\infty\frac{\dd\rho}{\rho^2\Ll(2\rho^2,z)}\le \frac K{\sqrt{\alpha}\Ll(2\alpha,z)}.
        \end{equs}
        If $\alpha\ge1$, recalling $z\ge 1$, we simply bound
        \[
\Ll(2\rho^2,z)=1+\lambda^2(z+\log(1+1/(2\rho^2))\gtrsim \Ll(2\alpha,z)
        \]
        and the desired estimate immediately follows. If instead $\alpha\le 1$,
        we split the integral as
        \begin{equs}
\int_{\sqrt{\alpha}}^\infty\frac{\dd\rho}{\rho^2\Ll(2\rho^2,z)}=I_3(\alpha,z)+I_4(\alpha,z):=\int_{\sqrt{\alpha}}^{\alpha^{1/4}}\frac{\dd\rho}{\rho^2\Ll(2\rho^2,z)}+\int_{{\alpha^{1/4}}}^\infty\frac{\dd\rho}{\rho^2\Ll(2\rho^2,z)}.
        \end{equs}
        For $I_3$ we simply use $\Ll(2\rho^2,z)\ge\Ll(2\sqrt\alpha,z)\gtrsim \Ll(2\alpha,z)$ and then it is upper bounded as the r.h.s. of \eqref{e:itrema}. For $I_4$, instead, we use $\Ll(2\rho^2,z)\ge \Ll(\infty,z)$ so that
        \begin{equs}
I_4(\alpha,z)\lesssim \frac1{\alpha^{1/4}(1+\lambda^2 z)}\lesssim \frac1{\sqrt{\alpha}(1+\lambda^2 (z+\log(1+1/(2\alpha)))}
        \end{equs}
        as desired.
        Putting everything together, \eqref{e:IntOffDiag} follows. 
\end{proof}

%In order to cure the dependence on $\mu$ we will also use the following lemma
%
%\begin{lemma}\label{l:lambda}
%Let $f$ be a function greater or equal than 1, then 
%\begin{equ}
%\int_0^1\frac{x\dd x}{\mu+x^2 f(x^2+\mu)}\leq 1+\int_0^1\frac{x \dd x}{(\mu+x^2) f(x^2+\mu)}
%\end{equ}
%\end{lemma}
%\begin{proof}
%We can write
%\begin{equs}
%\int_0^1&\Big(\frac{1}{\mu+x^2 f(x^2+\mu)}-\frac{1}{(\mu+x^2) f(x^2+\mu)}\Big)x\dd x\\
%&=\mu\int_0^1 \frac{ f(x^2+\mu)-1}{(\mu+x^2 f(x^2+\mu))(\mu+x^2) f(x^2+\mu)}x\dd x\\
%&\leq \mu\int_0^1\frac{x}{(\mu+x^2)^2}\dd x \leq \frac{\mu}{2}\int_\mu^\infty\frac{\dd x}{x^2}\leq \half,
%\end{equs}
%thus the claim follows.
%\end{proof}
%

% To lighten notations, we also set
% \begin{equs}
%   \label{eq:leggerissimo}
%   \mu_N\eqdef\mu/N^2, \quad \beta_N\eqdef |k_{2:n}/N|^2, \quad \alpha_N\eqdef \mu_N+\frac12|k_{1:n}/N|^2\\
%   , \quad A^2(x)\eqdef \frac12(|x|^2+|k_1/N-x|^2+\beta_N).
% \end{equs}

\begin{lemma}\label{l:Approx}
Under Assumption \ref{ass:FG}, define $\Gamma(\ell,m,k_{2:n})$ as in \eqref{e:Conv1},
\begin{equ}[e:Conv]
\alpha=\alpha(\mu,k_{1:n})\eqdef\mu+\tfrac12|k_{1:n}|^2\,,\qquad\text{and}\qquad \alpha_N\eqdef \alpha/N^2
\end{equ} 
and $F^N(\cdot,z):=F(\cdot/N^2,z)$, $G^N(\cdot,z):=G(\cdot/N^2,z)$. Then, 
 there exists a positive constant $K$ (depending only on $\lambda$) such that
 \begin{equs}\label{e:lemmaC4nuovo}
  \Big|&\sum_{\ell+m=k_1}\frac{(\nonlin_{\ell,m})^2}{\mu+\Gamma(\ell,m,k_{2:n})F^N(\mu+\Gamma(\ell,m,k_{2:n}),z)}\\
 &-
     \sum_\ell\frac{(\nonlin_{\ell,-\ell})^2}{(|\ell|^2+\alpha)(1+|\ell/N|^2+\alpha_N^2)F^N(|\ell|^2+\alpha,z)}
    \Big|\le K\frac{G^N((\mu+\frac12|k_{1:n}|^2)\vee 1,z)}{\lambda \sqrt z}.
\end{equs}

  \end{lemma}

  \begin{remark}
Actually, the right hand side of \eqref{e:lemmaC4nuovo} could be replaced by a constant depending only on $z$. 
However, it is convenient to have the bound in this form since in the iteration, the terms giving the main contribution 
will be upper or lower bounded by a quantity like $G^N(\mu+\frac12|k_{1:n}|^2,z)$ 
in which case~\eqref{e:lemmaC4nuovo} (with $z$ being taken suitably large) will be regarded as an error term.
  \end{remark}
  
  \begin{proof}
    We proceed in three steps, starting from the first sum in \eqref{e:lemmaC4nuovo}:
    \begin{enumerate}[noitemsep]
    \item  first, we replace the denominator by $ [\mu+\Gamma(\ell,m,k_{2:n})]F^N(\mu+\Gamma(\ell,m,k_{2:n}),z)$,
      
    \item then, we replace the denominator by $(|\ell|^2+\alpha)(1+|\ell/N|^2+\alpha_N^2)F^N(|\ell|^2+\alpha,z)$;

    \item finally, we replace       $\nonlin_{\ell,m}$ with $\nonlin_{\ell,-\ell}$

    % \item finally,  we replace in the denominator  $A(\ell,k_{1:n})F^N(A(\ell,k_{1:n}),z)$ with $A(\ell,k_{1:n})(1+A(\ell,k_{1:n})/N^2)F^N(A(\ell,k_{1:n}),z)$
    \end{enumerate}
and it will turn out that each step produces an error term of the same form as the one 
at the right hand side of \eqref{e:lemmaC4nuovo}.
\medskip

\noindent{\bf Step 1.} Since $|\nonlin_{\ell,m}|\le 1$, it is enough to bound
\begin{equs}[eq:t1]
\sum_{\ell+m=k_1}&\Big|\frac{1}{\mu+\Gamma(\ell,m,k_{2:n})F^N(\mu+\Gamma(\ell,m,k_{2:n}),z)}\\
&-\frac1{ [\mu+\Gamma(\ell,m,k_{2:n})]F^N(\mu+\Gamma(\ell,m,k_{2:n}),z)}
      \Big|.
    \end{equs} 
Using that $|F^N-1|/F^N\le 1$, the sum is upper bounded by
    \begin{equs}
      \mu\sum_{\ell+m=k_1} \frac{1}{[\mu+\Gamma(\ell,m,k_{2:n})F^N(\mu+\Gamma(\ell,m,k_{2:n}),z)](\mu+\Gamma(\ell,m,k_{2:n}))}.
    \end{equs}
We split $\Z^2$ into  $\Omega_1=\{\ell: \mu+\Gamma(\ell,k_1-\ell,k_{2:n})\leq\alpha\vee 1\}$, 
for $\alpha$ defined as in~\eqref{e:Conv}, 
and $\Omega_2=\Z^2\setminus \Omega_1$. 
In the first case $F^N(\mu+\Gamma(\ell,k_1-\ell,k_{2:n}),z)\geq F^N(\alpha\vee 1,z)$ 
and thus we obtain the upper bound
  \begin{equation}\label{e:Step1Int}
  \begin{aligned}
    &\frac{\mu}{F^N(\alpha\vee 1,z)}\sum_{\substack{\ell+m=k_1\\ \ell\in \Omega_1}}\frac1{[\frac{\mu}{(F(\alpha\vee 1,z)}+\Gamma(\ell,m,k_{2:n})](\mu+\Gamma(\ell,m,k_{2:n}))}\\
    &\le \frac{\mu}{F^N(\alpha\vee 1,z)}
      \Big(\sum_\ell\frac1{(\frac{\mu}{F(\alpha\vee 1,z)}+\frac12|\ell|^2)^2}\Big)^\frac12
      \Big(\sum_\ell\frac1{(\mu+\frac12|\ell|^2)^2}\Big)^\frac12\\
    &\lesssim\frac{\mu}{F^N(\alpha\vee 1,z)}\frac{\sqrt{F^N(\alpha\vee 1,z)}}{\mu}=\frac{1}{\sqrt{F^N(\alpha\vee 1,z)}}
      \le  \frac{G^N((\mu+\frac12|k_{1:n}|^2)\vee 1,z)}{\lambda\sqrt{z}} 
      \end{aligned}
      \end{equation}
    where we used the relation between $F$ and $G$, the assumption $G\ge 1$ and the bound $\Ll(x,z)\ge \lambda^2 z$.
    For $\ell\in\Omega_2$, note that 
\[
F^N(\mu+\Gamma(\ell,m,k_{2:n}),z)= \frac{\Ll^N(\mu+\Gamma(\ell,m,k_{2:n}),z)}{G^N(\mu+\Gamma(\ell,m,k_{2:n}),z)}\geq\frac{z\lambda^2}{G^N((\mu+\frac12|k_{1:n}|^2)\vee 1,z)}.
\]
We can proceed as in~\eqref{e:Step1Int} but with the right hand side above in place of $F^N$. 
Hence, also the sum over $\Omega_2$ is upper bounded by 
  \begin{equ}
    \label{e:altot1}
 \frac{G^N((\mu+\frac12|k_{1:n}|^2)\vee 1,z)}{\lambda\sqrt{z}}\,.
  \end{equ}
and consequently so is~\eqref{eq:t1}. 
\medskip

\noindent{\bf Step 2.} At first, we bound $|\nonlin_{\ell,m}|$ by the indicator function of $|\ell|\leq N$. 
The quantity we have to control takes the form (we write for lightness of notation $\Gamma$ instead of $\Gamma(\ell,m,k_{2:n})$ and $A$ instead of $A(\ell,k_{1:n})$)
      \begin{equs}
          \label{eq:t3}
      % \sum_{\ell+m=k_1}\left|\frac{1}{(\mu+\Gamma)F^N(\mu+\Gamma,z)}-\frac1{A F^N(A,z)}
      % \right|=
          &\sum_{\substack{\ell+m=k_1\\ |\ell|\le N}}\Big|
\frac1{(\mu+\Gamma)F^N(\mu+\Gamma,z)}-\frac1{(|\ell|^2+\alpha) F^N(|\ell|^2+\alpha,z)}\\
&+\frac1{(|\ell|^2+\alpha) F^N(|\ell|^2+\alpha,z)}-\frac1{(|\ell|^2+\alpha)(1+|\ell/N|^2+\alpha_N)F^N(|\ell|^2+\alpha,z)}
            \Big|
          \le \one +\two\,,
    \end{equs}
where, setting $H(x)\eqdef x F^N(x,z)$,
\begin{equs}
\one &= \sum_{\substack{\ell+m=k_1\\ |\ell|\le N}}\left|\frac{H(|\ell|^2+\alpha)-H(\mu+\Gamma)}{(\mu+\Gamma)(|\ell|^2+\alpha)F^N(\mu+\Gamma,z)F^N(|\ell|^2+\alpha,z)}
\right|\,,\\
\two &= \frac1{N^2}\sum_{\substack{\ell+m=k_1\\ |\ell|\le N}}\frac1{F^N(|\ell|^2+\alpha,z)}\,,	
\end{equs}
Recalling assumption \eqref{e:addass} and that $F(\cdot,z)$ is decreasing, for any $a,b\in\R$ we have
  \[
  |H(a)-H(b)|\le (1+\lambda^2) F^N(a\wedge b,z)|a-b|
  \]
  which we apply in the sum, with $a=|\ell|^2+\alpha,\, b=\mu+\Gamma$, so that
  $|a-b|\le |\ell| |k_1|$. Also, with the same choice of 
  $a,b$, by~\eqref{e:lm}, $a\vee b \le 4a$, so that by the monotonicity of $F^N$, one has
  \[
  \frac{F^N(a\wedge b,z)}{F^N(a,z)F^N(b,z)}\le \frac1{F^N(a\vee b,z)}\le \frac1{F^N(4a,z)}\,.
  \]
  In conclusion, invoking~\eqref{e:lm} once more, $\one$ can be upper bounded as
  \begin{equ}
\one\lesssim |k_1|\sum_{\ell}\frac{|\ell|}{(|\ell|^2+\alpha)^2 F^N(4(|\ell|^2+\alpha),z)}\lesssim|k_1/N|\int_0^\infty\frac{\dd \rho}{(\rho^2+4\alpha_N)F(\rho^2+4\alpha_N,z)}\,.
  \end{equ}
  Using Lemma \ref{l:OffDiag}, the latter expression is bounded by
  \begin{equ}
    \frac{|k_1/N|}{\sqrt{\alpha_N}F(8\alpha_N,z)}\lesssim \frac1{F^N(8(\alpha\vee1),z)}\le \frac{G^N((
      \mu+\frac12|k_{1:n}|^2)\vee1,z)}{\lambda^2 z}\,.
  \end{equ}
  which gives the correct bound on $\one$. 
  As for $\two$, the monotonicity of $F$ and $G$,~\eqref{e:FG} and~\eqref{e:UBExp} give
  \begin{equs}
\frac1{N^2}\sum_{\substack{\ell+m=k_1\\ |\ell|\le N}}\frac1{F^N(|\ell|^2+\alpha,z)}\leq \frac1{N^2}\sum_{\substack{\ell+m=k_1\\ |\ell|\le N}}\frac1{F^N((|\ell|^2+\alpha)\vee 1,z)}\leq\frac{G^N(\alpha\vee 1,z)}{\lambda^2 z}\,.
  \end{equs}
Therefore, the bound on \eqref{eq:t3} is concluded. 
%  
%  \begin{equs}
%  \frac1{N^2} \sum_{|\ell|\le N}\frac1{F^N(A\vee 1,z)}\le \frac1{N^2}\sum_{|\ell|\le N}\frac{G^N((\mu+\frac12|k_{1:n}|^2)\vee 1,z)}{\lambda^2 z}\lesssim \frac{G^N((\mu+\frac12|k_{1:n}|^2)\vee 1,z)}{\lambda^2 z}.
%  \end{equs}
%  Altogether
%  \begin{equs}\label{e:altot3}
%    \eqref{eq:t3}\lesssim    \frac{G^N(
%      (\mu+\frac12|k_{1:n}|^2)\vee 1,z)}{\lambda^2 z}.
%  \end{equs}
%
%  
\medskip

\noindent{\bf Step 3.} We need to upper bound
 \begin{equ}
    \label{eq:t2}
    \sum_{\ell+m=k_1}   
    \frac{|(\nonlin_{\ell,m})^2-(\nonlin_{\ell,-\ell})^2|}{ (|\ell|^2+\alpha)(1+|\ell/N|^2+\alpha_N)F^N(|\ell|^2+\alpha,z)}\leq\sum_{\ell+m=k_1}   
    \frac{|(\nonlin_{\ell,m})^2-(\nonlin_{\ell,-\ell})^2|}{ (|\ell|^2+\alpha)F^N(|\ell|^2+\alpha,z)}\,.
    \end{equ}
    We split $\mathbb Z^2$ into 
    $\Omega_1=\{\ell:|k_1-\ell|<\frac12|k_1|\}$ and its complement $\Omega_2$. 
    In $\Omega_1$, it is immediate to see that one has $\frac12|k_1|\le |\ell|\le \frac32|k_1|$; bounding the
    term inside the absolute value in \eqref{eq:t2} by a constant,
    we are left with
  \begin{equs}
\sum_{\frac12|k_1|\le |\ell|\le \frac32|k_1|}\frac{1}{ (|\ell|^2+\alpha)F^N(|\ell|^2+\alpha,z)}\,.
  \end{equs}
  In the relevant region of summation one has $|\ell|^2+\alpha\le 6(\alpha\vee1)$, 
  so that, since $F^N$ is decreasing in the first argument, we can upper bound the sum as
  \begin{equs}
\frac1{F^N(6(\alpha\vee 1),z)}\sum_{\frac12|k_1|\le |\ell|\le \frac32|k_1|}\frac{1}{|\ell|^2}\lesssim \frac{G^N((\mu+\frac12|k_{1:n}|^2)\vee 1,z)}{\lambda^2 z}\,.
  \end{equs}
  In $\Omega_2$, instead, we use  the definition \eqref{e:nonlinCoefficient} of $\nonlin$ and we note that
  \[
  c(\ell,k_1-\ell)^2-c(\ell,\ell)^2=c(\ell,k_1)c(\ell,k_1-2\ell)\,.
  \]
Hence, 
  \begin{equs}
\left|
      (\nonlin_{\ell, k_1-\ell})^2-(\nonlin_{\ell,-\ell})^2
      \right|\le \frac{c(\ell,\ell)^2}{|\ell|^4}\frac{|k_1\cdot(k_1-2\ell)|}{|k_1-\ell|^2}+\frac{|c(\ell,k_1)c(\ell,k_1-\ell)|}{|\ell|^2|k_1-\ell|^2}\\
      \lesssim \frac{|k_1|}{|k_1-\ell|}\frac{|k_1-2\ell|}{|k_1-\ell|}+\frac{|k_1|}{|k_1-\ell|}\lesssim\frac{|k_1|}{|k_1-\ell|}\left(1+\frac{|k_1|}{|k_1-\ell|}\right)
      \lesssim\frac{|k_1|}{|k_1-\ell|}
  \end{equs}
  where the last inequality follows from the definition of $\Omega_2$ (note that the denominator cannot vanish therein). 
  Also,  one has $|\ell|^2+\alpha\ge \alpha$ so that
  \[
\frac1{F^N(|\ell|^2+\alpha,z)}\le \frac1{F^N((|\ell|^2+\alpha)\vee 1,z)}\le \frac{G^N((\mu+\frac12|k_{1:n}|^2)\vee1,z)}{z\lambda^2}\,.
    \]
  Hence, the sum at the right hand side of \eqref{eq:t2} restricted to $\Omega_2$ is bounded from above by
  \begin{equs}
    \frac{G^N((\mu+\frac12|k_{1:n}|^2)\vee 1,z)|k_1|}{\lambda^2 z}\sum_{ \ell\ne k_1}\frac{1}{|k_1-\ell|}\frac{1}{|\ell|^2+|k_1-\ell|^2}\lesssim\frac{G^N((\mu+\frac12|k_{1:n}|^2)\vee 1,z)}{\lambda^2 z}
  \end{equs}
  where the restriction ${k_1\ne \ell}$ comes from the definition of $\Omega_2$ and the last bound is easily obtained by splitting into the region where $|\ell|$ is larger or smaller than $\frac12|k_1|$.
%  Therefore,
%  \begin{equ}
%    \label{e:altot2}
%    \eqref{eq:t2}\lesssim \frac{G^N((\mu+\frac12|k_{1:n}|^2)\vee1,z)}{\lambda^2 z}.
%  \end{equ}
\medskip

Putting together the bounds obtained in Steps 1,2 and 3, 
we have bounded the left hand side of \eqref{e:lemmaC4nuovo} (modulo absolute constants) by 
\begin{equ}
\frac{G^N((\mu+\frac12|k_{1:n}|^2)\vee 1,z)}{\lambda \sqrt{z}}+2\frac{G^N((\mu+\frac12|k_{1:n}|^2)\vee 1,z)}{\lambda^2 z}\,.
\end{equ}
Hence, since $z\ge1$, there clearly exists a constant $K$ (depending only on $\lambda$) for which the statement 
holds and the proof is completed. 
  \end{proof}
  
We conclude this appendix by showing that the bounds on the Riemann-sums  performed in the proofs of 
Lemmas~\ref{l:UBtoLB},~\ref{l:LBtoUB} and~\ref{l:OffDiagBound} are uniform in the scale parameter $N\in\N$.

\begin{lemma}
  \label{lemma:RiemannNew}
  Let $F$ satisfy Assumption \ref{ass:FG} and define $\alpha$ and $\alpha_N$ according to~\eqref{e:Conv}. 
 Then, there exists a constant $K$ (depending only on $\lambda$) such that
 \begin{equs}
   \label{e:Riemann}
    \Big|&
      \sum_\ell\frac{(\nonlin_{\ell,-\ell})^2}{(|\ell|^2+\alpha)(1+|\ell/N|^2+\alpha_N)F^N(|\ell|^2+\alpha,z)}\\
      &
-\int_{\R^2}\frac{(\nonlin_{xN,-xN})^2\dd x}{(|x|^2+\alpha_N)(1+|x|^2+\alpha_N)F( |x|^2+\alpha_N,z)}
    \Big|\le K\frac{G^N((\mu+\frac12|k_{1:n}|^2)\vee 1,z)}{\lambda^2 z}.
  \end{equs}
\end{lemma}

\begin{proof}
  A first observation is that, letting $\ell=x N$, the summand is
  exactly $1/N^2$ times the integrand. The claim is then a Riemann sum approximation statement but
  some care has to be taken, on the one hand because the integrand is
  singular at the origin, and on the other because we want the constant
  $K$ not to depend on $F$.

  Since, by definition~\eqref{e:nonlinCoefficient} $\nonlin$ contains an indicator function~\eqref{eq:JN} 
  which forces  $\ell\neq0$ in the sum and
  $1/N\le |x|\le 1$ in the integral, we can assume these two conditions to be in place, and 
  therefore the difference in \eqref{e:Riemann} equals
  \begin{equs}[e:IntApprox]
 \sum_{1\le |\ell|\le N} &\int_{Q^N_\ell} |I_1(\ell/N)I_2(\ell/N)- I_1(x)I_2(x)|\dd x\leq \frac{1}{N} \sum_{1\le |\ell|\le N}\int_{Q^N_\ell}\dd x \sup_{x\in Q_\ell^N}|\nabla (I_1(x)I_2(x))|\\
 &=\frac{1}{N^3}\sum_{1\le |\ell|\le N}\sup_{x\in Q_\ell^N}|\nabla (I_1(x)I_2(x))|\\
 &\leq \frac{1}{N^3}\sum_{1\le |\ell|\le N}\sup_{x\in Q_\ell^N}|\nabla I_1(x)||I_2(x)|+\frac{1}{N^3}\sum_{1\le |\ell|\le N}\sup_{x\in Q_\ell^N}|I_1(x)||\nabla I_2(x)|
  \end{equs}
where $Q_\ell^N$ is the square of side-length $1/N$ centred at $\ell$, 
while $I_1$ and $I_2$ are the functions defined as 
\begin{equ}
I_1(x)\eqdef\frac{1}{4\pi^2} \frac{c(x,-x)^2}{|x|^2(|x|^2+\alpha_N)(1+|x|^2+\alpha_N)}\,,\qquad I_2(x)\eqdef \frac1{F( |x|^2+\alpha_N,z)}\,.
\end{equ}
We will separately bound the suprema appearing in the sums above. 
For the first, it is not hard to see that since
\begin{equ}
  \label{exjac}
I_2(x)= \frac1{F( |x|^2+\alpha_N,z)}\le  \frac1{F( (|x|^2+\alpha_N)\vee 1/N^2,z)}\le \frac{G^N(\alpha\vee 1,z)}{\lambda^2 z}\,, 
\end{equ}
which is independent of $\ell$, we have 
\begin{equ}[e:sup1]
\sup_{x\in Q_\ell^N}|\nabla I_1(x)||I_2(x)| \lesssim \frac{G^N(\alpha\vee 1,z)}{\lambda^2 z}\sup_{x\in Q_\ell^N}\frac1{|x|(|x|^2+\alpha_N)}\lesssim \frac{G^N(\alpha\vee 1,z)}{\lambda^2 z}\frac{N^3}{|\ell|(|\ell|^2+\alpha)}\,.
\end{equ}
For the other instead, note that
\begin{equs}
  |\nabla I_2(x)|
   &\lesssim\frac{|x| |F'(|x|^2+\alpha_N,z)|}{F(|x|^2+\alpha_N,z)^2}\le \frac{(|x|^2+\alpha_N)|F'(|x|^2+\alpha_N,z)|}{(|x|^2+\alpha_N)^\half F(|x|^2+\alpha_N,z)^2} \\
   &\lesssim \frac1{(|x|^2+\alpha_N)^\half F(|x|^2+\alpha_N,z)}\leq \frac1{(|x|^2+\alpha_N)^\half}\frac{G^N(\alpha\vee 1,z)}{\lambda^2 z}
 \end{equs}
where we exploited assumption \eqref{e:addass} when passing from the first to the second line, 
and~\eqref{exjac} in the last step. Hence, 
\begin{equ}[e:sup2]
\sup_{x\in Q_\ell^N}|I_1(x)||\nabla I_2(x)|\lesssim \frac{G^N(\alpha\vee 1,z)}{\lambda^2 z}\frac{N^3}{|\ell|(|\ell|^2+\alpha)}\,.
\end{equ}
We now plug the bounds~\eqref{e:sup1} and~\eqref{e:sup2} into~\eqref{e:IntApprox}, 
which, consequently, is upper bounded by
\begin{equ}
\frac{G^N(\alpha\vee 1,z)}{\lambda^2 z}\sum_{1\le |\ell|\le N}
  \frac1{|\ell|(|\ell|^2+\alpha)}\lesssim \frac{G^N(\alpha\vee 1,z)}{\lambda^2 z}\,,
\end{equ}
so that the statement follows at once. 
\end{proof}

\end{appendix}

\section*{Acknowledgements}
G. C. gratefully acknowledges financial support via the EPSRC grant EP/S012524/1. 
D. E. gratefully acknowledges financial support 
from the National Council for Scientific and Technological Development - CNPq via a 
Universal grant 409259/2018-7, and a Bolsa de Produtividade 303520/2019-1. D.E moreover acknowledges financial support from the Coordenac\~ao de Aperfeicoamento de Pessoal de N\'ivel Superior - Brasil (CAPES) via a Capes print UFBA 02.2019 scholarship. Moreover, D.E acknowledges support by the Serrapilheira Institute which supported this work (grant number Serra - R-2011-37582). F. T. gratefully acknowledges financial support of Agence Nationale de la Recherche via the
ANR-15-CE40-0020-03 Grant LSD.
We are grateful to  the Hausdorff Institute in Bonn, where this work was initiated, for the kind hospitality.

\bibliography{bibtex}
\bibliographystyle{Martin}

\end{document}